\documentclass{article}
\usepackage{amsmath,amsthm,amssymb}
\usepackage[utf8]{inputenc}
\usepackage[english]{babel} 
\usepackage{esint}
\usepackage{xcolor}
\usepackage{hyperref}
\usepackage{bbold}
\usepackage{enumerate}
\usepackage{pgfplots}
\usepackage{siunitx}
\usepackage{stmaryrd}
 \pgfplotsset{width=10cm,compat=1.9}
 \usepackage[a4paper, total={17.5cm, 24cm}]{geometry}
\usepackage{tikz}
\usetikzlibrary{arrows}
\usepackage{url}
\usepackage{todonotes}
\makeatletter
\numberwithin{equation}{section}
 \usepackage[backend=biber,   isbn=false, doi=false, url=false, maxnames=5] {biblatex}
\usepackage[toc,page]{appendix}
\usepackage{authblk}

\theoremstyle{plain}
        \newtheorem{theorem}{Theorem}[section]
        \newtheorem{proposition}[theorem]{Proposition}
        \newtheorem{lemma}[theorem]{Lemma}
        \newtheorem{corollary}[theorem]{Corollary}

        \newtheorem{remark}[theorem]{Remark}  
         
        \newtheorem*{claim*}{Claim}  
         
\newtheorem*{theorem*}{Theorem}
\newtheorem*{definition*}{Definition}
\newtheorem*{proposition*}{Proposition}
\usepackage{graphicx}
\let\oldmarginpar\marginpar
\renewcommand\marginpar[1]{\-\oldmarginpar[\raggedleft\footnotesize #1]
{\raggedright\footnotesize #1}}

\def\dt'{\,{\rm d}t'\,}
\renewcommand\div{\mathrm{div\,}}

\renewcommand \Re {\text{Re}}

\newcommand{\R}{\mathbb{R}}

\newcommand{\N}{\mathbb{N}}

\renewcommand{\d}{\partial} 

\newcommand{\dtau}{\partial_{\overline{\tau}}}
\newcommand{\taub}{\overline{\tau}}
\newcommand{\brl}{\Bigl(}
\newcommand{\brr}{\Bigr)}
\newcommand{\ddiv}{\mathrm{div}_2\,}
\newcommand{\rttoo}{P_1}
\newcommand{\rto}{P_2}
\newcommand{\Dj}{D^{(j)}}

\renewcommand\dt{\mathrm{dt}}

\title {Global-in-time well-posedness for the  two-dimensional   incompressible Navier-Stokes equations with freely transported  viscosity coefficient}
\author {Xian Liao and Rebekka Zimmermann}

\begin{document}

\date{}

\maketitle
\begin{abstract}
We establish the global-in-time well-posedness of the two-dimensional incompressible Navier-Stokes equations with   freely transported viscosity coefficient, under a scaling-invariant smallness condition on the initial data. The viscosity coefficient is allowed to exhibit large jumps across $W^{2,2+\epsilon}$-interfaces.

The viscous stress tensor $\mu Su$ is carefully analyzed. Specifically, $(R^\perp\otimes R):(\mu Su)$, where $R$ denotes the Riesz operator, defines a ``good  unknown'' that satisfies time-weighted $H^1$-energy estimates. Combined with tangential regularity, this leads to the $W^{1,2+\epsilon}$-regularity of another ``good  unknown'', $(\bar{\tau}\otimes n):(\mu Su)$, where $\bar{\tau}$ and $n$ denote the unit tangential and normal vectors of the interfaces, respectively. These results collectively provide a Lipschitz estimate for the velocity field, even in the presence of significant discontinuities in $\mu$.

As applications, we investigate the well-posedness of the    Boussinesq equations without heat conduction and the density-dependent incompressible Navier-Stokes equations in two spatial dimensions.
\end{abstract}

\tableofcontents

\section{Introduction}\label{sect:intro}

This paper addresses the global-in-time well-posedness of the Cauchy problem for systems of equations that describe the evolution of an incompressible inhomogeneous viscous fluid in two spatial  dimensions. We primarily focus on constant-density fluids where the viscosity coefficient exhibits large variation, such as in the mixing of two rivers with different temperatures.
The motion can be described by the following incompressible   Navier-Stokes equations with freely transported, variable viscosity coefficient
\begin{equation}\label{muNS}
    \begin{cases}
        \d_t\mu+u\cdot\nabla \mu=0, \quad (t,x)\in (0,\infty)\times \R^2,\\
         \d_t  u +   u\cdot \nabla u  - \div (\mu Su)+\nabla \pi = 0, \\
        \div u=0 .
    \end{cases}
\end{equation}
Here, $t\in [0,\infty)$ and $x=\begin{pmatrix}
    x_1\\x_2
\end{pmatrix} \in\R^2$ denote the time and space variables, respectively. The unknowns of the equations are  the velocity vector field $u=u(t,x)=\begin{pmatrix}
    u_1\\ u_2
\end{pmatrix}\in\R^2$, the viscosity coefficient $\mu=\mu(t,x)\in (0,\infty)$ and the gradient of the pressure $\nabla\pi=\nabla \pi(t,x)=\begin{pmatrix}
    \d_1\pi\\ \d_2\pi
\end{pmatrix}\in\R^2$, which is the Lagrangian multiplier associated to the divergence-free condition on the velocity $\eqref{muNS}_3$.

We aim to investigate the nonlinear interplay between the unknown viscosity coefficient $\mu$, which satisfies the free transport equation $\eqref{muNS}_1$,  and the velocity field $u$, which satisfies the incompressible Navier-Stokes equations $\eqref{muNS}_2$ with this varying viscosity coefficient $\mu$. 

\subsection{Divergence of the viscous stress tensor}
We start with a detailed analysis of the divergence of the viscous stress tensor   in $\eqref{muNS}_2$:
\begin{equation}\label{ViscosityTerm}
     \div(\mu Su),
\end{equation}
where the matrix $Su\in \R^{2\times 2}$ denotes twice the symmetric part of the velocity gradient: 
$$(Su)_{ij}=2\cdot\frac12(\d_{i}u_j+\d_{j}u_i),\quad i,j=1,2.$$

If $\mu=\nu>0$ is a positive constant, then  the divergence-free condition $\div u=0$ simplifies the above viscosity term
\eqref{ViscosityTerm} into
\begin{equation}\label{nuDelta}
\div(\mu Su)= \nu\Delta u,
\end{equation}
a diffusion term that plays an important role in the classical Navier-Stokes solution theory     in  J. Leray's pioneer work \cite{leray1933essai}. 
It is well-known, following    the celebrated work of O. A. Ladyzhenskaya \cite{ladyzhenskaya1969mathematical},  that in space dimension \textit{two}, J. Leray's weak solutions in the energy space $L^2(\R^2;\R^2)$ are unique and  the Cauchy problem for the classical Navier-Stokes equations (i.e. the system \eqref{muNS} with $\mu=\nu>0$) is well-posed globally in time.
In \textit{three} spatial dimensions, the uniqueness and the regularity of  Leray's weak solutions  are extensively studied, and at the same time, it has been shown that strong solutions with \textit{small} initial data exist uniquely for all time; see the
 recent monographs \cite{Lemarie2016,Tsai2018} and references therein. 
The global-in-time well-posedness problem for arbitrarily \textit{large} initial data  in three dimensions  remains  open and is famously known as the Millennium Problem for the Navier-Stokes equations \cite{CMI}.

The study of   fluid motion   with \textit{variable} viscosity coefficient $$\mu=\mu(t,x)$$
is of significant physical  relevance, cf.  \cite[Section 6]{Lide}, 
and has attracted considerable interest in the  mathematical community,  cf. the books \cite{antontsev1989boundary,feireisl, lions1996mathematical,lions1996mathematical2,Majda}. 
In the following we  present three prototypical   \textit{incompressible} inhomogeneous models     in the literature, highlighting   their relations with our model \eqref{muNS}.
We also review briefly three typical approaches for handling   the viscosity term \eqref{ViscosityTerm}, before introducing our own approach in Subsection \ref{subss:decomposition}. 
Finally, as  applications of our main result, Theorem \ref{exthm}, for the system \eqref{muNS} in Section \ref{subs:result}, we give mathematical results for these models in Corollary \ref{propthm}.

\bigbreak  

\noindent\textbf{Smooth viscosity  case.}
Variable viscosity coefficients have been successfully incorporated  into the study of the two-dimensional Boussinesq equations with heat conduction
\begin{equation}\label{Boussinesq}
    \begin{cases}
        \d_t\vartheta+u\cdot\nabla \vartheta-\div(\kappa\nabla\vartheta)=0, \quad (t,x)\in (0,\infty)\times \R^2,\\
         \d_t  u +   u\cdot \nabla u - \div (\mu Su)+\nabla \pi = \vartheta e_2, \\
        \div u=0.
    \end{cases}
\end{equation} 
Here, the unknowns are the temperature $\vartheta\in\R$, the velocity field $u\in \R^2$ and the pressure $\pi\in\R$.
The heat conduction coefficient and the viscosity coefficient
$$\kappa=\kappa_\vartheta(\vartheta),\quad \mu=\mu_\vartheta(\vartheta)$$
are both smooth functions
\footnote{It is common to adapt  cf. \cite[Part I]{PTBC} 
\begin{align*}
    &\hbox{constant heat conductivity law }\kappa_\vartheta=C_1\hbox{  and exponential viscosity law } \mu_\vartheta(\vartheta)=C_2\exp(C_3/(C_4+\vartheta)) \hbox{ for liquids,}
    \\
    &\hbox{while }\kappa_\vartheta(\vartheta)=C_5\mu(\vartheta)
    \hbox{ and Sutherland's Law }\mu_\vartheta(\vartheta)=\underline{\mu}(\frac{\vartheta}{\underline{\vartheta}})^{\frac32}\frac{\underline{\vartheta}+C_6}{\vartheta+C_7}   \hbox{ for gases},
\end{align*}
where $C_j$, $j=1,\cdots,7$ are  constants and $\vartheta_0,\underline{\mu}=\mu_\vartheta(\underline{\vartheta})$ are reference temperature and viscosity coefficient.
In particular,  Andrade's Law: $\mu_\vartheta(\vartheta)=C_2\exp(C_3/\vartheta)$
with $C_2=e^{-12.9896}, C_3=1780.622, C_4=0$ gives good accurate values in the range of $[10-100^\circ]$ for waters, and Sutherland's Law $\mu_\vartheta(\vartheta)=\underline{\mu}
(\frac{\vartheta}{\underline{\vartheta}})^{\frac32}\frac{\underline{\vartheta}+C_6}{\vartheta+C_7}$ with $\underline{\vartheta}=273 \,K$, $\underline{\mu}=1.716\times10^{-5}$, $C_6=C_7=110.5\,K$ is good approximation for air close to the reference temperature $273\,K$.}
 of the unknown temperature $\vartheta$.
 The buoyancy force  term $\vartheta e_2$ in $\eqref{Boussinesq}_2$ accounts for  the gravitational effects. 
 The Boussinesq equations \eqref{Boussinesq} has been known as one of the most important models in geophysical fluid dynamics \cite{Gill}.

In the case of strong heat conduction $\kappa(t,x)\geq \kappa_\ast>0$, the diffusion term $\div(\kappa\nabla\vartheta)$ regularizes the temperature $\vartheta$ over time, leading to a smooth viscosity coefficient $\mu=\mu_\vartheta(\vartheta)$. 
Consequently, the viscosity term \eqref{ViscosityTerm} can be rewritten as
\begin{equation}\label{ViscosityTerm1}
 \div(\mu Su)= \mu \Delta u+\nabla \mu \cdot Su,
\end{equation}
where  $\nabla \mu\cdot Su$ is considered as a lower-order term with respect to $u$.
This formulation results in  global-in-time well-posedness results, as discussed in  \cite{ heliao-correction, heliao, lorcaBoldrini, WangZhang} and references therein.
The classical  constant coefficient scenario has been extensively studied in the literature, 
see 
the review notes \cite{WuJiahong} for more general results.

In the case of very weak heat conduction with $\kappa=0$,   the temperature $\vartheta$ satisfies the free transport equation, transforming \eqref{Boussinesq} into
\begin{equation}\label{eq:theta}
    \begin{cases}
        \d_t\vartheta+u\cdot\nabla \vartheta=0, \quad (t,x)\in (0,\infty)\times \R^2,\\
         \d_t  u +   u\cdot \nabla u - \div (\mu Su)+\nabla \pi = \vartheta e_2, \\
        \div u=0.
    \end{cases}
\end{equation}  
This  motivates our consideration of \eqref{muNS}, which is derived from \eqref{eq:theta} by neglecting  the buoyancy effect $\vartheta e_2$ on the right hand side of $\eqref{eq:theta}_2$.
Specifically, multiplying $\eqref{eq:theta}_1$  by $\mu_\vartheta'(\vartheta)$ (formally) yields  the free transport equation of $\mu$ in $\eqref{muNS}_1$.

Recently there has been notable progress in the mathematical analysis of \eqref{muNS} and \eqref{eq:theta}, cf. \cite{Abidi09, AbidiZhang19, NiuLu24}, under either the smoothness assumption $\nabla \mu_0\in L^p$ or small variation assumption (see \eqref{small-visc} below).
It remains an open problem whether global-in-time well-posedness results still hold in the presence of large rough variation in the initial data. 
Our primary global-in-time well-posedness result for the system \eqref{muNS}, under a scaling-invariant smallness assumption, is presented in Theorem \ref{exthm} below. Notably, this result permits   \textit{large jumps} in the viscosity coefficient. As a corollary, we establish a lower bound on the existence time of solutions to \eqref{eq:theta}, expressed in terms of the initial data, in Corollary \ref{propthm} that follows.

\bigbreak 
\noindent\textbf{Small variation case.}
Variable viscosity coefficients have  also been investigated recently  in the context of density-dependent  incompressible fluids with freely transported  density function, described by the system
\begin{equation}\label{NS}
    \begin{cases}
        \d_t\rho+u\cdot\nabla \rho=0, \quad (t,x)\in (0,\infty)\times \R^2,\\
        \rho(\d_t  u +   u\cdot \nabla u) - \div (\mu Su)+\nabla \pi = 0, \\
        \div u=0 .
    \end{cases}
\end{equation} 
Here   $\rho=\rho(t,x)\geq 0$ is the unknown density function, and the viscosity coefficient $\mu$ is a given smooth function of $\rho$  as
$$\mu=  \mu_\rho(\rho): [0,\infty)\to (0,\infty).$$
The three equations in \eqref{NS} represent the mass conservation law, the momentum conservation law, and the incompressibility condition, respectively.
Formally, the system  \eqref{muNS} can be seen as the density-dependent  incompressible Navier-Stokes equations \eqref{NS} with   the density dependence in the transport term in the momentum equation $\eqref{NS}_2$ being neglected. Specifically, similarly as above, multiplying $\eqref{NS}_1$ by $\mu_\rho'(\rho)$ gives $\eqref{muNS}_1$,  while $\eqref{NS}_2$ simplifies to  $\eqref{muNS}_2$ by replacing $\rho(\d_t u+u\cdot\nabla u)$ by $(\d_t u+u\cdot\nabla u)$   (similar as in the Boussinesq-approximation).

The system (\ref{NS}) has been widely explored by numerous  mathematicians.   P.-L. Lions  establishes the existence of global-in-time weak solutions in \cite{lions1996mathematical}, which improves an earlier work  
\cite{simon1990nonhomogeneous}
for the constant viscosity case.
In the case of constant viscosity $\mu=\nu>0$, 
the existence and uniqueness of strong solutions of \eqref{NS} in the case of smooth initial data $(\rho_0, u_0)$ are demonstrated by O. A. Ladyzhenskaya and V. A. Solonnikov \cite{ladyzhenskaya1978unique}.
Motivated by the natural scaling of (\ref{NS}), a number of works have been dedicated to the study of the system in critical functional spaces which are invariant under the same scaling, see for example \cite{abidi2007equation, abidi2021global,      danchin2003density,   huang2013global2} and references therein.
Recently, the global-in-time well-posedness results in the more general case with   discontinuous densities in the presence of vaccuum  are now known to hold true, thanks to the remarkable contributions by  
R. Danchin and P. B. Mucha \cite{danchin2012lagrangian, danchin2013incompressible, danchin2019incompressible}.


For general viscosity $\mu=\mu_\rho(\rho)$,  local-in-time well-posedness  for smooth initial data for \eqref{NS} was established in Y. Cho and H. Kim \cite{cho2004unique}, see also the book \cite{antontsev1989boundary}.
Under \textit{small variation} assumptions, either with small density variation   \cite{gui2009global, huang2013global, liu2020global}  or   small viscosity variation  \cite{abidi2015global2D,gancedo2023global, huang2012decay,huang2014global,paicu2020striated},  global-in-time well-posedness results have been achieved in two spatial  dimensions. 
An earlier work by Desjardins \cite{desjardins1997regularity} addresses the regularity of  P.-L. Lions' weak solutions.
For the three spatial dimensional case, see  \cite{abidi2015global, he2021global, huang2015global, zhang2015global} and references therein. 

In the case where   $\mu$ is close to a positive constant $\nu>0$:
\begin{align}\label{small-visc}
    \Vert \mu-\nu \Vert_{L^\infty(\R^2)} \ll 1,
\end{align}
a key ingredient in the analysis is  the following decomposition of  the viscosity term \eqref{ViscosityTerm}: 
\begin{align}\label{divmuSu:small-decomp}
     \div(\mu Su) =\nu \Delta u +\div((\mu-\nu )Su),
\end{align}
where $\div((\mu-\nu)Su)$ is considered as a perturbation term.
However, this decomposition does not apply when $\mu$ varies significantly.
It remains open whether the global-in-time wellposedness of \eqref{NS} holds in two space dimensions with large   initial data. 
We give in Corollary \ref{propthm} below the global-in-time wellposedness of \eqref{NS}, assuming some smallness condition  while allowing  for large variations in the density.

\bigbreak 

\noindent\textbf{Piecewise-constant case.}
When describing the time evolution of two immiscible  fluids,  which are separated by a free  interface, one considers the following two-phase Navier-Stokes equations
\begin{align}\label{2pNS}
\left\{
\begin{array}{ll}
   \rho( \d_t u + u\cdot \nabla u) -\div(\mu Su) +   \nabla \pi =0,
   \quad \div u=0 & \text{in}\; \Omega_t^-\cup \Omega_t^+, \\
    \llbracket u \rrbracket =0,\quad \llbracket T(u,\pi)n \rrbracket =\sigma Hn,
    \quad V=u\cdot n & \text{on}\; \Gamma_t.
\end{array}
\right.
\end{align}  
Here, two fluids occupy the domains $\Omega_t^+, \Omega_t^-$  respectively,  with $\Gamma_t$ as the separating interface. 
The vector $n=n(t,x)$ denotes the outward unit normal to $\Omega_t^+$, and $\llbracket \cdot \rrbracket$ represents the jump of a function across the interface $\Gamma_t$ in the direction of $n$.
The functions $H=H(t,x)$ and $V=V(t,x)$ denote   the  curvature and the normal velocity of $\Gamma_t$  with respect to $n$, respectively, and $\sigma\geq 0$ is the surface tension coefficient. The total stress tensor $T(u,\pi)$ is defined by
\begin{align*}
    T(u,\pi)= \mu  Su -\pi \mathrm{Id},\hbox{ with }\mathrm{Id}\in \R^{2\times 2}\hbox{ denoting the unit matrix}.
\end{align*}

In the case where two different fluids having positive constant densities $\rho^+, \rho^-$ and positive constant viscosity coefficients $\mu^+=\rho^+\nu^+, \mu^-=\rho^-\nu^-$,  the momentum equation in $\eqref{2pNS}_1$ reads as
\begin{align}\label{2pNS-pw}
     \d_t u + u\cdot \nabla u - \nu^\pm \Delta u + \frac{1}{\rho^\pm} \nabla \pi =0 & \hbox{ in }\; \Omega_t^-\cup \Omega_t^+.
\end{align}
In this scenario, the viscosity term \eqref{ViscosityTerm} simplifies to 
\begin{equation}\label{ViscosityTerm3}
    \div(\mu Su)= \mu^\pm\Delta u \quad \hbox{ in }\; \Omega_t^-\cup \Omega_t^+,
\end{equation}
which reduces the problem \eqref{2pNS} to solving the Navier-Stokes equations with a \textit{constant} viscosity coefficient within each domain. The main challenge then lies in  determining the free interface $\Gamma_t$. 

Notice that in the absence of surface tension ($\sigma=0$), if $(\rho, u, \nabla\pi)$   solves the density-dependent  incompressible Navier-Stokes equations \eqref{NS} with the initial density $\rho_0=\rho^+1_{\Omega_0^+}+\rho^-1_{\Omega_0^-}$, then 
it also satisfies \eqref{2pNS}-\eqref{2pNS-pw}, provided that both the vectors $u$ and $T(u,\pi)n$ are continuous across the  freely transported interface $\Gamma_t$ (as long as $\Gamma_t$ remains well-defined).
Similarly, in the case of   constant density function $\rho^\pm=1$, if $(\mu,u,\nabla\pi)$ solves \eqref{muNS} with the initial viscosity $\mu_0=\mu^+1_{\Omega_0^+}+\mu_-1_{\Omega_0^-}$ and both   $u$ and $T(u,\pi)n$ are continuous across the well-defined free-transported interface $\Gamma_t$,
then it satisfies   \eqref{2pNS}, which in this context becomes
\begin{align}\label{mu2pNS}
\left\{
\begin{array}{ll}
    \d_t u + u\cdot \nabla u - \div(\mu Su) +  \nabla \pi =0,\quad \div u=0 & \text{in}\; \Omega_t^-\cup \Omega_t^+, \\ 
    \llbracket u \rrbracket= \llbracket T(u,\pi)n \rrbracket =0,\quad V=u\cdot n & \text{on}\; \Gamma_t.
\end{array}
\right.
\end{align}   
The  model \eqref{ViscosityTerm3}-\eqref{mu2pNS} is known as a sharp interface model.
For discussions on the sharp interface limit of Navier-Stokes/Allen-Cahn or Navier-Stokes/Cahn-Hilliard equations, see \cite{AbelsFei,LiuShen} and the references therein.

The two-phase Navier-Stokes equations \eqref{2pNS} with piecewise-constant densities and viscosity coefficients \eqref{2pNS-pw} have been thoroughly studied  since the 1980s in various configurations of $\Omega_t^-$ and $\Omega_t^+$; see the books \cite{denisova2021motion, pruss2016moving}
for a comprehensive overview. 
In the presence of surface tension ($\sigma>0$), local-in-time existence and uniqueness results are provided in e.g. \cite{denisova1996classical, pruss2011analytic} and global-in-time well-posedness is proved in \cite{denisova2012global, solonnikov2014p}. 
See also \cite{Abels} for the global-in-time existence  of varifold solutions with rather general initial data.
When the surface tension is absent ($\sigma =0$), global-in-time well-posedness has   been obtained in e.g. \cite{ denisova2014global, denisova2007global, saito2020some}. 
However, it remains unclear whether $\rho^\pm, \mu^\pm$ can be taken as \textit{largely variable}   smooth functions within their respective  domains $\Omega_t^\pm$. 
In Corollary \ref{propthm} below we address this issue for the systems \eqref{2pNS} (with $\sigma=0$) and \eqref{mu2pNS}.

\bigbreak 

The literature includes extensive discussions on the \textit{regularity} of solutions for other evolutionary models with variable viscosity coefficients. 
This includes for instance compressible models \cite{hoff2002dynamics, matsumura1979initial, 
zodji2023discontinuous}, zero Mach-number systems and combustion models \cite{danchinliao},  MHD equations with density-dependent viscosity \cite{huang2013mhd}.
However, to our knowledge, at least one of the above decompositions \eqref{ViscosityTerm1} (regular case), \eqref{divmuSu:small-decomp} (perturbed case), and \eqref{ViscosityTerm3} (piecewise-constant case) for the viscosity term \eqref{ViscosityTerm} has been applied  in the regularity theory.  In this paper, we aim to  address more general variable viscosity coefficients, relying on the following decomposition.

\subsubsection{Decomposition for the divergence of the viscous stress tensor}\label{subss:decomposition}
In the present paper, building on insights from the previous work \cite{he2020solvability} by Z. He and the first author, for the stationary Navier-Stokes equations with variable viscosity coefficient, we  decompose the divergence of the viscous stress tensor \eqref{ViscosityTerm} straightforwardly into a divergence-free component and a curl-free component.
This approach allows us to handle more general variable viscosity coefficients effectively.
\begin{lemma}[Decomposition for the divergence of the viscous stress tensor]\label{lemma:decomp}
Let $u=\nabla^\perp\phi$ with $\nabla^\perp:=\begin{pmatrix}
    -\partial_{2}\\ \partial_{1}
\end{pmatrix}$. Then the following (formal) decomposition holds
\begin{align}\label{divmuSu:decomp}
    \div(\mu Su)= \nabla^\perp a + \nabla b
\end{align}
where
\begin{align}
    \Delta a&=L_\mu\phi,\hbox{ with }L_\mu:=   (\d_{22}-\d_{11})\mu (\d_{22}-\d_{11})+(2\d_{12})\mu(2\d_{12}), \label{Lmu:def}\\
    \Delta   b&=A_\mu\phi,\hbox{ with }A_\mu:=   (\d_{22}-\d_{11})\mu(2\d_{12})-(2\d_{12}) \mu (\d_{22}-\d_{11}) . \label{Amu:def}
\end{align}

Let $\mu\in L^\infty(\R^2)$, $\nabla u\in L^2(\R^2; \R^{2\times2})$.
In the $L^2(\R^2)$-functional setting (where the Fourier transform applies), $a,b\in L^2(\R^2)$ can be determined by  $\mu Su$ as follows:
\begin{align}
    &a=-(-\Delta)^{-1}\nabla^\perp\cdot\div(\mu Su)=-(-\Delta)^{-1} (\nabla^\perp\otimes\nabla):(\mu Su)=(R^\perp\otimes R):(\mu Su),\label{a:R}\\
    &b=-(-\Delta)^{-1}\nabla\cdot\div(\mu Su)
    =-(-\Delta)^{-1}(\nabla\otimes\nabla):(\mu Su)=(R\otimes R):(\mu Su),\label{b:R}
\end{align}
where $R=\frac{\frac 1i\nabla}{\sqrt{-\Delta}}$ and $R^\perp=\frac{\frac 1i\nabla^\perp}{\sqrt{-\Delta}}$ are the Riesz operators.
If we introduce the scalar fluid vorticity $\omega=\nabla^\perp \cdot u=\Delta\phi$,  then 
 $a,b$ can be respresented in terms of $\mu, \omega$ and Riesz operators as follows:
\begin{align}
  a=R_\mu\omega, \hbox{ with }  R_\mu &:=(R_2R_2-R_1R_1)\mu(R_2R_2-R_1R_1) + (2R_1R_2)\mu (2R_1R_2), \label{Rmu:def}\\
  b=Q_\mu\omega, \hbox{ with }  Q_\mu &:=(R_2R_2-R_1R_1)\mu(2R_1R_2)- (2R_1R_2)\mu (R_2R_2-R_1R_1). \label{Qmu:def}
\end{align}

Here and in what follows the tensor product $u\otimes v$ of any two vectors $u=\begin{pmatrix}
    u_1\\ u_2
\end{pmatrix},v=\begin{pmatrix}
    v_1\\v_2
\end{pmatrix}$ refers to the matrix with the entries $(u\otimes v)_{ij}=u_iv_j$, $i,j=1,2$, and the product $A:B$ of any two matrices $A=\begin{pmatrix}
    A_{11}& A_{12}\\ A_{21}&A_{22}
\end{pmatrix},B=\begin{pmatrix}
    B_{11}& B_{12}\\ B_{21}&B_{22}
\end{pmatrix}$ refers to $A:B=\sum_{i,j=1}^2 A_{ij}B_{ij}$.
\end{lemma} 

The decomposition \eqref{divmuSu:decomp} can be checked (see also \cite{he2020solvability}) by straightforward computations: 
\begin{equation}\label{Lmu:compute}
    \nabla^\perp\cdot\div(\mu S\nabla^\perp\phi)=L_\mu\phi
    \quad \hbox{ and }\quad 
    \nabla\cdot\div(\mu S\nabla^\perp\phi)=A_\mu\phi.\end{equation}
\eqref{divmuSu:decomp} is equivalent to \eqref{a:R}-\eqref{b:R}.
The relations \eqref{Rmu:def}-\eqref{Qmu:def}  between $a, b$ and $\omega$ follow from \eqref{Lmu:def}-\eqref{Amu:def} directly.
This completes the proof of Lemma \ref{lemma:decomp}.

We aim to obtain global-in-time wellposedness of the system \eqref{muNS} with possibly large jumps across certain regular interfaces in the variable viscosity coefficient $\mu$.
In this case, none of the decompositions \eqref{ViscosityTerm1}, \eqref{divmuSu:small-decomp} and \eqref{ViscosityTerm3} for $\div(\mu Su)$ applies.  
With the above decomposition \eqref{divmuSu:decomp} we can apply $\nabla^\perp\cdot$ to the velocity equation $\eqref{muNS}_2$ to derive the equation for the vorticity
\begin{equation}\label{intro:omega}
    \d_t\omega+u\cdot\nabla\omega-\Delta a=0, 
\end{equation}
where $u=\nabla^\perp\Delta^{-1}\omega$ is given by the Biot-Savart law.
With $\mu$ freely transported by the velocity field $u$ as in $\eqref{muNS}_1$, 
$a=R_\mu\omega$ is given by applying nonlocal Riesz operators $R$ composed with the local multiplication operator by $\mu$ to $\omega$.
This ``nonlocal" vorticity equation \eqref{intro:omega} is essence of the system \eqref{muNS}.
We show later (time-weighted) $H^1$-energy estimates for the ``good  unknown'' $a$.
The challenge is then to derive the bounds for $\omega$ or $\nabla u$ from the estimates of $a$.

 
\subsubsection{Assumptions on the initial viscosity: $L^{2+\epsilon}$-estimate and tangential regularity}\label{subss:epsilon}
We now recall some facts from \cite{he2020solvability} for the stationary case of \eqref{muNS}, which motivate our assumptions on the initial viscosity $\mu_0$ in our main Theorem \ref{exthm} below:
\begin{enumerate}[(i)]\label{Assumption}
    \item Assume positive lower and upper bounds for $\mu_0$: $\mu_\ast\leq \mu_0\leq\mu^\ast$, with $\mu_\ast, \mu^\ast>0$ being two positive constants. These bounds are preserved by virtue of the free transport equation for $\mu$ a priori:
    \begin{equation}\label{mu:bound}
        0<\mu_\ast\leq \mu(t,x)\leq\mu^\ast.
    \end{equation}
    Then the operator given in \eqref{Lmu:def} above
    $$L_\mu=(\d_{22}-\d_{11})\mu (\d_{22}-\d_{11})+(2\d_{12})\mu(2\d_{12})$$
    is a fourth-order elliptic operator, since   we can reformulate $L_\mu$ as (see also \cite{he2020solvability})
\begin{align*}
    L_\mu &= \d_{11}(\mu \d_{11})+\d_{22}(\mu\d_{22})- \d_{11}\Bigl((\mu-\frac{\mu_\ast}{2})\d_{22}\Bigr) - \d_{22}\Bigl((\mu-\frac{\mu_\ast}{2})\d_{11}\Bigr) + \d_{12}\Bigl((4\mu- \mu_\ast)\d_{12}\Bigr) \\
    &=: \sum_{\vert\alpha\vert=\vert\beta\vert=2} D^\alpha (l_{\alpha\beta}^\mu D^\beta),
\end{align*}
where
\begin{align}\label{ellipticity}
    \frac{\mu_\ast}{2} \vert \xi \vert^2 \leq \sum_{\vert \alpha \vert = \vert \beta \vert=2} l_{\alpha \beta}^\mu \xi_\alpha \xi_\beta \leq 2\mu^\ast \vert \xi \vert^2, \quad\forall \xi=(\xi_\alpha)_{\vert\alpha\vert=2} \in \R^3.
\end{align}
    
    Note that if $\mu=\nu$ is a positive constant, then $L_\mu=\nu\Delta^2$ is a biharmonic operator, while   $a=\nu\omega$ and  $b=0$ by \eqref{Rmu:def} and \eqref{Qmu:def}, respectively.
    
    \item Assume tangential regularity for $\mu_0$:
    \begin{align}\label{TR:mu0}
    &\d_{\tau_0}\mu_0\in L^{p_0}(\R^2),\hbox{ for some }p_0>2,
    \\
&    \hbox{ where }\tau_0\in (L^\infty\cap \dot W^{1,p_0})(\R^2;\R^2)\hbox{ is some 
     nondegenerate regular vector field}.\nonumber
     \end{align}
    
    For any $p>2$, there exists a bounded measurable (highly oscillating)  function $\tilde \mu$ taking only two possible values, $\tilde \mu\in \{\frac{1}{K}, K\}$ with $K=\frac{2}{p-2}+1>1$, such that there exist solutions to the \textit{homogeneous} elliptic equation $L_{\tilde \mu}\phi=0$ with \begin{equation}\label{failure:Lp}
        \nabla u=\nabla\nabla^\perp\phi\not\in L^p_{\textrm{loc}}(\R^2).
    \end{equation}  
    In particular, this case corresponds to $a=0$ while $\nabla u\not\in L^p_{\textrm{loc}}(\R^2)$  by \eqref{Lmu:def}, that is, $a$ can not control $\nabla u$ in $L^p(\R^2)$.
    \eqref{failure:Lp} represents a generalization of the second-order elliptic operator $\div(\mu \nabla)$ studied in \cite{astala2008} to a fourth-order elliptic operator 
$L_\mu$.
    
    Therefore, since the regularity propagation   requires the Lipschitz-continuity of the velocity field: $\nabla u\in L^\infty(\R^2)$ (after integration in time), the boundedness assumption above \eqref{mu:bound} alone is not sufficient.
    We have to assume some regularity for $\mu_0$, and in this paper we take primarily  the   tangential regularity assumption \eqref{TR:mu0} on the coefficient $\mu_0$ with respect to some nondegenerate regular vector field $\tau_0$.
    
\end{enumerate} 

Let us discuss the above assumptions further.

\noindent\textbf{$L^{2+\epsilon}$-estimate.}
Under the boundedness condition \eqref{mu:bound} for the viscosity coefficient, it is straightforward to derive the equivalence of   the $L^2$-norms between $\omega$ and  $a=R_\mu\omega$ (defined in \eqref{Rmu:def} above)
\begin{equation}\label{equiv:a,omega}
    \mu_\ast\|\omega\|_{L^2(\R^2)}\leq \|a\|_{L^2(\R^2)}\leq 8\mu^\ast\|\omega\|_{L^2(\R^2)}.
\end{equation}
Indeed, on one side, by use of the operator norm $1$   of the Riesz operators on $L^2(\R^2)$, we have 
\begin{align}\label{L2:a,omega}
    \|a\|_{L^2(\R^2)}\leq 8\mu^\ast\|\omega\|_{L^2(\R^2)}.
\end{align}
On the other side, by  the fact that $\hbox{id}=R_1R_1+R_2R_2$ and $(R_1R_1+R_2R_2)^2=(R_2R_2-R_1R_1)^2+(2R_1R_2)^2$ (understood as operators defined on $L^2(\R^2)$) and the symmetry of the double Riesz transform on $L^2(\R^2)$, we derive
\begin{align*}
    \mu_\ast &\Vert \omega \Vert_{L^2(\R^2)}^2 
     = \mu_\ast \Bigl\langle  \omega, (R_1R_1+R_2R_2)^2\omega \Big\rangle_{L^2(\R^2)} = \mu_\ast  \Bigl\langle  \omega, \bigl((R_2R_2-R_1R_1)^2+(2R_1R_2)^2\bigr)\omega \Big\rangle_{L^2(\R^2)} \\
    &=\mu_\ast \Bigl\langle (R_2R_2-R_1R_1) \omega, (R_2R_2-R_1R_1) \omega \Big\rangle_{L^2(\R^2)}
    +\mu_\ast\Bigl\langle (2R_1R_2)\omega,  (2R_1R_2)\omega \Big\rangle_{L^2(\R^2)}\\
    &\leq \Bigl\langle \mu (R_2R_2-R_1R_1) \omega, (R_2R_2-R_1R_1) \omega \Big\rangle_{L^2(\R^2)}
    + \Bigl\langle \mu (2R_1R_2)\omega,  (2R_1R_2)\omega \Big\rangle_{L^2(\R^2)}
    \\
    &= \Bigl\langle (R_2R_2-R_1R_1) \mu (R_2R_2-R_1R_1) \omega, \omega \Big\rangle_{L^2(\R^2)}
    + \Bigl\langle (2R_1R_2) \mu (2R_1R_2)\omega,  \omega \Big\rangle_{L^2(\R^2)}
    \mathop{=}\limits^{\hbox{\eqref{Rmu:def}}}\langle a, \omega\rangle_{L^2(\R^2)},
\end{align*}
which, together with the Cauchy-Schwarz  inequality, implies that
\begin{align}\label{L2:omega,a}
    \Vert \omega \Vert_{L^2(\R^2)} \leq \frac{1}{\mu_\ast} \Vert a \Vert_{L^2(\R^2)} .
\end{align}

Without any further regularity assumptions on $\mu$ than \eqref{mu:bound}, we can indeed improve this estimate in $L^2(\R^2)$ to $L^p(\R^2)$ for $p>2$ close to $2$, as described in the following lemma.
\begin{lemma}[$L^{2+\epsilon}(\R^2)$-estimate] \label{Rmu-inv:prop}
Let $\mu\in L^\infty(\R^2;[\mu_\ast,\mu^\ast])$ be a function with a positive lower and upper bound. Then there exists an $\epsilon_0>0$ depending only on $\mu_\ast,\mu^\ast$, such that the operator $R_\mu$ in (\ref{Rmu:def}) defines an isomorphism on $L^{2+\epsilon}(\R^2)$, for all $\epsilon\in (0,\epsilon_0]$.
\end{lemma}

The proof is postponed to Appendix \ref{sect:intro-proofs}, and is strongly related to the ellipticity \eqref{ellipticity}  of the operator $L_\mu$.
For the remainder of this paper we fix   $\epsilon>0$ given by Lemma \ref{Rmu-inv:prop}, and without loss of generality we assume $\epsilon\leq 2$.
By the relation $\nabla u = RR^\perp \omega$ with the Riesz transform $R=\frac{\frac 1i\nabla}{\sqrt{-\Delta}}$ and $R^\perp=\frac{\frac 1i\nabla^\perp}{\sqrt{-\Delta}}$, we have the a priori $L^{2+\epsilon}(\R^2)$-estimate
\begin{equation}\label{nablau:eps}
\|\nabla u\|_{L^{2+\epsilon}(\R^2)}=\|RR^\perp\omega\|_{L^{2+\epsilon}(\R^2)}\leq C\|\omega\|_{L^{2+\epsilon}(\R^2)}\leq C(\mu_\ast,\mu^\ast)\|R_\mu\omega\|_{L^{2+\epsilon}(\R^2)}
=C(\mu_\ast,\mu^\ast)\|a\|_{L^{2+\epsilon}(\R^2)},
\end{equation}
where we used the $L^{2+\epsilon}(\R^2)$-boundedness of the Riesz-transform $R,R^\perp$ and $R_\mu^{-1}$ in the first and second inequality, respectively. 
Notice that by virtue of \eqref{failure:Lp} above, $\epsilon$ depends on $\mu_\ast,\mu^\ast$ and we can not take $\epsilon$ arbitrarily large.

The $L^{2+\epsilon}$-estimate of $R_\mu^{-1}$ plays an important role in deriving the Lipschitz estimate for the velocity field later.
It is related to the Sobolev embedding $W^{1,2+\epsilon}(\R^2)\hookrightarrow L^\infty(\R^2)$, which can be compared to the failure of the Sobolev embedding $H^1(\R^2)\not\hookrightarrow L^\infty(\R^2)$ in space dimension two. 
Specifically, we use the \textit{a priori} estimate of the tangential derivative $\d_\tau\nabla u$ in terms of $\d_\tau a$ in $L^{2+\epsilon}(\R^2)$ later.

\bigbreak 

\noindent\textbf{Vector field $\tau$.}
It is time to discuss the nondegenerate vector field  
$\tau=\tau(t,x)\in \R^2$.
It  is transported by the velocity field $u$ of the Navier-Stokes flow as follows
\begin{equation}\label{eq:tau}
\left\{\begin{array}{l}
    \d_t\tau+u\cdot\nabla\tau=\tau\cdot\nabla u,
    \\
    \tau|_{t=0}=\tau_0,
    \end{array}\right.
\end{equation}
that is, the tangential derivative $\d_\tau:=\tau\cdot\nabla$ commutes with the material derivative $\frac{D}{Dt}:=\d_t+u\cdot\nabla$, as
\begin{equation*}
[\frac{D}{Dt}, \d_\tau]=(\frac{D}{Dt}\tau-\d_\tau u)\cdot\nabla
=(\d_t\tau+u\cdot\nabla\tau-\tau\cdot\nabla u)\cdot\nabla=0.
\end{equation*}
This, together with the free transport equation $\eqref{muNS}_1$:  $\frac{D}{Dt}\mu=0$, implies  the free transport of the tangential derivative $\d_\tau \mu$:
\begin{equation}\label{eq:dtaumu} 
  \d_\tau\frac{D}{Dt}\mu=0
  \Leftrightarrow \frac{D}{Dt}(\d_\tau\mu)=0.
\end{equation}
The $L^p(\R^2)$-norm of $\d_\tau\mu$ is hence preserved by the flow a priori, $p\in [1,\infty]$. 
Nevertheless, the tangential regularity of $\mu$ with respect to the vector field $\tau$ involves not only $\|\d_\tau\mu\|_{L^{p_0}([0,t)\times\R^2)}$, but also the  regularity of the vector field $\tau$ itself  (see e.g. \cite{chemin1993persistance})
$$ 
\|\nabla \tau\|_{L^\infty([0,t); L^{p_0}(\R^2; \R^{2\times2}))}, 
$$
for some ${p_0}\in (2,\infty)$.
Technically this regularity requirement comes for instance from   estimating  the commutator of type 
$[\d_\tau, \nabla]f=\nabla\tau\cdot\nabla f.$

 We take the spatial derivative to the $\tau$-equation \eqref{eq:tau} and test it by $|\nabla\tau|^{p_0-2}\nabla\tau$, to derive the following bound for $\nabla \tau$ 
    \begin{align}\label{tau:exp}
        \Vert \nabla \tau \Vert_{L^\infty([0,t); L^{p_0}(\R^2))} &\leq  \Bigl(\Vert \nabla \tau_0 \Vert_{L^{p_0}(\R^2)} + \int_0^t \Vert \nabla \d_\tau  u \Vert_{L^{p_0}(\R^2)} dt'\Bigr)
        \exp( \|\nabla u \|_{L^1([0,t];L^\infty(\R^2))}).
    \end{align}
    Notice that the time-space norm $\|\nabla \tau\|_{L^\infty([0,t); L^{p_0}(\R^2))}$ grows exponentially in the time integration of the Lipschitz-norm of the velocity field as
$\exp( \|\nabla u \|_{L^1([0,t];L^\infty(\R^2))})$. 
In order to finally achieve a global-in-time control of $ \|\nabla u \|_{L^1([0,t];L^\infty(\R^2))}$ by use of the tangential regularity, we need some \textit{smallness} assumption on initial data 
to complete the bootstrap argument.

\subsection{Main results}\label{subs:result}

Our main result reads as follows.

\begin{theorem}[Global-in-time well-posedness of (\ref{muNS})-\eqref{eq:tau}]\label{exthm}
Let $\mu_0 \in L^\infty(\R^2;[\mu_\ast,\mu^\ast])$, $0<\mu_\ast\leq \mu^\ast$, be an initial viscosity function satisfying $\mu_0-1\in L^2(\R^2)$, and let $u_0\in H^1 \cap \dot H^{-1}(\R^2;\R^2)$ be a divergence-free vector field.
Furthermore, let $\tau_0\in L^\infty(\R^2;\R^2)$ be a nondegenerate  vector field  such that $\vert \tau_0 \vert^{-1}\in L^\infty(\R^2)$ and $(\nabla\tau_0,  \d_{\tau_0}\mu_0)\in  L^{2+\epsilon}(\R^2; \R^{2\times2+1})$ in the sense of distributions, where $\epsilon=\epsilon(\mu_\ast,\mu^\ast)>0$ is given by Lemma \ref{Rmu-inv:prop}.  

If the following smallness condition is fulfilled
\begin{align}\label{u0:cond}
&\Vert u_0 \Vert_{L^2(\R^2)}^{\frac{\epsilon}{2}} \cdot \Bigl(\|u_0\|_{\dot H^{-1}(\R^2)}+\|\mu_0-1\|_{L^2(\R^2)}\|u_0\|_{L^2(\R^2)}\Bigr)
\cdot\Bigl(\|\nabla u_0\|_{L^2(\R^2)}+\bigl\|(\nabla\taub_0, \d_{\taub_0}\mu_0)\bigr\|_{L^{2+\epsilon}(\R^2)}^{\frac{2+\epsilon}{\epsilon}}\Bigr)    \leq c_0,  
\end{align}
where $\taub_0=\frac{\tau_0}{|\tau_0|}$ and $c_0$ is a positive constant  depending only on $\mu_\ast, \mu^\ast$, then the system (\ref{muNS})-\eqref{eq:tau} supplemented with the initial data $(\mu_0,u_0, \tau_0)$ has a unique global-in-time solution $(\mu,u, \nabla \pi, \tau)$ satisfying
\begin{align}\label{exspace}
\begin{split}
    &\mu \in L^\infty([0,\infty)\times\R^2;[\mu_\ast,\mu^\ast]),
    \quad \mu-1\in C_b([0,\infty); L^q(\R^2) ), \,\,\forall q\in [2,\infty), \\
    &u\in C_b([0,\infty);L^2(\R^2;\R^2)) \cap L^2([0,\infty);\dot H^1(\R^2;\R^2)), \\
    &\nabla u \in C_b([0,\infty);L^2(\R^2;\R^{2\times 2}))\cap L^1((0,\infty);L^\infty(\R^2; \R^{2\times 2})) , \\
    &\nabla(\pi-b)\in L^2((0,\infty);L^2(\R^2;\R^2)) ,
  \\
    &\tau \in L^\infty([0,\infty);L^\infty \cap \dot W^{1,2+\epsilon}(\R^2;\R^2)),\quad \frac{1}{|\tau|}\in L^\infty([0,\infty)\times\R^2),
    \\
    &\d_\tau\mu\in L^\infty([0,\infty);L^{2+\epsilon}(\R^2)) \hbox{ in the distribution sense},
\end{split}
\end{align}
where $b=Q_\mu\omega$, with $\omega=\nabla^\perp\cdot u$, is defined in (\ref{Qmu:def}) above. 

Furthermore, we have 
\begin{itemize}
    \item Energy estimates for the ``good  unknown'' $a=R_\mu \omega$ defined in \eqref{Rmu:def}
\begin{align}\label{a:reg}
    \begin{split}
        &a\in C_b([0,\infty);L^2(\R^2))\cap L^2((0,\infty);\dot H^1(\R^2)), \\
        &t^{\frac12}\nabla a\in L^\infty((0,\infty);L^2(\R^2; \R^2)) \cap L^2((0,\infty);\dot H^1(\R^2; \R^2));
    \end{split}
\end{align} 
 \item $W^{1,2+\epsilon}(\R^2)$-boundedness 
    \begin{equation}\label{alpha:intro0}
      a, \alpha, \d_\tau u \in L^1((0,\infty); W^{1,2+\epsilon}(\R^2)), \hbox{ with } \alpha= (\taub\otimes n):  (\mu Su)=\taub\cdot(n\cdot\mu Su)=\taub\cdot(\mu Su \, n),
    \end{equation}
where  $\taub=\frac{\tau}{|\tau|}$ and $n=\frac{\tau^\perp}{|\tau|}$ denote the (unit) tangential and normal vectors respectively;
\item $H^1(\R^2)$-boundedness for the material derivative $\frac{D}{Dt}u=\d_t u+u\cdot\nabla u$  and the divergence of the total stress tensor $T(u,\pi)=\mu Su-\pi\hbox{Id}$
\begin{align}\label{Tensor}
 \frac{D}{Dt}u=\div T(u,\pi)\in L^2((0,\infty); L^2(\R^2;\R^2)), \quad 
 t^{\frac12}\frac{D}{Dt}u=t^{\frac12}\div T(u,\pi)\in   L^2((0,\infty);\dot H^1(\R^2; \R^2)).
\end{align}
\end{itemize}  


\end{theorem}

Let us make a few comments on the   results in Theorem \ref{exthm}.
The proof ideas for the global-in-time a priori estimates are discussed in Subsection \ref{subs:proof} below, and the proof of Theorem \ref{exthm} is found in  Subsection \ref{sect:proofs}.
\begin{remark}
\begin{enumerate}[(i)]
    \item (Jump of $\d_n u$ in case of jumping $\mu$).
  We have the following expression for the normal derivative of the velocity $\d_n u$ by use of $\alpha,$ $\mu$, $\taub$, $n$ and the tangential derivative $\dtau u$  (see \eqref{alpha:reform} below)
    \begin{equation}\label{dnu}
    \d_n u=\d_n\nabla^\perp \phi= \frac{\alpha}{\mu}   \taub  - 2 (n \cdot \dtau  u)   \taub  -  (\dtau u)^\perp.\end{equation}
    The regularity of $\tau$ in \eqref{exspace} and the regularity of $\alpha, \dtau u$ in \eqref{alpha:intro0} imply that $\d_n u$ has  a jump exactly when $\mu$ has a jump. 
    See Corollary \ref{propthm} below also for the smooth viscosity coefficient case.

    \item (``Local good  unknown'' $\alpha$). 
    Motivated by the derivation of the fourth-order elliptic operator $L_\mu$ as $L_\mu\phi=\nabla^\perp\cdot(\nabla\cdot (\mu Su))$ (recalling \eqref{Lmu:compute}), we define $\alpha$  as $\alpha=\taub\cdot(n\cdot\mu Su)$ (recalling \eqref{alpha:intro0}). We have the  relation between $L_\mu\phi$ and $\alpha$ of the following form (see \eqref{alpha:eq} below for more details)  
    \begin{align}\label{alpha:intro}
        L_\mu\phi=\d_n^\ast\d_n\alpha+\sum_{j=1}^2\d_j\Bigl(  \bigl(A_j (\dtau\mu)+B_j(\mu \nabla\taub)\bigr):    \nabla u 
        +C_j \mu :\dtau(\nabla u)\Bigr),
    \end{align}
    where the coefficients $A_j=A_j(\taub), B_j=B_j(\taub)$, $C_j=C_j(\taub)$, $j=1,2$ are uniformly bounded. 
    Thus, the $L^{2+\epsilon}(\R^2)$-boundedness of $\nabla \alpha$ follows from the tangential regularity 
    $$(TR):  \quad\bigl(  (\dtau\mu, \nabla\taub)\otimes\nabla u,\,\, \dtau \nabla u \bigr)  \in L^{2+\epsilon}(\R^2),  $$
    and $L_\mu\phi\in \dot W^{-1,2+\epsilon}(\R^2)$, which is equivalent to $\nabla a\in L^{2+\epsilon}(\R^2)$  by virtue of the relation $L_\mu\phi=\Delta a$ (recalling \eqref{Lmu:def}). 
    Hence, the fact that $a\in W^{1,2+\epsilon}(\R^2)$ and the tangential regularity  (TR) imply $\alpha\in W^{1,2+\epsilon}(\R^2)$, and thus $\nabla u\in L^\infty(\R^2)$ follows from \eqref{dnu} and Gagliardo-Nirenberg's inequality (GN)   (see \eqref{u:Lip:eps} below):
    \begin{align*}
        &\|\nabla u\|_{L^\infty(\R^2)}
         \mathop{\lesssim }\limits^{\hbox{\eqref{dnu}}}\|\alpha\|_{L^\infty(\R^2)}+\|\d_{\taub} u\|_{L^\infty(\R^2)}
        \mathop{\lesssim}\limits^{\hbox{(GN)}} \|\alpha\|_{L^{2+\epsilon}(\R^2)}^{\frac{\epsilon}{2+\epsilon}}\|\nabla\alpha\|_{L^{2+\epsilon}(\R^2)}^{\frac{2}{2+\epsilon}}+ \|\d_{\taub} u\|_{L^{2+\epsilon}(\R^2)}^{\frac{\epsilon}{2+\epsilon}}\|\nabla\d_{\taub} u\|_{L^{2+\epsilon}(\R^2)}^{\frac{2}{2+\epsilon}}
        \\
        &\mathop{\lesssim}\limits^{\hbox{\eqref{nablau:eps}}} \|a\|_{L^{2+\epsilon}(\R^2)}^{\frac{\epsilon}{2+\epsilon}}
        \|(\nabla\alpha, \nabla\d_{\taub}u)\|_{L^{2+\epsilon}(\R^2)}^{\frac{2}{2+\epsilon}},
        \hbox{ with }\nabla\d_{\taub}u=\d_{\taub}\nabla u+[\nabla,\d_{\taub}]u=\d_{\taub}(RR^\perp R_\mu^{-1}a)+\nabla\taub\cdot\nabla u,
        \\
&        \mathop{\lesssim}\limits^{\hbox{\eqref{alpha:intro}}}
        \|a\|_{L^{2+\epsilon}(\R^2)}^{\frac{\epsilon}{2+\epsilon}}
        \Bigl( \|\nabla a\|_{L^{2+\epsilon}(\R^2)} + 
        \Vert (\nabla \taub, \dtau\mu) \Vert_{L^{2+\epsilon}}    \Vert (\nabla u, a) \Vert_{L^\infty} \Bigr)^{\frac{2}{2+\epsilon}}.
    \end{align*}
    This is the key step in deriving the Lipschitz estimate for the velocity field, where the smallness assumption \eqref{u0:cond} helps to close the bootstrap argument. 
    
    Observe that if we multiply the  jump condition $\sigma Hn=\llbracket 
    T(u,\pi)n  \rrbracket$ in (\ref{2pNS})  on the interface  $\Gamma_t$ by the continuous tangent vector $\overline{\tau}$ we derive that
\begin{align*}
  0=\taub\cdot\sigma Hn = \llbracket 
  \taub\cdot(T(u,\pi)n) \rrbracket =\llbracket 
  \taub\cdot\bigl(  \mu Su\, n\bigr) \rrbracket=\llbracket\alpha\rrbracket,
\end{align*}
where we used the definition  $T(u,\pi)=\mu Su-\pi\hbox{Id}$.
Thus, $\alpha$ is continuous, which is consistent with our analysis.
The idea of multiplication by the tangent vector has appeared e.g. in Nalimov's formulation of the one-dimensional water waves problem \cite{Nalimov}.
We believe that our definition and analysis of $\alpha$ in the variable viscosity setting is new.

Recall the decomposition \eqref{divmuSu:decomp}:
\begin{equation*}
    \div T(u,\pi)=\nabla^\perp a+\nabla(b-\pi)=\div\Bigl(a\begin{pmatrix}
        0&-1\\1&0
    \end{pmatrix}+(b-\pi)\hbox{Id}\Bigr)=:\div\widetilde T(u,\pi).
\end{equation*} 
Hence on any (well-defined) interface $\Gamma_t$ with $\taub$ and $n$ as the tangential and normal unit vectors respectively, 
\begin{align}
    &T(u,\pi)n=\widetilde T(u,\pi)n=-a\overline{\tau}+(b-\pi)n\hbox{ is continuous on }\Gamma_t;\label{Tn}
    \\
    &\alpha=\taub\cdot(T(u,\pi)n)=\taub\cdot(\widetilde T(u,\pi)n)=-a \hbox{ on }\Gamma_t.\label{alpha:-a}
\end{align}
    Notice that $\alpha$ is determined by $\nabla u$ in \eqref{alpha:intro0}  ``locally'', while $a=R_\mu\omega$ is determined by $\nabla u$ in terms of the Riesz operators in \eqref{Rmu:def} ``nonlocally''.
    The ``local good unknown" $\alpha$ and the ``global good unknown" $-a$ coincide on the interface $\Gamma_t$, and indeed also in $W^{1,2+\epsilon}(\R^2)$, up to   tangential regularity terms (see \eqref{eq:a,alpha} below).
    
    \item (Assumptions revisited). The proof of Theorem \ref{exthm} shows that the condition $u_0\in  \dot H^{-1}(\R^2)$ can  be relaxed to $u_0\in  \dot H^{-2\delta}(\R^2)$ for some $\delta\in (0,\frac12)$ sufficiently close to $\frac12$ (depending on $\epsilon$). 
    We can also replace the assumption $\d_{\tau_0}\mu_0 \in L^{2+\epsilon}(\R^2)$ by $\d_{\tau_0}\mu_0 \in L^r(\R^2)$ for some $r\in (2,\infty]$, as seen from the proof of Proposition \ref{u:Lip:prop} and the commutator estimates \eqref{Rmu:comm}, \eqref{Rmu:comm+}. 
     It is however unclear whether $\dot W^{1,p}$-regularity, for $p>2+\epsilon_0$ away from $2$, of the vector field $\tau_0$ can be propagated.
    This is related to the question whether $a\in H^2$ can control  $\dot W^{1,p}$-regularity of the right hand side $\tau\cdot\nabla u=\d_\tau u$ of \eqref{eq:tau}, or equivalently  $\d_\tau\omega\in L^p$.
    Heuristically, for this one has to show $\d_\tau^2\omega\in L^2$, and the latter requires further regularity assumptions on $\tau$, say $\nabla\d_{\tau_0}{\tau_0}\in L^2$.
    We plan to investigate this high regularity case in the near future.

    The low frequency control by $\|u_0\|_{\dot H^{-1}}$ and $\|\mu_0-1\|_{L^2}\|u_0\|_{L^2}$ provides sufficient time decay (see Proposition \ref{u-decay:prop} below), while the high frequency control by $\|\nabla u_0\|_{L^2}$ and $\|(\nabla\bar\tau_0, \d_{\bar\tau_0}\mu_0)\|_{L^{2+\epsilon}}$  provides sufficient regularity (see Proposition \ref{energy:prop}).
    The combination of these   bounds on the left hand side in \eqref{u0:cond}, which is invariant under the scaling 
    $$(\mu_{0,\lambda}, u_{0,\lambda}, \taub_{0,\lambda})(x)=(\mu_0, \lambda^{-1}u_0, \taub_0)(\lambda^{-1}x),\quad \lambda>0,$$
    controls the critical norm $\|\nabla u\|_{L^1_tL^\infty_x}$ (see Proposition \ref{u:L1Lip:prop}).   
    In particular, \eqref{u0:cond} permits arbitrarily large initial norms $\|\mu_0-1\|_{L^2}$ and $\|(\nabla\bar\tau_0, \d_{\bar\tau_0}\mu_0)\|_{L^{2+\epsilon}}$, as long as the norm $\|u_0\|_{L^2}$ is sufficiently small.

    Due to \eqref{failure:Lp}, we expect finite-time formation of singularity if no regularity assumptions are imposed on the significantly varying viscosity coefficient.
   
\end{enumerate}
\end{remark}



We have the following  consequences of (the proof of) Theorem \ref{exthm}. The proofs can be found in Subsection \ref{sect:d-prop}.

\begin{corollary}\label{propthm}
Recall the systems \eqref{muNS}, \eqref{eq:theta}, \eqref{NS}, \eqref{2pNS}, \eqref{mu2pNS} and \eqref{eq:tau}.
\begin{enumerate}
    \item \label{Coro1}(Viscosity patch-type problem  for \eqref{muNS}). 
   Let the initial viscosity be of the form
\begin{align}\label{vpatch-init}
    \mu_0(x)=  \mu_{0}^+(x) 1_{D}(x) +\mu_0^-(x) 1_{D^C}(x),
    \hbox{ such that }\mu_0\in [\mu_\ast, \mu^\ast]
    \hbox{ with }0<\mu_\ast\leq1\leq\mu^\ast,
\end{align} 
where 
  $D\subset\R^2$ is a bounded, simply connected domain, such that its boundary $\d D$ is of class $W^{2,2+\epsilon}(\R^2)$, and 
  $\mu_{0}^+\in W^{1,2+\epsilon}(\overline{D})$   is a positive continuous bounded function defined on $\overline{D}$ while $\mu_0^--1\in L^2\cap W^{1,2+\epsilon}(\overline{D^C})$ is a  continuous bounded function defined on $\overline{D^C}$.
  Here  $\epsilon>0$ depends only on $\mu_\ast,\mu^\ast$ and  is given in Lemma \ref{Rmu-inv:prop}.    Let  $u_0\in H^1 \cap \dot H^{-1}(\R^2;\R^2)$ be divergence-free.
    
    If there exists a vector field $\tau_0\in L^\infty\cap \dot W^{1,2+\epsilon} (\R^2;\R^2)$ with $\frac{1}{|\tau_0|}\in L^\infty(\R^2)$ such that $\tau_0$ is tangential to   the boundary $\d D$   and the initial condition \eqref{u0:cond} holds,
    then the system (\ref{muNS}) supplemented with the initial data $(\mu_0,u_0)$ has a unique global-in-time solution $(\mu,u,\nabla\pi)$ which satisfies the estimates in Theorem \ref{exthm}.
    Furthermore,  for all times $t>0$, $$\mu(t,\cdot)=   \mu^+(t,\cdot) 1_{D_t}(x) +  \mu^-(t,\cdot) 1_{ (D_t)^C}(x),$$
    where $D_t\subset \R^2$ is a bounded, simply connected domain   whose boundary is of class $W^{2,2+\epsilon}(\R^2)$, and $\mu^+(t,\cdot)\in W^{1,2+\epsilon}(\overline{D_t})$, $\mu^-(t,\cdot)-1\in L^2\cap W^{1,2+\epsilon}(\overline{D_t^C})$.
    Correspondingly, this solution solves the two-phase Navier-Stokes equations with constant density \eqref{mu2pNS}, with  $\Omega_t^+=D_t$, $\Omega_t^-=D_t^C$ and the interface $\Gamma_t= \partial D_t$.

\item \label{Coro2}(Smooth viscosity coefficient case for \eqref{muNS}).
Let $\mu_0\in L^\infty \cap \dot W^{1,q}(\R^2; [\mu_\ast,\mu^\ast])$ with $q\in (2,\infty]$ and $0<\mu_\ast\leq\mu^\ast$, such that $\mu_0-1\in L^2(\R^2)$. Let $u_0\in H^1\cap \dot H^{-1}(\R^2;\R^2)$ be divergence-free. 

If there exists a nondegnerate vector field $\tau_0\in L^\infty\cap \dot W^{1,2+\epsilon}(\R^2;\R^2)$ such that  \eqref{u0:cond} holds for some $\epsilon=\epsilon(\mu_\ast,\mu^\ast)\in (0,q-2]$ given in Lemma \ref{Rmu-inv:prop}, then   Theorem \ref{exthm} holds, and   additionally 
$\mu\in L^\infty([0,\infty);\dot W^{1,q}(\R^2))$ and
\begin{align}\label{gradu:H1}
    \nabla u \in L^\infty([0,\infty);L^2(\R^2;\R^{2\times 2})) \cap L^2((0,\infty);\dot H^1(\R^2;\R^{2\times 2})).
\end{align}
In particular, the following smallness condition, which is the initial condition \eqref{u0:cond} with a nonzero constant vector field $\taub_0=\begin{pmatrix}
       1\\0
   \end{pmatrix}$,   
   \begin{align*}
       \Vert u_0 \Vert_{L^2(\R^2)}^{\frac{\epsilon}{2}} \cdot \bigl(\|u_0\|_{\dot H^{-1}(\R^2)}+\|\mu_0-1\|_{L^2(\R^2)}\|u_0\|_{L^2(\R^2)}\bigr)
\cdot\Bigl(\|\nabla u_0\|_{L^2(\R^2)}+\|  \d_{1}\mu_0 \|_{L^{2+\epsilon}(\R^2)}^{\frac{2+\epsilon}{\epsilon}}\Bigr)    \leq c_0
   \end{align*}
implies the well-posedness results in Theorem \ref{exthm}.

\item  \label{Coro3}(Lower bound for  existence time of solutions to the Boussinesq equations without heat conduction \eqref{eq:theta}). 
Let $u_0\in H^1(\R^2;\R^2)$ be a divergence-free vector field and $\vartheta_0 \in L^1\cap L^r(\R^2)$ for some $r\in (2,\infty]$.
Assume the dependence of the viscosity coefficient $\mu$ on the temperature function $\vartheta$ to be $\mu={\mu}_\vartheta(\vartheta)$ for some $\mu_\vartheta\in C_b(\R;[\mu_\ast,\mu^\ast])$, $0<\mu_\ast\leq \mu^\ast$.
Let $\tau_0\in L^\infty(\R^2;\R^2)$ be a vector field such that $\vert \tau_0 \vert^{-1} \in L^\infty(\R^2)$ and $(\nabla \tau_0,\d_{\tau_0}\mu_0)\in L^{2+\epsilon}(\R^2;\R^{2\times 2+1})$, for some $\epsilon=\epsilon(\mu_\ast,\mu^\ast)\in (0,r-2]$ given in Lemma \ref{Rmu-inv:prop}.

Then there exists a positive time $T>0$, which can be bounded from below as follows
\begin{align}\label{B:cond}
\begin{split}
&T \geq   c_1 \Bigl( \max    \Bigl\{
\Vert \vartheta_0 \Vert_{L^q(\R^2)}^{\frac{1}{\frac32-\frac1q}}, \,
\Bigl( \Vert \vartheta_0\Vert_{L^q(\R^2)}^{\frac{1}{\frac32-\frac1q}} +\sigma_1^2\Bigr)
\cdot\Bigl(\Vert u_0 \Vert_{L^2(\R^2)}^{\theta_1^B}+\|u_0\|_{L^2(\R^2)}^{ \theta_2^B}\Bigr), \,
\Bigl(\Vert \vartheta_0 \Vert_{L^q(\R^2)} 
 \sigma_1 ^{\theta_3^B} \Bigr)^{\frac{1}{\frac32-\frac1q+\frac{\theta_3^B}{2}}},\,
  \\
& \Bigl(\Vert \vartheta_0 \Vert_{L^q(\R^2)} 
 \sigma_1^{ \theta_4^B} \Bigr)^{\frac{1}{\frac32-\frac1q+\frac{ \theta_4^B}{2}}},\, q=1,2+\epsilon
\Bigr\}\Bigr)^{-1},\, \sigma_1 =\Vert \nabla u_0 \Vert_{L^2(\R^2)} + \Vert (\nabla \taub_0, \d_{\taub_0}\mu_0) \Vert_{L^{2+\epsilon}(\R^2)}^\frac{2+\epsilon}{\epsilon},
\end{split}
\end{align}
where $\theta_1^B,  \theta_2^B, \theta_3^B,  \theta_4^B, 
c_1$ are positive constants depending only on $\mu_\ast,\mu^\ast$, 
such  that the system \eqref{Boussinesq}-\eqref{eq:tau} supplemented with the initial data $(\vartheta_0,u_0,\tau_0)$ has a unique solution $(\vartheta,u,\nabla\pi,\tau)$ on the time interval $[0,T]$, which  satisfies
$\vartheta\in C_b([0,T]; \cap_{1\leq \tilde r\leq r, \tilde r<\infty} L^{\tilde r}(\R^2))\cap  L^\infty([0,T];L^1\cap L^r(\R^2))$  and \eqref{exspace} on   $[0,T]$, except the property for $\mu-1$ in \eqref{exspace}.

Furthermore, for the quantity $a_\vartheta: = a - \mathcal{R}_{-1}\vartheta$, with $a=R_\mu \omega$  defined in \eqref{Rmu:def} and $\mathcal{R}_{-1}:= \d_1 (-\Delta)^{-1}$, we have the energy estimates
\begin{align*}
    & a_\vartheta \in  C_b([0,T]; L^2(\R^2))\cap L^2 ([0,T];L^2(\R^2;\R^2)), \\
    &t^{\frac12} \nabla a_\vartheta \in L^\infty ([0,T];L^2(\R^2;\R^2)) \cap L^2 ([0,T];\dot H^1(\R^2;\R^2)).
\end{align*}
We also have   $a,\alpha,\d_{\tau}u\in L^1([0,T]; W^{1,2+\epsilon}(\R^2))$ and $\frac{D}{Dt}u=\div T(u,\pi)+\vartheta e_2\in L^1([0,T]; L^{2+\epsilon}(\R^2))$, with the same notations $ \alpha, \frac{D}{Dt}u, T(u,\pi)$ as given in Theorem \ref{exthm}.

\item\label{d-propthm} (Global-in-time well-posedness of   the density-dependent incompressible Navier-Stokes equations \eqref{NS}-\eqref{eq:tau}). Let $\rho_0\in L^\infty(\R^2;[\rho_\ast,\rho^\ast])$, $0<\rho_\ast\leq\rho^\ast$, be an initial density satisfying $\rho_0-1\in L^2(\R^2)$. Assume the dependence of the viscosity coefficient $\mu$ on the density function $\rho$ to be $\mu={\mu}_\rho(\rho)$ for some $ \mu_\rho \in W^{1,\infty}([\rho_\ast,\rho^\ast];[\mu_\ast,\mu^\ast])$ with $0<\mu_\ast\leq\mu^\ast$. 
Let $u_0\in H^1 \cap \dot H^{-1}(\R^2;\R^2)$  and $\tau_0 \in L^\infty(\R^2;\R^2)$   such that $\vert \tau_0 \vert^{-1}\in L^\infty(\R^2)$, $(\nabla \tau_0, \d_{\tau_0}\rho_0)\in L^{2+\epsilon}(\R^2;\R^{2\times 2+1})$,  and
\begin{equation}\label{NSu0:cond}
\begin{split}
   &e^{c_2\Vert u_0 \Vert_{L^2(\R^2)}^2}
   \Bigl(\Vert u_0 \Vert_{L^2(\R^2)} + \Vert \rho_0-1 \Vert_{L^2(\R^2)} \Vert \nabla u_0 \Vert_{L^2(\R^2)}\Bigr)^{\frac{\epsilon}{2}}
   \\
   &\quad \cdot 
   \Bigl(\|u_0\|_{\dot H^{-1}(\R^2)}+\|\rho_0-1\|_{L^2(\R^2)} \|u_0\|_{L^2(\R^2)}\Bigr)\cdot \Bigl(\|\nabla u_0\|_{L^2(\R^2)} + \Vert (\nabla \taub_0, \d_{\taub_0}\mu_0) \Vert_{L^{2+\epsilon}(\R^2)}^\frac{2+\epsilon}{\epsilon}\Bigr) \leq c_3,  
\end{split}
\end{equation}
for some  $\epsilon>0$  given by Lemma \ref{Rmu-inv:prop},
where $c_2, c_3$ are positive constants depending only on $\rho_\ast, \rho^\ast, \mu_\ast,\mu^\ast$ and $\Vert \mu_\rho' \Vert_{L^\infty([\rho_\ast,\rho^\ast])}$.
Then the  system (\ref{NS})-\eqref{eq:tau} supplemented with the initial data $(\rho_0,u_0,\tau_0)$ has a unique global-in-time solution $(\rho,u,\nabla\pi,\tau)$ such that \eqref{exspace} holds, with $\mu$ replaced by $\rho$.
{}
Furthermore, we have the energy estimates
\begin{align*}  &a\in C_b([0,\infty);L^2(\R^2))\cap L^2((0,\infty);\dot H^1(\R^2)), \\
        &t^{\frac12}\frac{D}{Dt}u\in L^\infty((0,\infty);L^2(\R^2; \R^2)) \cap L^2((0,\infty);\dot H^1(\R^2; \R^2));
\end{align*}
and the bounds $a,\alpha,\d_{\tau}u\in L^1((0,\infty); W^{1,2+\epsilon}(\R^2))$ and $ \rho\frac{D}{Dt}u= \div T(u,\pi) \in L^1((0,\infty); L^{2+\epsilon}(\R^2))$, with the same notations $ a, \alpha, \frac{D}{Dt}u, T(u,\pi)$ as given in Theorem \ref{exthm}. 

In particular, if the initial density is of the patch-type
\[
\rho_0(x)=\rho_0^+(x)1_D(x)+\rho_0^-(x)1_{D^C}(x), \; \text{such that } \rho_0\in [\rho_\ast,\rho^\ast]\hbox{ with }0<\rho_\ast\leq 1\leq \rho^\ast,
\]
for some bounded, simply connected domain $D\subset\R^2$ with $W^{2,2+\epsilon}$-boundary, and functions $\rho_0^+\in W^{1,2+\epsilon}(\overline{D})$, $\rho_0^- -1\in L^2\cap W^{1,2+\epsilon}(\overline{D^C})$, and if the vector field $\tau_0$ from above is tangential to the boundary $\d D$, then the unique solution above satisfies for all times $t>0$,
\[
\rho(t,x)= \rho^+(t,x)1_{D_t}(x)+\rho^-(t,x)1_{(D_t)^C}(x),
\]
for some bounded, simply connected domain $D_t\subset\R^2$ with $W^{2,2+\epsilon}$-boundary, and functions $\rho^+(t,\cdot)\in W^{1,2+\epsilon}(\overline{D_t})$, $\rho^-(t,\cdot)-1\in L^2\cap W^{1,2+\epsilon}(\overline{D_t^C})$. 
Thus, the density-patch-type problem in the absence of vacuum  for the density-dependent incompressible Navier-Stokes equations \eqref{NS} is uniquely globally-in-time solvable under the smallness assumption \eqref{NSu0:cond}. 
This solution also solves the two-phase Navier-Stokes equations \eqref{2pNS} without surface tension ($\sigma=0$) with $\Omega_t^+=D_t$, $\Omega_t^-=D_t^C$ and the interface $\Gamma_t=\d D_t$.
\end{enumerate}
\end{corollary}

We give some comments below on the above  results.
\begin{remark}
\begin{enumerate}[(i)]
\item (Construction of  a vector field for the viscosity patch-type problem.) There are many different ways to construct a nondegenerate vector field $\tau_0\in L^\infty\cap \dot W^{1,2+\epsilon}(\R^2;\R^2)$ which is tangent to   the boundary $\d D$ given in \eqref{vpatch-init}. 
One way can be described as follows. 

We begin with the simplest case in which   $D=B$ is the unit disk in $\R^2$ with the origin as the center.  
We aim to construct a nondegenerate regular vector field $\tau_B\in L^\infty\cap \dot W^{k,p}(\R^2;\R^2)$, $\forall k\in \N$, $p\in [1,\infty]$ with $|\tau_B|\geq\frac12$, such that the renormalized unit vector field
\begin{align*}
&    \bar \tau_B(x)=\frac{\tau_B}{|\tau_B|}(x) = 
\begin{cases}
\begin{pmatrix}
    -\frac{x_2}{\vert x \vert} \\ \frac{x_1}{\vert x \vert}
\end{pmatrix}=:e_\theta, \quad \text{for } \vert x \vert \in [\frac34, \frac 54], \\ 
  \begin{pmatrix}
    1 \\ 0
\end{pmatrix}=:e_1, \quad \text{for } \vert x \vert \in [0,\frac14] \cup [\frac74,\infty),
\end{cases} 
\end{align*} 
is tangent to the boundary $\d D=\d B=\{x\in\R^2\,|\, |x|=1\}$.
To this end, we connect the tangential vector $e_\theta$ at $|x|=\frac34, \frac54$ to the unit vector $e_1$ at $|x|=\frac14, \frac74$ respectively as follows
\begin{align}\label{tau0-ball}
    \tau_B(r\cos\theta,r\sin\theta) 
    =\begin{cases}
      \begin{pmatrix}
         \sin\bigl(3\pi(r-\frac34)-2\theta(r-\frac14)\bigr) \\ \cos\bigl(3\pi(r-\frac34)-2\theta(r-\frac14)\bigr)
    \end{pmatrix}=:\tau_B^-(r\cos\theta, r\sin\theta),
    &  r\in [\frac14,\frac34],
    \\
   \begin{pmatrix}
         -\sin\bigl(3\pi(r-\frac54)-2\theta(r-\frac74)\bigr) \\ \cos\bigl(3\pi(r-\frac54)-2\theta(r-\frac74)\bigr)
    \end{pmatrix}=:\tau_B^+(r\cos\theta, r\sin\theta),
    &  r\in [\frac54,\frac74],
    \\
    h(r)e_\theta, & r\in [\frac34,\frac54],
    \\
    \tilde h(r,\theta) e_1,  
     &  r\in [0,\frac14]\cup [\frac{7}{4},\infty),  
    \end{cases} 
\end{align}
where we have   connected $e_1|_{r\in [0,\frac18]}, \tau_B^-|_{r\in [\frac14,\frac34]}, e_\theta|_{r\in [\frac78, \frac98]}, \tau_B^+|_{r\in [\frac54,\frac74]}, e_1|_{r\in [\frac{15}{8},\infty)}$ smoothly (noticing $\tau_B^-|_{r=\frac34}=\tau_B^+|_{r=\frac54}=e_\theta$ and $\tau_B^-|_{r=\frac14}=\tau_B^+|_{r=\frac74}=e_1$)  by use of two smooth   functions $h(r),\tilde h(r,\theta)$    satisfying
\[
h(r)\begin{cases}
    =1,  \text{ for  }  
    r\in [\frac14,\frac34]\cup [\frac78,\frac98]  \cup [\frac54, \frac74],
    \\
    \in [\frac12,1],\text{ for  }  
    r\in  [\frac34, \frac78]\cup [\frac98, \frac54],  
\end{cases}  \quad 
\tilde h(r,\theta)  \begin{cases}
    \in [\frac12,1], \text{ for  }  
    r\in  [\frac18,\frac14]\cup [\frac74, \frac{15}8], \\
    =1, \text{ for  }  
    r\in [0,\frac18] \cup [\frac14,\frac74] \cup [\frac{15}8,\infty).
\end{cases} 
\]

Now, for a general bounded, simply connected domain  $D\subset \R^2$, 
 by the Riemann mapping theorem there exists a bijective, holomorphic map $\varphi:D\to B$. Since $\varphi\in W^{1,2+\epsilon}(D;\R^2)$ and the boundary $\d D$ is of class $W^{2,2+\epsilon}$ there exists a $W^{1,2+\epsilon}$-extension $\tilde \varphi :\R^2\to \R^2$ of $\varphi$. Then the vector field
\begin{align}\label{tau0-D}
    \tau_D(x) := \tau_B(\tilde \varphi(x)), \quad x\in \R^2 ,
\end{align}
  is what we search for, since $\tau_D$ is tangent to $\d D$ and
\[
\Vert \tau_D \Vert_{L^\infty(\R^2)}\leq \Vert \tau_B \Vert_{L^\infty(\R^2)}, \quad 
\Vert \nabla \tau_D \Vert_{L^{2+\epsilon}(\R^2)} \leq \Vert \nabla \tau_B \Vert_{L^\infty(\R^2)} \Vert \nabla \tilde \varphi \Vert_{L^{2+\epsilon}(\R^2)} \leq C \Vert \nabla \tau_B \Vert_{L^\infty(\R^2)} ,
\quad |\tau_D|\geq\frac12,
\]
for some constant $C$ depending only on the domain $D$.

As the functions $\mu_0^+\in W^{1,2+\epsilon}(\overline{D})$, $\mu_0^--1\in L^2\cap W^{1,2+\epsilon}(\overline{D^C})$ are arbitrarily large  functions, a large jump across $\d D$ in $\mu_0$ is allowed. The smallness assumption \eqref{u0:cond} with $\bar\tau_D$ above implies the smallness of the initial velocity field $u_0$, in terms of $\mu_\ast,\mu^\ast, \|\mu_0-1\|_{L^2}$, $\|\mu_0^+\|_{W^{1,2+\epsilon}(\overline{D})}$, $\|\mu_0^--1\|_{W^{1,2+\epsilon}(\overline{D^C})}$ and $\|\nabla\bar\tau_D\|_{L^{2+\epsilon}}$.

    \item (Viscosity layer problem and smallness condition revisited.) We can  straightforwardly generalize the results for the viscosity patch-type problem \eqref{vpatch-init} to   the $N$-viscosity layer problem with the initial viscosity   
\begin{align}\label{vlayer-init}
    \mu_0(x)=\sum_{j=1}^N \eta_{j,0}(x) 1_{\Dj}(x) + 1_{(\cup_{j=1}^N \Dj)^C}(x),
    \hbox{ such that }\mu_0\in [\mu_\ast, \mu^\ast].
\end{align} 
Here  $\Dj\subset\R^2$, $j=1,...,N$, are bounded, simply connected domains, such that the boundaries $\d \Dj$ are of class $W^{2,2+\epsilon}(\R^2)$ and are mutually non-intersecting: $\d \Dj \cap \d D^{(i)} =\emptyset$ for $i\neq j$, and $\eta_{j,0}\in W^{1,2+\epsilon}(\overline{\Dj})$ are continuous bounded functions defined on $\overline{\Dj}$, $j=1,\cdots,N$, where $\epsilon>0$ depends only on $\mu_\ast,\mu^\ast$ and  is given in Lemma \ref{Rmu-inv:prop}. Hence, either all the domains $D^{(j)}$, $j=1,\cdots,N$ are  disdjoint,  or   $D^{(i)}\subset D^{(j)}$ for some $i\neq j$.
The key is to construct an initial nondegenerate regular vector field $\tau_0$ which is tangential to all the boundaries $\d\Dj$, $j=1,\cdots,N$.


As an illustrative example, we consider the case where $\eta_{j,0}$ are positive constants, $D^{(j)}$ are discs with strictly increasing radii $r^{(j)}$ and with the origin as the center. There are different choices of initial nondegenerate regular vector fields, e.g.
\begin{itemize}
    \item For each $j=1,...,N$ let $\delta^{(j)}<\frac13\min(r^{(j+1)}-r^{(j)}, r^{(j)}-r^{(j-1)})$ with $r^{(0)}:=0$, let $\chi^{(j)}:\R^2\to[0,1]$ be a smooth cut-off function such that
$\chi^{(j)}(x) = \begin{cases}
    1, \, \hbox{if dist}(x, \d D^{(j)})<\delta^{(j)},  \\
    0, \, \hbox{if dist}(x, \d D^{(j-1)})<\delta^{(j-1)} \hbox{ or dist}(x, \d D^{(j+1)})<\delta^{(j+1)},
\end{cases} 
$ with $\sum_{j}\chi^{(j)}=1$,
   and  let $\tau^{(j)}(x)=\tau_B^{(j)}(\frac{x}{r^{(j)}})$, where $\tau_B^{(j)}(y)$ is defined as in \eqref{tau0-ball} with  $r=|y|$ replaced by $ 1-\frac{1-|y|}{\delta^{(j)}/{r^{(j)}}}$.
   Then $\tau_0(x)=\frac1N \sum_{j=1}^N \chi^{(j)} \tau^{(j)}$ is one choice, such that $\d_{\tau_0}\mu_0=0$ and $\|\nabla\bar \tau_0\|_{L^{2+\epsilon}}^{\frac{2+\epsilon}{\epsilon}}\sim \bigl(\sum_{j=1,\cdots,N}\frac{r^{(j)}}{(\delta^{(j)})^{1+\epsilon}}\bigr)^{\frac1\epsilon}$.
   
   This construction can be easily generalized to other more general cases where the profiles of different boundaries vary largely, such that the distance between two layers  play an important role in the construction and hence the estimates.

   \item Alternatively, we can simply connect $e_1|_{r\in [0,\frac18r^{(1)}]}, e_\theta|_{r\in [r^{(1)}, r^{(N)}]}, e_1|_{r\in [\frac{15}{8}r^{(N)},\infty)}$ smoothly, similarly as in \eqref{tau0-ball}, such that     $\|\nabla\bar\tau_0\|_{L^{2+\epsilon}}^{\frac{2+\epsilon}{\epsilon}}\sim\frac{1}{r^{(1)}}$.
   The smallness assumption \eqref{u0:cond}
   \begin{align}\label{u0:cond,ball}
       \Vert u_0 \Vert_{L^2(\R^2)}^{\frac{\epsilon}{2}} \cdot \bigl(\|u_0\|_{\dot H^{-1}(\R^2)}+\|\mu_0-1\|_{L^2(\R^2)}\|u_0\|_{L^2(\R^2)}\bigr)
\cdot\bigl(\|\nabla u_0\|_{L^2(\R^2)}+\frac{1}{r^{(1)}}\bigr)    \leq \tilde c_0,
   \end{align}
   implies the smallness of $u_0$, (only) in terms of   $\mu_\ast, \mu^\ast$, $\|\mu_0-1\|_{L^2}$ and $\frac{1}{r^{(1)}}$, but not of $r^{(j)}-r^{(i)}$ or $N$.
   That is, there can be arbitrarily many concentric discs and the  boundaries $\d D^{(j)}$, $j=1,...,N$ can be arbitrarily close. 
   
   The smallness condition \eqref{u0:cond,ball} is the smallness condition \eqref{u0:cond} for the  viscosity patch-type problem \eqref{vpatch-init} when   $\mu_0^+>0$ is a positive constant function, $\mu_0^-=1$ and $D=B_{r^{(1)}}$ is the disc with radius $r^{(1)}$ and with center at the origin.
\end{itemize} 

The density layer  problem for the density-dependent Navier-Stokes equations \eqref{NS} can be formulated similarly. We omit details here.

   \item The main observation that allows us to apply the methods used to study the system \eqref{muNS} to  the Boussinesq system  \eqref{eq:theta} and the density-dependent case \eqref{NS} is the validity of the corresponding $H^1(\R^2)$-energy estimates, which imply the $\|a\|_{W^{1,2+\epsilon}(\R^2)}$-estimate and finally the  $\|\nabla u\|_{L^\infty(\R^2)}$-estimate follows.
   \begin{itemize}
       \item For the Boussinesq equations the $H^1(\R^2)$-energy estimates hold for $a_\vartheta$, which is $a=R_\mu\omega$ corrected by $\mathcal{R}_{-1}\vartheta$ due to the additional buoyancy force $\vartheta e_2$ in \eqref{eq:theta}.
      As there is no regularity assumption on $\vartheta$, we do not have $H^1(\R^2)$-energy estimates for $a$ in this case.
The bound \eqref{B:cond} is inspired by the invariance of the quantities
\begin{align*}
    t^{\frac32-\frac1q}\|\vartheta_0\|_{L^q(\R^2)},\quad 
    \|u_0\|_{L^2(\R^2)},
    \quad t^{\frac12}\Bigl(\|\nabla u_0\|_{L^2(\R^2)}+\Vert (\nabla \taub_0, \d_{\taub_0}\mu_0) \Vert_{L^{2+\epsilon}(\R^2)}^\frac{2+\epsilon}{\epsilon}\Bigr),
\end{align*}
under the scaling
\begin{align*} (\vartheta_{\lambda}, u_{\lambda})(t,x)=(\lambda^{-3}\vartheta, \lambda^{-1}u) (\lambda^{-2}t, \lambda^{-1}x),\quad \lambda>0.
\end{align*}

      \item For the density-dependent case \eqref{NS} the $H^1(\R^2)$-energy estimates hold for the material derivative $\frac{D}{Dt} u=(\d_t +u\cdot\nabla)u$. 
      As there is no regularity assumption on $\rho$, we do not have $H^1(\R^2)$-energy estimates for $a$, which is related to $\frac{D}{Dt} u$ by $\nabla^\perp a=\mathbb{P}(\rho\frac{D}{Dt} u)$, with $\mathbb{P}$ denoting the Leray-Helmholtz projection on the divergence-free vector fields. 
      The left hand side of \eqref{NSu0:cond} is   invariant under the scaling 
      $$(\rho_{0,\lambda}, u_{0,\lambda})(x)=(\rho_0, \lambda^{-1}u_0)(\lambda^{-1}x),\quad \lambda>0.$$
   \end{itemize}  

\end{enumerate}
        

\end{remark}

To conclude this subsection, we review very briefly the progress in the   analysis  of the vortex-patch problem and the density-patch problem in   fluid mechanics:
\begin{itemize}
    \item Vortex-patch problem for the (classical) incompressible Euler equations with the initial vorticity $\omega_0=1_{D_0}$.
    
    J.-Y. Chemin's celebrated works \cite{chemin1998perfect,chemin1993persistance} confirm the regularity propagation of the domain boundary $\d D_0$ for all time, by use of a nondegenerate family of vector fields.
    See also A. L. Bertozzi and P. Constantin's work \cite{bertozzi1993global} from a more geometric viewpoint. Their strategy was also used recently to solve the regularity propagation of temperature-fronts for the Boussinesq equations \eqref{Boussinesq} in \cite{chae2022global}.
A thorough review of results on the two-dimensional vortex-patch problem can be found in \cite{gerard1992resultats}. See also \cite{gamblin1995three} for the problem in three space dimensions and \cite{fanelli2012conservation} for the inhomogeneous case. 
    \item Density-patch problem for the inhomogeneous Navier-Stokes equations with the initial density $\rho_0=1_{D_0}$. 

    In the case of constant viscosity coefficient $\mu=\nu>0$ and in the absence of vacuum with $\rho_0=\rho^+ 1_{D_0}+1_{D_0^C}$, $\rho^+>0$, it was proven by the first author and P. Zhang \cite{liao2019global, liao2016global} that the $W^{k+2,p}$-regularity of the interface $\d D_0$ is propagated throughout time, $k\in\N$, $p\in (2,4)$. A similar result was obtained by F. Gancedo and E. Garcia-Juarez in \cite{gancedo2018global} using bootstrapping arguments. 
The density-patch problem in a bounded domain was solved by R. Danchin and P. B. Mucha in \cite{danchin2019incompressible}. Specifically, they showed that the $C^{1,\alpha}$-regularity of the fluid-vacuum interface is preserved over time ($\alpha \in (0,1)$ in dimension two and $\alpha \in (0,\frac12)$ in dimension three). Very recently, an analogous result for the density-patch problem in $\R^2$ was obtained by T. Hao et al. \cite{hao2024density}. See also the earlier works \cite{danchin2017persistence, liao2016global} for a small density jump and \cite{liu2019global} for the three-dimensional case.

If $\mu$ is variable but close to a positive constant \eqref{small-visc} and the density is bounded away from zero, then global-in-time results were successfully obtained: M. Paicu and P. Zhang \cite{paicu2020striated} proved the propagation of $H^\frac52$-regularity, and F. Gancedo and E. Garcia-Juarez \cite{gancedo2023global} the propagation of $C^{1,\alpha}$-regularity, $\alpha \in (0,1)$, both in two space dimensions. 
\end{itemize}  
To the best of the authors' knowledge, the density patch problem for (\ref{NS}) with general viscosity which might have \textit{large jumps} was not addressed in the literature before. 

\subsection{Proof ideas for the global-in-time a priori estimates}\label{subs:proof}

We prove the global-in-time a priori estimates for \eqref{muNS} in three steps:
\begin{itemize}
    \item Step I. $L^2(\R^2)$-energy estimates for $u$ and    $H^1(\R^2)$-energy estimates for $a$ in terms of $ \nabla u\in {L^\infty(\R^2)}$;
    \item Step II. Time-independent Lipschitz estimate for $u$ in terms of $a\in W^{1,2+\epsilon}(\R^2)$, $ \nabla \taub, \d_{\taub}\mu, \nabla \d_{\taub} u\in L^{2+\epsilon}(\R^2)$;
    \item Step III. $L^1_t\textrm{Lip}(\R^2)$-bound for $u$ and the conclusion of  $H^1(\R^2)$-energy estimates for $a$.
\end{itemize}
In the following we explain the main ideas.

\subsubsection{Step I. (Time-weighted) energy estimates}\label{subss:energy}

Smooth solutions of the density-dependent Navier-Stokes equations (\ref{NS}) in $d$ space dimensions, $d\geq 2$, come with the following energy balance
\begin{align}\label{energy}
    \int_{\R^d} \rho \vert u \vert^2 dx + \int_0^t \int_{\R^d} \mu \vert Su \vert^2 dxdt' = \int_{\R^d} \frac{\vert m_0 \vert^2}{\rho_0} dx.
\end{align}
In the above, $m_0$ denotes the initial momentum of the fluid. Based on this energy balance, P.-L. Lions \cite{lions1996mathematical} proved the global in time existence of weak solutions to (\ref{NS}) with finite energy in any space dimension $d\geq 2$. The uniqueness and regularity of such weak solutions are still open questions even in two space dimensions. 
Under the additional assumption that the viscosity jump is sufficiently small \eqref{small-visc} and the initial velocity belongs to $H^1(\mathbb{T}^2)$, B. Desjardins \cite{desjardins1997regularity} proved that the global weak solution $(\rho,u,\nabla \pi)$ of \cite{lions1996mathematical} on the two-dimensional torus $\mathbb{T}^2$ satisfies $u\in L^\infty_{\mathrm{loc}}([0,\infty);H^1(\mathbb{T}^2))$. 
With additional regularity assumptions on the initial data he could also establish $u\in L^2([0,T_\ast];H^2(\mathbb{T}^2))$ for some short time $T_\ast$. However, these regularity results still do not give an answer to the uniqueness and regularity question.

In the same spirit, for the Navier-Stokes equations with freely transported viscosity coefficient (\ref{muNS}) we aim to establish
\begin{itemize}
    \item an energy balance similar to (\ref{energy}) as well as its time weighted version for  $\|(u, 
    {t'}^{(\frac12)_-}u)\|_{L^\infty_tL^2\cap L^2_t \dot H^1}$, by use of the initial data $u_0\in L^2\cap \dot H^{-1}(\R^2)$, $\mu_0-1\in L^2(\R^2)$;
    \item an $L^2$-estimate as well as its time weighted version for $\|(a,  {t'}^{\frac12}a, {t'}^{1_-}a)\|_{L^\infty_t L^2\cap L^2_t  \dot H^1 }$ in terms of $V(t):= \exp ( C\|\nabla u\|_{L^1_tL^\infty}  )$ 
    and the initial data $u_0\in H^1\cap \dot H^{-1}(\R^2)$, 
    based on the vorticity equation (\ref{intro:omega});
    \item a time-weighted $\dot H^1$-estimate  for $\|  t'^{\frac12}\nabla a \|_{L^\infty_t L^2\cap L^2_t  \dot H^1 }$ in terms of $V(t)$,  $\|t'^{\frac12}\nabla u\|_{L^2_tL^\infty}$ and the initial data $u_0\in \dot H^1(\R^2)$,
    based on the vorticity equation (\ref{intro:omega}).
\end{itemize}
The time-weighted estimate $\| {t'}^{(\frac12)_-}u \|_{L^\infty_tL^2\cap L^2_t \dot H^1}$  has   been established for the density-dependent Navier-Stokes equations (\ref{NS}) in e.g. \cite{abidi2015global2D, wiegner1987decay}; see also \cite{abidi2015global} for the three-dimensional case.
Roughly speaking, the strong decay assumption in the low frequency part $u_0\in \dot H^{-1}(\R^2)$ implies  stronger decay in time of the solution  $u$. A similar consideration applies to the time-weighted estimates for $a$.
Compared with the derivation of  classical energy estimates for $u$, due to the non-local representation of $a=R_\mu \omega$ (recalling \eqref{Rmu:def}) in terms of $\omega$, we have to make use of  commutator estimates for the Riesz transform, as well as the commutation relation $[\mu, \frac{D}{Dt}]=0$, that is the transport equation $\frac{D}{Dt}\mu=0$, when deriving energy estimates for $a$.
Notice that  in the energy estimates for $a$  we simply use the Lipschitz norm of the   velocity field $\|\nabla u\|_{L^1_tL^\infty}$ and $\|t'^{\frac12}\nabla u\|_{L^2_tL^\infty}$, instead of the classical $\|\nabla u\|_{L^4_tL^4}$-norm (see e.g. \cite{paicu2020striated}). 
Indeed, although a priori the initial lower and upper bounds $\mu_\ast, \mu^\ast$ for $\mu_0$ are transported by the free transport equation $\d_t\mu+u\cdot\nabla\mu=0$ as in \eqref{mu:bound}:
$$\mu_\ast\leq\mu(t,x)\leq\mu^\ast,$$
we can not control $\|\omega\|_{L^4(\R^2)}$ or $\|\nabla u\|_{L^4(\R^2)}$ by $\Vert a \Vert_{H^1(\R^2)}$ by use of $a=R_\mu\omega$  with only positive bounded $\mu$, unless we have regularity or small variation assumptions on $\mu$ (recalling \eqref{failure:Lp}).
See more discussions in Step II below.

The energy estimates for $a$ are not yet closed, and we discuss in Step II the (time-independent) Lipschitz  estimate for $u$ in terms of $\| a\|_{ W^{1,2+\epsilon}(\R^2)}$ and the tangential regularity. Finally, a bootstrap argument concludes the global-in-time estimates in Step III.

\subsubsection{Step II. The time-independent Lipschitz estimate}

It is well-known that for   evolution equations arising in fluid mechanics, the $L^1_t\mathrm{Lip}(\R^2)$-regularity of the fluid velocity is crucial for regularity theory. In order to obtain such an estimate we begin by establishing a \textit{time-independent} Lipschitz estimate for the velocity vector field, which is  key step.  

The main obstacle to derive the desired Lipschitz estimate is that one can not bound $\Vert \nabla u \Vert_{L^\infty(\R^2)}$ by $\Vert a \Vert_{H^2(\R^2)}$ (from the energy estimates in Step I) directly, and even worse, we can not control $\|\nabla u\|_{L^4(\R^2)}$  a priori by $\Vert a \Vert_{H^1(\R^2)}$ or $\Vert a \Vert_{L^4(\R^2)}$,  provided with   the a priori bound   $\mu_\ast\leq\mu(t,x)\leq\mu^\ast$, as mentioned above.  

Recall that the velocity gradient $\nabla u=\nabla \nabla^\perp\phi$ is related to $a$ by \eqref{Lmu:def}:
\begin{align}\label{Lmu,a}
    L_\mu \phi = \Delta a, \hbox{ with } L_\mu= (\d_{22}-\d_{11})\mu (\d_{22}-\d_{11})+(2\d_{12})\mu(2\d_{12}),
\end{align}
where $L_\mu$ is a fourth-order elliptic operator \eqref{ellipticity},
or equivalently, the velocity gradient $\nabla u=\nabla\nabla^\perp\Delta^{-1}\omega$ with $\omega=\Delta\phi$ denoting the vorticity is related to $a$ by \eqref{Rmu:def}:
\begin{align}\label{Rmu,a}
    \nabla u=R R^\perp\omega,\quad a=R_\mu\omega,
    \hbox{ with }R_\mu=(R_2R_2-R_1R_1)\mu(R_2R_2-R_1R_1) + (2R_1R_2)\mu (2R_1R_2). 
\end{align}
Given the failure of the $L^p(\R^2)$-estimate \eqref{failure:Lp}, we impose a certain tangential regularity assumption \eqref{TR:mu0} on the initial  viscosity $\mu_0$ with respect to some vector field $\tau_0$, aiming to obtain the Lipschitz estimate for the velocity by exploiting ellipticity and tangential regularity.   
Note that the discontinuity of $\mu$ in the normal direction $\tau_0^\perp$ is allowed.

In the past twenty years significant developments have been made in the study of elliptic and parabolic systems with rough coefficients, see e.g. the book \cite{krylov2008lectures}. 
H. Dong and D. Kim established in \cite{dong2011parabolic} $L^p$-estimates for solutions of higher order elliptic and parabolic systems 
  with so-called variably partially BMO coefficients, which in particular includes discontinuous coefficients which may have jumps in one direction and are continuous in the other directions. 
  Roughly speaking, this means that for every localized cylinder  there exists a local coordinate system such that the coefficients $\mu(y', y_d)$ are BMO with respect to the first $d-1$ components $y'\in \R^{d-1}$, while only measurable and bounded in the last component $y_d\in \R$. 
  This partial regularity in the coefficients implies then the regularity of the solution in $y'$, and finally the ellipticity (or parabolicity) of the equation   allows one to recover the desired regularity of the solution in $y_d$ as well.

Observe that functions with tangential regularity, e.g. the initial data $\mu_0$ given in Theorem \ref{exthm}, fall into Dong-Kim's coefficient category.
Indeed, for the \textit{stationary} Navier-stokes equation, it was shown by use of Dong-Kim's results in \cite{he2020solvability}  that on a bounded $C^{1,1}$-domain $\Omega \subset \R^2$, given a weak solution $(\rho,u)\in L^\infty(\Omega;[0,\infty)) \times H^1( \Omega)$ satisfying appropriate boundary conditions 
and provided the coefficient $\mu$ has tangential regularity, we have
\begin{align*}
    \nabla u \in L^p(\Omega) \quad \text{for any } p\in (1,\infty) ,
\end{align*}
(note that   $p=\infty$ can not be achieved by Dong-Kim's results).
Unfortunately, Dong-Kim's estimates for $L_\mu\phi=\Delta a$ can not give the explicit dependence on the tangential regularity of the coefficient $\mu$, which is extremely important for us since the tangential regularity also evolves in time and should be tracked. We follow the essential idea to separate the ``good" and ``bad" directions, but in a more transparent way,  below.

\begin{lemma}[Decomposition of $L_\mu$ in tangential and normal directions in terms of ``good unknown'' $\alpha$]\label{lemma:tau,n}
Let $\tau=\begin{pmatrix}
    \tau_1\\ \tau_2
\end{pmatrix}(x)$ 
be a regular nondegenerate vector field such that
\begin{align}\label{tau}
    \tau\in L^\infty(\R^2;\R^2), \quad \nabla\tau\in L^p(\R^2;\R^{2\times2}),\hbox{ for some }p\in (2,\infty),\quad \frac{1}{|\tau|}\in L^\infty(\R^2).
\end{align}  
We introduce correspondingly
\begin{itemize}
    \item The unit tangential and normal vectors
\begin{equation}
    \overline{\tau}=\frac{\tau}{\vert\tau\vert}=\begin{pmatrix}
         \frac{\tau_1}{|\tau|}\\ \frac{\tau_2}{|\tau|}
    \end{pmatrix}=:\begin{pmatrix}
        \taub_1\\ \taub_2
    \end{pmatrix},
    \quad n=\overline{\tau}^\perp=\frac{\tau^\perp}{\vert \tau \vert}=\begin{pmatrix}
        -\frac{\tau_2}{|\tau|}\\ \frac{\tau_1}{|\tau|}
    \end{pmatrix}=\begin{pmatrix}
        -\taub_2\\ \taub_1
    \end{pmatrix},
\end{equation}
and their tensor products
\begin{align*}
    \overline{\tau}\otimes \overline{\tau} = \begin{pmatrix}
       \taub_1^2 & \taub_1\taub_2 \\ \taub_1\taub_2 & \taub_2^2
     \end{pmatrix}, \quad 
    n\otimes n =   
    \begin{pmatrix}
        \taub_2^2 & -\taub_1\taub_2 \\ -\taub_1\taub_2 & \taub_1^2
    \end{pmatrix}, \quad 
    \overline{\tau}\otimes n = (n\otimes \overline{\tau})^T = 
    \begin{pmatrix}
        -\taub_1\taub_2 & \taub_1^2 \\ -\taub_2^2 & \taub_1\taub_2
    \end{pmatrix} .
\end{align*}
\item The associated directional differential operators 
\begin{align}\label{dtaun:def}
    \d_{\overline{\tau}}=\overline{\tau}\cdot \nabla, \quad  \d_n=n\cdot \nabla, 
\end{align}
and their adjoint operators 
\begin{align}\label{dtaun:def,ast}
    \d_{\overline{\tau}}^\ast = -\div \overline{\tau}  , \quad  \d_n^\ast=-\div n, 
\end{align}
where the operator $\div v$ is understood as $\div v (f)=\div(v f)=\sum_{j=1}^2\d_{j}(v_j f)$, for $v=\overline{\tau}$, $n$.
\end{itemize} 
Then the following formulas hold
\begin{enumerate}
    \item 
    \begin{enumerate}
        \item \label{nabla:tau,n} $\nabla=\overline{\tau}\d_{\overline{\tau}}+n\d_n=-\d_{\overline{\tau}}^\ast \, \overline{\tau}  -\d_n^\ast\, n$  and $\nabla^\perp=n\d_{\overline{\tau}}-\overline{\tau}\d_n=-\d_{\overline{\tau}}^\ast \, n  +\d_n^\ast\,  \overline{\tau}$.
        More precisely,  
        \begin{align}\label{dtaun:formula}
    \begin{array}{ll}
        \d_{1} = \taub_1 \dtau - \taub_2\d_n & =-\dtau^\ast(\taub_1\cdot) + \d_n^\ast(\taub_2\cdot), \\
        \d_{2} = \taub_2\dtau +\taub_1 \d_n & = -\dtau^\ast(\taub_2\cdot) - \d_n^\ast(\taub_1\cdot),
    \end{array}
\end{align}
        \item \label{nLaplace} 
        $\Delta=\nabla\cdot\nabla=-\d_{\overline{\tau}}^\ast \, \d_{\overline{\tau}} -\d_n^\ast\,  \d_n $ and $n\Delta =\dtau \nabla^\perp + \d_n\nabla,$
        \item \label{nabpnab} $\nabla^\perp\otimes\nabla =-\d_{\overline{\tau}}^\ast(n\otimes \overline{\tau} )\d_{\overline{\tau}} -\d_{\overline{\tau}}^\ast(n \otimes n )\d_n + \d_n^\ast( \overline{\tau} \otimes  \overline{\tau})\d_{\overline{\tau}} +\d_n^\ast(\overline{\tau}  \otimes n)\d_n.$ 
    \end{enumerate}
\item Let $\mu\in L^\infty(\R^2)$, and denote the operator
$$L_\mu \phi  = (\nabla^\perp \otimes \nabla) : (\mu S\nabla^\perp\phi),
\hbox{ with }S\nabla^\perp\phi=\nabla\nabla^\perp\phi+(\nabla\nabla^\perp\phi)^T=\begin{pmatrix}
        -2\d_{12}\phi & (\d_{11}-\d_{22})\phi\\
        (\d_{11}-\d_{22})\phi & 2\d_{12}\phi
    \end{pmatrix}.$$ 
    \begin{itemize}
   \item  We can reformulate the operator $L_\mu$ as follows
    \begin{align}\label{Lmu:taun}
    \begin{split}
        L_\mu \phi &= -\dtau^\ast\brl(\taub_2^2-\taub_1^2)\dtau\omega_1\brr - \dtau^\ast \brl 2\taub_1\taub_2 \d_n \omega_1\brr  - \d_n^\ast\brl2\taub_1\taub_2 \dtau\omega_1\brr + \d_n^\ast\brl(\taub_2^2-\taub_1^2)\d_n \omega_1\brr \\
        &\quad - \dtau^\ast \brl2\taub_1\taub_2 \dtau\omega_2\brr + \dtau^\ast \brl(\taub_2^2-\taub_1^2) \d_n\omega_2\brr  + \d_n^\ast\brl(\taub_2^2-\taub_1^2)\dtau\omega_2\brr + \d_n^\ast\brl2\taub_1\taub_2 \d_n\omega_2\brr,
    \end{split}
    \end{align}
    where we denote
    \begin{equation*}
        \omega_1=\mu(\d_{22}-\d_{11})\phi,
        \quad \omega_2=\mu 2\d_{12}\phi,
        \quad \hbox{ such that }\quad \mu S\nabla^\perp\phi=\begin{pmatrix}
            -\omega_2& -\omega_1\\
            -\omega_1&\omega_2
        \end{pmatrix}.
    \end{equation*}

  \item  We can furthermore decompose $L_\mu\phi$ into  
\begin{align}\label{alpha:eq}
\begin{split}
    L_\mu\phi  =& \d_n^\ast\d_n \alpha +L_\mu^\tau\phi,
\end{split}
\end{align}
where
\begin{align*}
    L_\mu^\tau\phi
    =&- \dtau^\ast \brl(\taub_2^2-\taub_1^2) \dtau\omega_1+2\taub_1\taub_2 \dtau\omega_2\brr - 2\d_n^\ast\brl2\taub_1\taub_2 \dtau\omega_1  -   (\taub_2^2-\taub_1^2) \dtau \omega_2\brr \\
    & - \d_1\brl \d_2(2\taub_1\taub_2)\omega_1-\d_2 (\taub_2^2-\taub_1^2)\omega_2 \brr 
    + \d_2\brl \d_1(2\taub_1\taub_2)\omega_1-\d_1(\taub_2^2-\taub_1^2)\omega_2\brr   \\
    & - \d_n^\ast \brl\d_n(\taub_2^2-\taub_1^2)\omega_1 + \d_n(2\taub_1\taub_2) \omega_2\brr \\
    =& \nabla\cdot \Bigl( \bigl(\taub(\taub_2^2-\taub_1^2)+2n(2\taub_1\taub_2) \bigr)\dtau \omega_1
    + \bigl(\taub 2\taub_1\taub_2 +2n(\taub_2^2-\taub_1^2) \bigr)\dtau \omega_2\Bigr) \\
    &+\nabla^\perp\cdot\Bigl(- \omega_1\nabla (2\taub_1\taub_2) +\omega_2\nabla(\taub_2^2-\taub_1^2)\Bigr) 
    +\nabla\cdot\brl\omega_1  \,\d_n(\taub_2^2-\taub_1^2) n + \omega_2\, \d_n(2\taub_1\taub_2) n\brr.
\end{align*}
In the above,  we denote
$$\alpha=(\taub_2^2-\taub_1^2) \omega_1 +2\taub_1\taub_2 \omega_2 =(\taub_2^2-\taub_1^2) \mu(\d_{22}-\d_{11})\phi+2\taub_1\taub_2 \mu(2\d_{12})\phi,$$
that is, 
\begin{align}\label{alpha:def}
    \alpha = (\overline{\tau} \otimes n): (\mu S\nabla^\perp\phi),
\end{align}
  or equivalently, \begin{align}\label{alpha:reform0} 
 \alpha  = -\mu\Delta \phi + 2\mu (\overline{\tau}\cdot \dtau \nabla \phi),
\end{align} 
which implies the relation between $\d_n\nabla^\perp\phi$ and $\alpha$ below (if $\mu\neq0$)  \begin{align}\label{alpha:reform} 
 \d_{n}\nabla^\perp \phi =  \frac{\alpha}{\mu}   \overline{\tau}  - 2 (\overline{\tau} \cdot \dtau  \nabla \phi)   \overline{\tau}  +  \dtau \nabla  \phi .
\end{align} 
 
\item Define $a$ as in \eqref{Lmu:def}: $\Delta a=L_\mu\phi$, then we have the following relation
\begin{align}\label{eq:a,alpha}
    \nabla (a+\alpha)
    =&R R\cdot\Bigl( \taub\dtau \alpha + \bigl(\taub(\taub_2^2-\taub_1^2)+2n(2\taub_2\taub_2) \bigr)\dtau \omega_1
    + \bigl(2\taub_2\taub_2 +2n(\taub_2^2-\taub_1^2) \bigr)\dtau \omega_2\Bigr) \\
    &+R R^\perp\cdot\Bigl(- \omega_1\nabla(2\taub_2\taub_2)+\omega_2\nabla(\taub_2^2-\taub_1^2)\Bigr) 
    +RR\cdot\brl \omega_1  \,\d_n(\taub_2^2-\taub_1^2) n + \omega_2\, \d_n(2\taub_2\taub_2) n\brr,\nonumber
\end{align}
where $R=\frac{\frac1i\nabla}{\sqrt{-\Delta}}$ denotes the Riesz operator.
Here the equality can be  understood in $L^p(\R^2)$ if $a,\alpha\in W^{1,p}(\R^2)$, $\dtau\mu, \dtau\nabla^2\phi, \nabla\taub\in L^p(\R^2)$ and $\mu, \nabla^2\phi\in L^\infty(\R^2)$. 
\end{itemize}
\end{enumerate}
\end{lemma}
The formulas \eqref{nabla:tau,n}-\eqref{nabpnab} in the first statement follow from straightforward calculations.
The formula \eqref{Lmu:taun} follows from \eqref{nabpnab} directly.
We derive \eqref{alpha:eq} from \eqref{Lmu:taun}, by applying the following commutator identities (with appropriately chosen $f,g$)  to \eqref{Lmu:taun}:
\begin{align*}
    &\dtau^\ast ( f\d_n g)-\d_n^\ast(f\dtau g)=-\d_1(f\d_2g)+\d_2(f\d_1g)=\d_1((\d_2 f) g)-\d_2((\d_1 f) g),\\
    &\d_n^\ast (f\d_n g)-\d_n^\ast\d_n(fg)=-\d_n^\ast((\d_n f)g).
\end{align*}
     The reformulation \eqref{alpha:reform0} follows from \eqref{alpha:def} by direct computation and the relation \eqref{alpha:reform} follows from   \eqref{alpha:reform0} and \eqref{nLaplace}.
     Finally, \eqref{nLaplace} and  \eqref{alpha:eq} implies $\Delta a=L_\mu\phi=-\Delta \alpha-\dtau^\ast\dtau\alpha+L_\mu^\tau\phi$ and hence \eqref{eq:a,alpha} follows.
This completes the proof of Lemma \ref{lemma:tau,n}.
\smallbreak 
\begin{figure}
\centering
\begin{tikzpicture}
    \node[] at (4,0) (alpha) {$\alpha$};
    \node[] at (0,0) (a) {$a=R_\mu \omega$};
    \node[] at (0,2) (gradu) {$\nabla u=\nabla\nabla^\perp\phi=RR^\perp\omega$};
    \node[] at (4,2) (dnu) {$\d_n u=\d_n\nabla^\perp\phi$};
    \draw[->] (a.east) -- node [below] {
    Step 3} (alpha.west);
    \draw[->] (dnu.west) -- node [above] {Step 1} (gradu.east) ;
    \draw[->] (alpha.north) -- node [right] {Step 2 } (dnu.south);
    \draw[dashed,->] (a.north) -- node [left] {$L^\infty$} (gradu.south);
\end{tikzpicture}
\caption{Idea of the proof of Proposition \ref{u:Lip:prop}. }
\label{diagram}
\end{figure}
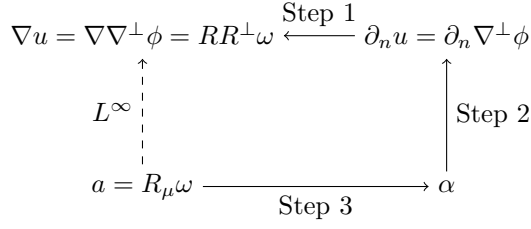

Making use of Lemma \ref{lemma:tau,n} we can derive   the following $L^\infty(\R^2)$-bound for $\nabla u=\nabla\nabla^\perp\phi$ in terms of $\omega=\Delta\phi, a=R_\mu\omega$ and the tangential regularity (see Proposition \ref{u:Lip:prop} below):
\begin{align}\label{intro:u,Lip}
    \Vert \nabla u \Vert_{L^\infty(\R^2)} \leq C(p) \Vert \omega \Vert_{L^p(\R^2)}^{1-\frac2p} \Bigl(\Vert \nabla a \Vert_{L^p(\R^2)} +  \Vert (\nabla \taub,  \dtau \mu)\Vert_{L^p(\R^2)} \Vert \nabla u \Vert_{L^\infty(\R^2)} 
     + \Vert \dtau \omega \Vert_{L^p(\R^2)}
    \Bigr)^ \frac{2}{p}.
\end{align}
To prove \eqref{intro:u,Lip} we start with the bound for the ``good" direction   in terms of the tangential regularity:
$$\|\dtau\nabla^2\phi\|_{L^p}, \|\nabla\dtau\nabla\phi\|_{L^p},
\|\dtau\nabla u\|_{L^p},
\|\nabla\dtau u\|_{L^p}\lesssim \Vert \nabla\taub \Vert_{L^\infty  } \Vert \nabla u \Vert_{L^\infty} + \Vert \dtau  \omega \Vert_{L^p},$$
 by use of commutator estimates.
Now, with the relations  \eqref{alpha:reform}, \eqref{eq:a,alpha} between $\alpha$ and  $\d_n u$,   $a$ respectively, we can derive the Lipschitz estimate for the velocity $u=\nabla^\perp\phi$ following the steps illustrated in  Figure \ref{diagram} (it is not possible to control $\nabla u$ by $a$ in $L^\infty$ directly):  
\begin{itemize}
         \item Step 1. It remains to control $\|\d_nu\|_{L^\infty}$, since the control in the ``good" direction $\|\dtau u\|_{L^\infty}$ follows from the interpolation between $\|\dtau u\|_{L^p}\sim \|\omega\|_{L^p}$ and $\|\nabla\dtau u\|_{L^p}$, which is the righthand side of \eqref{intro:u,Lip}.
         \item Step 2. It remains to control $\|\alpha\|_{L^\infty}$, by view of the expression \eqref{alpha:reform} of $\d_n u$ in terms of $\alpha$ and $\dtau u$.
         \item Step 3. The control on $\|\alpha\|_{L^\infty}$ follows from $\nabla a\in L^p$ and the   tangential regularity  by \eqref{eq:a,alpha}.
     \end{itemize}  
   
We  later take   $p=2+\epsilon$ (see Corollary \ref{coro} below),   with $\epsilon>0$ given in Lemma \ref{Rmu-inv:prop}, since we have to estimate the $L^{2+\epsilon}$-norm of $\omega, \dtau\omega$ in \eqref{intro:u,Lip} in terms of $a, \dtau a$, respectively,  
where the boundedness of $R_\mu^{-1}$ in $L^{2+\epsilon}$ is used.

We remark that although one can simply perform Young's inequality in \eqref{intro:u,Lip} to get a uniform bound for $\|\nabla u\|_{L^\infty(\R^2)}$, we don't do so since $\|\nabla \taub\|_{L^p(\R^2)}$ grows exponentially in (the time integration of) $\|\nabla u\|_{L^\infty(\R^2)}$ (recalling \eqref{tau:exp}).
Instead, we use the smallness assumption \eqref{u0:cond} to close the bootstrap argument in Step III.



\subsubsection{Step III. The $L^1_t\textrm{Lip}(\R^2)$-estimate}

After establishing the time-independent Lipschitz estimate for the velocity \eqref{intro:u,Lip}, we conclude the uniform-in-time bound for $\|\nabla u\|_{L^1_tL^\infty_x}$ by a bootstrap argument.

Recall 
\begin{itemize} 
    \item Time-weighted energy estimates for $u$ and $a$ from Step I, which imply the estimates for $\|a\|_{L^1_tW^{1,2+\epsilon}}$ and $\|t'^{\frac12}a\|_{L^2_t W^{1,2+\epsilon}}$ in terms of $\|\nabla u\|_{L^1_tL^\infty}$ and $\|{t'}^{\frac12}\nabla u\|_{L^2_tL^\infty}$;
    \item Time-independent Lipschitz estimate \eqref{intro:u,Lip} from Step II.
    \item $L^{2+\epsilon}$-estimate for $\nabla\tau(t,\cdot)$ in \eqref{tau:exp}, which depends linearly on $\|\nabla\d_\tau u\|_{L^1_t L^{2+\epsilon}}$ (which can be bounded by $\|\nabla a\|_{L^1_t L^{2+\epsilon}}$ up to $\int^t_0 \|\nabla\tau\|_{L^{2+\epsilon}}\|(\nabla u, \nabla a)\|_{L^\infty} \,dt'$) and  exponentially   on $\Vert \nabla u \Vert_{L^1_tL^\infty}$. 
\end{itemize}
%
In order to close the estimate for the scaling-invariant quantity $\Vert \nabla u \Vert_{L^1_tL^\infty}$, we make use of the scaling-invariant smallness condition \eqref{u0:cond}.
However, since  the norms $\|u_0\|_{L^2}$ and $\|u_0\|_{\dot H^{-1}}$, which appear both in the time-weighted estimate for $\|(t')^{(\frac12)_-}u\|_{L^\infty_t L^2}$, do not share the same scaling, it turns out to be more convenient to consider directly the rescaled solution.
See Subsection \ref{sect:L1Lip} for more details.

\bigbreak

\noindent\textbf{Organization of the paper.} 
The remainder of this paper is structured as follows.

In Section \ref{sect:aprioripf} we first establish the a priori estimates mentioned in Subsection \ref{subs:proof} step by step, and afterwards we prove Theorem \ref{exthm} and Corollary \ref{propthm}. 

The proof of Lemma \ref{Rmu-inv:prop} is given in Appendix \ref{sect:intro-proofs}. 
Some commutator estimates involving $L^\infty$-norms are proved in   Appendix \ref{sect:proof,comm}.
Finally in Appendix \ref{sect:u-en}  we show  the decay estimates for the fluid velocity.

\bigbreak

\noindent\textbf{Notation. }
Throughout this article we denote by $\dot f\equiv \frac{D}{Dt}f=\d_tf+u\cdot\nabla f$ the material derivative of a function $f$. For a vector field $\tau$ we write $\d_\tau=\tau\cdot \nabla$ for the directional derivative along $\tau$. For $t>0$ and $p,q\in [1,\infty]$ we denote $L_t^pL^q = L^p([0,t];L^q(\R^2))$ and $L^pL^q = L^p([0,\infty);L^q(\R^2))$. We denote $L^p(\R^2;\R^n)$ simply by $L^p(\R^2)$ for $n\in\N$, if the dimension $n$ is clear from the context, with norm $\Vert\cdot\Vert_{L^p} = \Vert\cdot \Vert_{L^p(\R^2;\R^n)}$. The commutator of two operators $A$ and $B$ is defined as $[A,B]=AB-BA$. Moreover, $\nabla^\perp= \begin{pmatrix}
    -\d_2 \\ \d_1
\end{pmatrix}$ and $v^\perp = \begin{pmatrix}
    -v_2 \\ v_1
\end{pmatrix}$, $v\in\R^2$, indicate a rotation of the vector in the plane by ninety degrees. 
We denote the exponential growth in the time integration of the velocity gradient by 
\begin{align}\label{V:def}
    V(t):=\exp\Bigl(C \Vert \nabla u  \Vert_{L^1_tL^\infty} \Bigr),
    \hbox{ and } \tilde V(t):=V(t)\exp\Bigl(C  \Vert t'^{\frac12}\nabla u  \Vert_{L^2_tL^\infty}  \Bigr).
\end{align}
Here and in what follows $C$ denotes some positive constant, which may depend only on $\mu_\ast, \mu^\ast$ and may vary from line to line.
Lastly, we denote $\langle t \rangle = e+t$ for times $t\in [0,\infty)$.

\section{Proofs}\label{sect:aprioripf}

The goal of this section is to prove Theorem \ref{exthm} and Corollary \ref{propthm}. 
To this end, we first establish  a priori estimates  in a series of propositions in Subsections \ref{sect:energy}, \ref{sect:Lip} and \ref{sect:L1Lip}. 
The proofs of Theorem \ref{exthm} and Corollary \ref{propthm} are found in Subsection \ref{sect:proofs} and Subsection \ref{sect:d-prop}, respectively.

We are going to use frequently the following well-known interpolation inequalities, see  e.g. \cite{bahouri2011fourier}.

\begin{lemma}  
If $g\in H^1(\R^2)\cap W^{1,r}(\R^2)$, $r\in (2,\infty)$, 
then
\begin{align}
    \Vert g\Vert_{L^r(\R^2)} &\lesssim_r \Vert g \Vert_{L^2(\R^2)}^\frac2r \Vert \nabla g \Vert_{L^2(\R^2)}^{1-\frac2r}, \label{ip-lr22}\\
    \Vert g \Vert_{L^\infty(\R^2)} &\lesssim_r \Vert g \Vert_{L^r(\R^2)}^{1-\frac2r} \Vert \nabla g \Vert_{L^r(\R^2)}^\frac2r.\label{ip-linfty}
\end{align}
\end{lemma} 

Let us recall some classical commutator estimates for the Riesz transform. 
\begin{lemma}\label{commutator:lem}
Let $R=\frac{\frac1i\nabla}{\sqrt{-\Delta}}$ denote the Riesz transform on $\R^2$. 
\begin{enumerate}
    \item For $p,p_1\in (1,\infty)$ and $p_2\in [1,\infty]$ satisfying
$  \frac{1}{p_1}+\frac{1}{p_2}=\frac{1}{p},
 $  
 we have the following commutator estimate
\begin{align}\label{comm}
   \Vert [R^2,\d_X]g \Vert_{L^p} \lesssim_{p,p_1,p_2} \Vert \nabla X \Vert_{L^{p_2}} \Vert g \Vert_{L^{p_1}},
\end{align}
where $g\in L^{p_1}(\R^2)$ and   $X\in C_c^1(\R^2;\R^2)$.
\item  
For $p\in (2,\infty)$, we have the following commutator estimate
\begin{align}
    \Vert \d_X R^2g \Vert_{L^p} &\lesssim_p \Vert \d_Xg \Vert_{L^p} + \Vert \nabla X \Vert_{L^p} \Vert R^2g \Vert_{L^\infty} ,\label{comm:infty-0} \\
    \Vert \d_X R^2g-R^2\div(Xg) \Vert_{L^p} &\lesssim_p \Vert \nabla X \Vert_{L^p} \Vert R^2g \Vert_{L^\infty},\label{comm:infty-1}  
\end{align}
for any  $g\in C^1_c(\R^2)$ and $X\in C^1_c(\R^2;\R^2)$.
Furthermore, for $\mu \in L^\infty(\R^2)$ with $\|\mu\|_{L^\infty}\leq\mu^\ast$ and $\d_X\mu\in L^q(\R^2)$, $q\in [p,\infty]$, we have  
\begin{align} 
    \Vert [R_\mu,\d_X]g \Vert_{L^p} &\lesssim_{p,\mu^\ast} 
    (\Vert \nabla X \Vert_{L^p}+\|\d_X\mu\|_{L^p}) \bigl(\Vert R^2 g \Vert_{L^\infty} + \Vert R_\mu g \Vert_{L^\infty} \bigr), 
    \hbox{ if }q=p,
    \label{Rmu:comm}\\
    \Vert [R_\mu,\d_X]g \Vert_{L^p} &\lesssim_{p,q,\mu^\ast} \Vert \nabla X \Vert_{L^p} \bigl(\Vert R^2 g \Vert_{L^\infty} + \Vert R_\mu g \Vert_{L^\infty} \bigr) + \Vert \d_X\mu \Vert_{L^q} \Vert g \Vert_{L^\frac{qp}{q-p}},  \hbox{ if } q\in (p,\infty], \label{Rmu:comm+}
\end{align}
where   $R_\mu=(R_2R_2-R_1R_1)\mu(R_2R_2-R_1R_1) + (2R_1R_2)\mu (2R_1R_2)$ is defined in \eqref{Rmu:def}.
\end{enumerate}
\end{lemma}
The proof of the first estimate (\ref{comm}) can be found in A. P. Calder\'on's article \cite[Theorem 1]{calderon1965commutators}. 
The proof of the second statement is very much in the spirit of \cite[Lemma 5.1]{paicu2020striated} and \cite[Lemma 2.10]{danchin2020well}, and is postponed in Appendix \ref{sect:proof,comm}.

\subsection{Step I. (Time-weighted) energy estimates}\label{sect:energy}

We start with some basic energy estimates for \eqref{muNS}. 
These have already been established for the density-dependent Navier-Stokes equations (\ref{NS}) in e.g. \cite{abidi2015global2D, wiegner1987decay}; see also \cite{abidi2015global} for the three-dimensional case. 
Using the same ideas we prove the following estimates  for our system \eqref{muNS} in  Appendix \ref{sect:u-en}.

\begin{proposition}[Energy estimates for $u$]\label{u-decay:prop}
Let $(\mu,u)$ be a sufficiently smooth solution of (\ref{muNS}) on some time interval $[0,T^\ast)$ with $\mu_0-1\in L^2(\R^2)$ and $u_0\in L^2(\R^2) \cap \dot H^{-2\delta}(\R^2)$ for some $\delta \in (0,\frac12)$. Then  the following energy estimates  hold for $t\in [0,T^\ast)$:
\begin{align}
    \Vert u \Vert_{L^\infty_tL^2}+ \Vert \nabla u \Vert_{L^2_tL^2} &
    \leq C(\mu_\ast)\Vert u_0 \Vert_{L^2}, \label{u:energy} \\
    \Vert \langle t \rangle^{\delta_-} u \Vert_{L^2} + \Vert \langle t' \rangle^{\delta_-} \nabla u \Vert_{L^2_tL^2} &\leq C(\mu_\ast, \delta, \delta-\delta_-) (\|u_0\|_{L^2\cap \dot H^{-2\delta}}+ \|\mu_0-1\|_{L^2}\|u_0\|_{L^2}), \label{u:decay+}
\end{align}
where $\delta_->0$ stands for any positive number strictly smaller than $\delta$.
\end{proposition}

We now turn to establishing energy estimates for the quantity $a$ introduced in \eqref{a:R} above. 
With the decomposition (\ref{divmuSu:decomp}) the velocity equation (\ref{muNS})$_2$ becomes
\begin{align*}
    \d_tu+u\cdot\nabla u-\nabla^\perp a + \nabla (\pi-b)=0.
\end{align*} 
where $a=R_\mu\omega$, $b=Q_\mu\omega$ are given in \eqref{Rmu:def}, \eqref{Qmu:def} respectively.
The vorticity equation is obtained by applying the curl operator $\nabla^\perp\cdot$ to the velocity equation:
\begin{align}\label{omega:eq}
\left\{\begin{array}{c}
    \d_t\omega+u\cdot\nabla\omega -\Delta a =0, \quad  (t,x)\in (0,\infty)\times \R^2 , \\
    u=\nabla^\perp\Delta^{-1}\omega,\quad a=R_\mu\omega.
\end{array}\right.
\end{align}
If $\mu$ is smooth, then the vorticity equation (\ref{omega:eq}) is parabolic. However, for more general (discontinuous) viscosities, it is not clear whether the equation has a parabolic character. This is largely because of the non-local operator $R_\mu$, which itself is composed of local and non-local operators.
Nevertheless, we have the following (time-weighted) energy estimates for the vorticity equation (\ref{omega:eq}).

\begin{proposition}[$H^1$-energy estimates for $a$]\label{energy:prop}
Let $\mu \in L^\infty([0,\infty)\times\R^2;[\mu_\ast,\mu^\ast])$ be a positive, bounded function with $0<\mu_\ast\leq \mu^\ast$. Let $u$ be a sufficiently regular divergence-free vector field with vorticity $\omega=\nabla^\perp \cdot u$ satisfying the vorticity equation (\ref{omega:eq}) on some time interval $[0,T^\ast)$. Then for all times $t\in [0,T^\ast)$,
\begin{align}
    \Vert a \Vert_{L^\infty_tL^2}^2 + \Vert \nabla a \Vert_{L^2_tL^2}^2 &\leq C(\mu_\ast,\mu^\ast) \Vert \omega_0 \Vert_{L^2}^2V(t),
    \label{a:energy-low}\\
    \Vert  {t'}^{\frac{1}{2}} a \Vert_{L^\infty_t L^2}^2 + \Vert  {t'}^{\frac{1}{2}}\nabla a \Vert_{L^2_tL^2}^2 &\leq C(\mu_\ast,\mu^\ast) \Vert u_0 \Vert_{L^2}^2 V(t), \label{a:tenergy-low}\\
    \Vert {t'}^{\frac{1}{2}}\nabla a \Vert_{L^\infty_t L^2}^2 + \Vert{t'}^{\frac{1}{2}}\Delta a \Vert_{L^2_tL^2}^2 
    &\leq C(\mu_\ast, \mu^\ast)
    \bigl(\Vert \nabla a \Vert_{L^2_tL^2}^2 
    +\|{t'}^{\frac12}\nabla u\|_{L^2_tL^\infty}^2\Vert  a \Vert_{L^\infty_t L^2}^2\bigr)V(t), \label{a:tenergy-high} 
\end{align}
where $V(t)=\exp(\int^t_0 C\|\nabla u\|_{L^\infty} dt')$ denotes the exponential growth in the time integration of the velocity  gradient. 
Moreover, if we additionally assume that the hypotheses of Proposition \ref{u-decay:prop} are satisfied, then
\begin{align}\label{a:tenergy-low+}
    \Vert   {t'} ^{\frac12+\delta_-}a \Vert_{L^\infty_t L^2}^2 + \Vert   {t'} ^{\frac12+\delta_-} \nabla a \Vert_{L^2_tL^2}^2 \leq C(\mu_\ast,\mu^\ast) \|\langle t'\rangle^{\delta_-}\nabla u\|_{L^2_tL^2}^2\,V(t).
\end{align}
\end{proposition}

\begin{proof}
\begin{itemize} 
\item \textbf{Proof of (\ref{a:energy-low}):}
We  take the $L^2(\R^2)$-inner product between    the $\omega$-equation (\ref{omega:eq}): 
\begin{equation}\label{omega:eq,dot}
    \dot\omega-\Delta a=0, \hbox{ with }\dot \omega:=\frac{D}{Dt}\omega=(\d_t+u\cdot\nabla)\omega,
\end{equation}   and $a=R_\mu\omega$ to obtain
\begin{align*}
    \int_{\R^2} \dot \omega  R_\mu \omega dx +\int_{\R^2}|\nabla a|^2 dx=0, 
\end{align*}
where the self-adjointness of the double Riesz transform yields (recalling the transport equation $\frac{D}{Dt}\mu=\dot \mu=0$ and the divergence free condition $\div u=0$ in \eqref{muNS})
\begin{align*}
     &\int_{\R^2} \dot \omega  R_\mu \omega dx 
     \\
     &= \int_{\R^2} 
     \Bigl( (R_2 R_2-R_1R_1)\dot \omega \,\cdot\, \mu  (R_2 R_2-R_1R_1)\omega
     + (2R_1R_2)\dot \omega  \,\cdot\,  \mu  (2R_1R_2)\omega\Bigr)  dx 
     \\
     &=\frac12\frac{d}{dt}\int_{\R^2} \mu \Bigl( ( (R_2 R_2-R_1R_1)\omega)^2
     +((2R_1R_2)\omega)^2\Bigr) dx
     \\
     &\quad +\int_{\R^2} 
     \mu \Bigl( [(R_2 R_2-R_1R_1), u\cdot\nabla ] \omega   \,\cdot\,    (R_2 R_2-R_1R_1)\omega
     + [(2R_1R_2), u\cdot\nabla ] \omega   \,\cdot\,   (2R_1R_2)\omega\Bigr) dx  .
\end{align*}
Thus,
\begin{align}\label{a:low-en-eq}
\begin{split}
    &\frac{1}{2}\frac{d}{dt} \int_{\R^2} \mu \Bigl( ((R_2R_2-R_1R_1) \omega)^2 + ((2R_1R_2) \omega)^2 \Bigr) dx + \int_{\R^2} \vert\nabla a\vert^2 dx \\
    &= - \int_{\R^2} \mu \Bigl([R_2R_2-R_1R_1,u\cdot\nabla]\omega  \,\cdot\, (R_2R_2-R_1R_1) \omega + [2R_1R_2,u\cdot\nabla]\omega  \,\cdot\,  (2R_1R_2)\omega\Bigr)dx .
\end{split}
\end{align}
Recall (the proof of) \eqref{L2:omega,a} for    the first integral on the left hand side  
\begin{align*}
    \int_{\R^2} \mu \Bigl( ((R_2R_2-R_1R_1) \omega)^2 + ((2R_1R_2) \omega)^2 \Bigr) dx
    =\langle a, \omega\rangle_{L^2(\R^2)}
   \geq \mu_\ast\|\omega\|_{L^2}^2.
\end{align*}
The integral on the righthand side can be bounded with the help of the commutator estimate from Lemma \ref{commutator:lem} by 
$C\mu^\ast \Vert\nabla u \Vert_{L^\infty} \Vert \omega \Vert_{L^2}^2$, 
and thus integrating the result over $[0,t]$ yields
\begin{align*}
    \frac{\mu_\ast}{2}\Vert\omega (t)\Vert_{L^2}^2 + \int_0^t \Vert \nabla a \Vert_{L^2}^2dt' 
    &\leq \frac{\mu^\ast}{2}\Vert \omega_0 \Vert_{L^2}^2 
    + C\mu^\ast \int_0^t \Vert \nabla u \Vert_{L^\infty} \Vert \omega \Vert_{L^2}^2 dt' .
\end{align*}
An application of Gronwall's inequality and the bound \eqref{L2:a,omega}
imply the estimate (\ref{a:energy-low}).

    \item \textbf{Proof of (\ref{a:tenergy-low}):} 
We multiply (\ref{a:low-en-eq}) by $ t $ to obtain
\begin{align*}
    &\frac{1}{2}\frac{d}{dt} \Bigl(  t \int_{\R^2} \mu \Bigl( ((R_2R_2-R_1R_1) \omega)^2 + ((2R_1R_2) \omega)^2 \Bigr) dx\Bigr) +   t \int_{\R^2} \vert\nabla a\vert^2 dx \\
    &= \frac{1}{2} \int_{\R^2} \mu \Bigl(  ((R_2R_2-R_1R_1) \omega)^2 + ((2R_1R_2) \omega)^2 \Bigr)  dx \\
    &-  t \int_{\R^2} \mu \Bigl([R_2R_2-R_1R_1,u\cdot\nabla]\omega \,\cdot\, (R_2R_2-R_1R_1) \omega + [2R_1R_2,u\cdot\nabla]\omega\,\cdot\, (2R_1R_2)\omega\Bigr)dx ,
\end{align*}
where integration over $[0,t]$ together with the commutator estimate \eqref{comm} implies
\begin{align*}
    \frac{\mu_\ast}{2} t \Vert\omega\Vert_{L^2}^2 + \int_0^t  t'  \Vert \nabla a \Vert_{L^2}^2 dt' 
    \lesssim_{\mu^\ast} 
    \int_0^t \Vert\omega\Vert_{L^2}^2 dt' + \int_0^t \Vert \nabla u \Vert_{L^\infty}   \Vert {t'}^{\frac12}\omega\Vert_{L^2}^2 dt' .
\end{align*}
Thus, (\ref{a:tenergy-low}) follows from Gronwall's inequality, \eqref{L2:a,omega}   and (\ref{u:energy}).

\item \textbf{Proof of (\ref{a:tenergy-high}):} 
For the higher order estimates we apply $R_\mu$ to the vorticity equation (\ref{omega:eq,dot}) to get
\begin{align*}
    R_\mu \dot\omega-R_\mu\Delta a=0,
\end{align*} 
and take the $L^2$ inner product with $\dot\omega$ to derive
\begin{align*}
    \int_{\R^2}R_\mu \dot\omega\, \dot\omega dx
    -\int_{\R^2} R_\mu\Delta a\, \dot\omega dx=0.
\end{align*}
We have by integration by parts that (recalling $\dot\omega=\frac{D}{Dt}\omega$)
\begin{align*}
    \int_{\R^2}R_\mu \dot\omega \,\dot\omega dx
    &=\int_{\R^2} \mu \Bigl(((R_2R_2-R_1R_1)\dot\omega)^2+((2R_1R_2)\dot\omega)^2\Bigr)dx ,
     \\
    -\int_{\R^2} R_\mu\Delta a\, \dot\omega dx
    &= -\int_{\R^2}(\Delta a)(\frac{D}{Dt}R_\mu\omega)dx - \int_{\R^2}(\Delta a)\,[R_\mu, \frac{D}{Dt}]\omega dx \\
    &=:I_1+I_2.
\end{align*}
As $R_\mu\omega=a$, we have by integration by parts (noticing $[\nabla, \frac{D}{Dt}]=[\nabla, u\cdot\nabla]$)
\begin{align*}
    I_1&=\int_{\R^2}\nabla a\cdot \frac{D}{Dt}\nabla adx
   +\int_{\R^2}\nabla a\cdot [\nabla, u\cdot\nabla] adx
   \\
   &=\frac12\frac{d}{dt}\int_{\R^2}|\nabla a|^2 dx
   +\int_{\R^2} \nabla a\cdot\nabla u\cdot\nabla a dx .
\end{align*}
Furthermore, since $\frac{D}{Dt}\mu=0$, the commutator in the second integral $I_2$ reads
\begin{align*}
    [R_\mu,\frac{D}{Dt}]&= (R_2R_2-R_1R_1)\mu[R_2R_2-R_1R_1,u\cdot \nabla] + (2R_1R_2)\mu [2R_1R_2,u\cdot \nabla] \\
    &\quad + [R_2R_2-R_1R_1,u\cdot \nabla]\mu (R_2R_2-R_1R_1) + [2R_1R_2,u\cdot \nabla] \mu (2R_1R_2) ,
\end{align*}
so that we arrive at (recalling the vorticity equation $\dot\omega=\Delta a$)
\begin{align}\label{grad:en-inequality}
    &\frac{1}{2}\frac{d}{dt}\Vert\nabla a \Vert_{L^2}^2 + \mu_\ast \Vert\Delta a \Vert_{L^2}^2\nonumber \\
    &\leq - \int_{\R^2} \nabla a\cdot\nabla u\cdot\nabla a dx\\
    &\quad -\int_{\R^2} (\Delta a)\Bigl((R_2R_2-R_1R_1) \bigl(\mu [R_2R_2-R_1R_1 , u\cdot\nabla]\omega\bigr) +(2R_1R_2) \bigl(\mu [2R_1R_2,u\cdot\nabla] \omega\bigr)\Bigr)dx \nonumber \\
    &\quad - \int_{\R^2} (\Delta a)\Bigl([R_2R_2-R_1R_1, u\cdot\nabla] \bigl(\mu(R_2R_2-R_1R_1)\omega \bigr)+ [2R_1R_2,u\cdot\nabla] \bigl(\mu (2R_1R_2) \omega\bigr)\Bigr)dx.\nonumber 
\end{align}
The last two integrals on the right hand side are bounded by $C\mu^\ast\Vert\Delta a \Vert_{L^2}\|\nabla u\|_{L^\infty}\Vert\omega\Vert_{L^2}$ due to the commutator estimate \eqref{comm}, and the first integral satisfies
\begin{align}\label{a:high-en-ineq}
    \Bigl\vert -\int_{\R^2} \nabla a\cdot\nabla u\cdot\nabla a dx \Bigr\vert \leq 
    \|\nabla u\|_{L^\infty}\|\nabla a\|_{L^2}^2.
\end{align}
Consequently,
\begin{align*}
    \frac{1}{2}\frac{d}{dt}\Vert\nabla a \Vert_{L^2}^2 + \mu_\ast \Vert\Delta a \Vert_{L^2}^2 \lesssim\mu^\ast\Vert\Delta a \Vert_{L^2}\Vert\nabla u \Vert_{L^\infty} \Vert \omega\Vert_{L^2} + 
    \|\nabla u\|_{L^\infty}\|\nabla a\|_{L^2}^2.
\end{align*}

We multiply (\ref{a:high-en-ineq}) by $ t $ to obtain
\begin{align*}
    &\frac{1}{2}\frac{d}{dt}\bigl(  t \Vert\nabla a \Vert_{L^2}^2\bigr) + \mu_\ast   t  \Vert \Delta a \Vert_{L^2}^2 \lesssim \Vert\nabla a\Vert_{L^2}^2 +  \mu^\ast   t \Vert\Delta a \Vert_{L^2} \Vert\nabla u \Vert_{L^\infty}\Vert\omega\Vert_{L^2} +     
    \|\nabla u\|_{L^\infty} t\|\nabla a\|_{L^2}^2.
\end{align*}
This implies
\begin{align*}
    \frac{1}{2}\frac{d}{dt}\bigl(t\Vert\nabla a \Vert_{L^2}^2\bigr) + \frac{\mu_\ast}{2} t \Vert \Delta a \Vert_{L^2}^2 
    \lesssim \Vert\nabla a\Vert_{L^2}^2 + \frac{(\mu^\ast)^2}{\mu_\ast }  t \Vert\nabla u \Vert_{L^\infty}^2 \Vert\omega\Vert_{L^2}^2 + \|\nabla u\|_{L^\infty} t\|\nabla a\|_{L^2}^2.
\end{align*}
so that (\ref{a:tenergy-high}) follows again by Gronwall's inequality and (\ref{u:energy}). 


\item \textbf{Proof of (\ref{a:tenergy-low+}):} 
We multiply (\ref{a:low-en-eq}) by $\langle t\rangle^{1+2\delta_-}$ to obtain
\begin{align*}
    \frac12 \frac{d}{dt} \Bigl(  t ^{1+2\delta_-} \int_{\R^2}\mu \Bigl ((R_2R_2-R_1R_1)\omega)^2 &+(2R_1R_2\omega)^2\Bigr)  dx\Bigr) +   t ^{1+2\delta_-} \Vert \nabla a \Vert_{L^2}^2 \\
    &\quad \lesssim \mu^\ast  t ^{2\delta_-} \Vert \omega \Vert_{L^2}^2 + \mu^\ast  t ^{1+2\delta_-} \Vert \nabla u \Vert_{L^\infty} \Vert \omega \Vert_{L^2}^2,
\end{align*}
where integration over $[0,t]$ yields
\begin{align*}
    \frac{  t ^{1+2\delta_-}}{\mu^\ast} \Vert \omega \Vert_{L^2}^2 + \Vert  {t'} ^{\frac12+\delta_-} \nabla a \Vert_{L^2_tL^2}^2  
    \lesssim 
    \mu^\ast \|  {t' }^{ \delta_-}   \omega \|_{L^2_tL^2}^2  + \mu^\ast\int_0^t {t'}^{1+2\delta_-} \Vert \omega \Vert_{L^2}^2 \Vert \nabla u \Vert_{L^\infty} dt'.
\end{align*}
Then (\ref{a:tenergy-low+}) follows from Gronwall's inequality. 

\end{itemize}
\end{proof}


\subsection{Step II. The time-Independent Lipschitz estimate}\label{sect:Lip}

In this subsection we establish the time-independent Lipschitz estimate for the fluid velocity. To do so, we follow the steps demonstrated in Figure \ref{diagram}.
Throughout this subsection time evolution is neglected, so that all quantities only depend on the spacial variable $x\in\R^2$.

\begin{proposition}[Time-independent Lipschitz estimate]\label{u:Lip:prop}
Let $a\in L^2\cap W^{1,p}(\R^2)$, $p\in (2,\infty)$ and $\mu \in L^\infty(\R^2;[\mu_\ast,\mu^\ast])$, $0<\mu_\ast\leq \mu^\ast$.
Assume further that  $\d_\tau \mu\in L^p(\R^2)$, where  $\tau\in L^\infty\cap \dot W^{1,p}(\R^2;\R^2)$ is a non-degenerate vector field.
Let $\phi\in H^2(\R^2)$ be the unique solution of \eqref{Lmu,a} on $\R^2$. 
Then $\|\nabla^2\phi\|_{L^\infty}$ can be bounded in terms of $\omega:=\Delta\phi$ and $\taub:=\frac{\tau}{|\tau|}$ as follows
\begin{align}\label{u:Lip}
    \Vert \nabla^2 \phi \Vert_{L^\infty} \lesssim \Vert \omega \Vert_{L^p}^{1-\frac2p} \Bigl(\Vert \nabla a\Vert_{L^p} +  \Vert (\nabla \taub,  \dtau\mu )\Vert_{L^p}   \Vert \nabla^2 \phi \Vert_{L^\infty}
     + \Vert \dtau\omega \Vert_{L^p}
    \Bigr)^ \frac{2}{p}.
\end{align}
In the above, the term $ \Vert (\nabla \taub,  \dtau\mu )\Vert_{L^p}   \Vert \nabla^2 \phi \Vert_{L^\infty} $ can be replaced by $ \Vert (\nabla \taub,  \dtau\mu )\Vert_{L^{p_1}}     \Vert \omega \Vert_{L^{p_2}} $  for $p_1,p_2\in (p,\infty)$, $\frac{1}{p_1}+\frac{1}{p_2}=\frac1p$.
\end{proposition}

\begin{proof}
Our goal  is to control $\|\nabla^2\phi\|_{L^\infty}$ by the right hand side of \eqref{u:Lip}
\begin{equation}\label{I}
 I:=   \Vert \omega \Vert_{L^p}^{1-\frac2p}  \Bigl(\Vert \nabla a \Vert_{L^p} +  \Vert (\nabla \taub,  \dtau\mu )\Vert_{L^p}   \Vert \nabla^2 \phi \Vert_{L^\infty}  
  + \Vert \dtau \omega \Vert_{L^p}
 \Bigr)^\frac{2}{p}.
\end{equation}
\noindent\textbf{Preliminary Estimate in the tangential direction $\|\dtau\nabla^2\phi\|_{L^p}$.} 
We first apply \eqref{comm:infty-0} with $X=\taub$ and $f=\omega$ to derive the following tangential regularity  (noticing $\nabla^2\phi=R^2\omega$) 
\begin{align}\label{dtau2phi:Lp}
    \|\dtau\nabla^2\phi\|_{L^p}\lesssim \Vert \nabla\taub \Vert_{L^p} \Vert \nabla^2\phi  \Vert_{L^\infty}+\|\dtau\omega\|_{L^p}.
\end{align}

\noindent\textbf{Step 1. Reduction to $\Vert  \d_n \nabla \phi \Vert_{L^\infty}$. }Using formula \eqref{nabla:tau,n} from Lemma \ref{lemma:tau,n} we write
\begin{align}\label{preliminary}
    \Vert \nabla^2\phi \Vert_{L^\infty} \leq \Vert \overline{\tau} \otimes\dtau \nabla \phi \Vert_{L^\infty} + \Vert n\otimes \d_n \nabla \phi \Vert_{L^\infty}
    \leq \Vert  \dtau \nabla \phi \Vert_{L^\infty} + \Vert  \d_n \nabla \phi \Vert_{L^\infty}.
\end{align}
It remains to control $\Vert  \d_n \nabla \phi \Vert_{L^\infty}$ by $I$, since we can use 
  the interpolation inequality \eqref{ip-linfty} and the above estimate \eqref{dtau2phi:Lp} to control the tangential derivative $\Vert  \dtau \nabla \phi \Vert_{L^\infty}$   by $I$:
\begin{align}\label{dtauu:Lp}
     \Vert  \dtau \nabla \phi \Vert_{L^\infty}
    \lesssim \Vert   \dtau \nabla \phi \Vert_{L^p}^{1-\frac2p} \Vert \nabla \bigl( \dtau \nabla \phi\bigr) \Vert_{L^p}^{\frac2p} 
    \lesssim \Vert \omega \Vert_{L^p}^{1-\frac{2}{p}} \Bigl(\Vert \nabla\taub \Vert_{L^p} \Vert \nabla^2\phi \Vert_{L^\infty} + \Vert \dtau \omega \Vert_{L^p} \Bigr)^{\frac{2}{p}}.
\end{align} 
where in the second inequality we   used also the definition  $\dtau=\taub\cdot\nabla$ and $L^p$-boundedness of Riesz operator.

\noindent\textbf{Step 2. Reduction to $\|\alpha\|_{L^\infty}$.}
We consider the normal derivative of $\nabla\phi$. 
Recall the reformulation \eqref{alpha:reform} in Lemma \ref{lemma:tau,n} such that 
\begin{align}\label{dnu-alpha}
     \d_{n}\nabla \phi = -\frac{\alpha}{\mu}   n  + 2 (\overline{\tau} \cdot \dtau  \nabla \phi)   n  -  \dtau \nabla^\perp \phi .
\end{align}
The last two terms on the right hand side are related to tangential derivatives and can be bounded  by $I$ by Step 1.
It remains to control $\|\alpha\|_{L^\infty}$ by $I$, since the first term satisfies
$\Vert n \frac{\alpha}{\mu} \Vert_{L^\infty} \leq \frac{1}{\mu_\ast} \Vert \alpha \Vert_{L^\infty}.
    $

\noindent\textbf{Step 3. Estimate for $\|\alpha\|_{W^{1,p}}$ and conclusion.}
Recall the  definition (\ref{alpha:def}) of $\alpha$:
\begin{align}\label{alpha:omega}
    \alpha= (\taub_2^2-\taub_1^2)\mu(\d_{22}-\d_{11})\phi + 2\taub_1\taub_2 \mu (2\d_{12}\phi)=(\taub_2^2-\taub_1^2)\mu (R_2R_2-R_1R_1)\omega + 2\taub_1\taub_2 \mu (2R_1R_2\omega).
\end{align}
We derive from the $L^p$-boundedness of the Riesz operator $R$  that 
\begin{equation}\label{alpha:Lp}
    \|\alpha\|_{L^p}\leq C(p,\mu^\ast) \|\omega\|_{L^p}.
\end{equation}
Applying $\dtau$ to \eqref{alpha:omega}  and recalling \eqref{dtau2phi:Lp} we derive 
\begin{align*}
    \|\dtau\alpha\|_{L^p}
    &\lesssim_{\mu^\ast} (\Vert \nabla\taub \Vert_{L^p} + \Vert \dtau\mu \Vert_{L^p}) \Vert \nabla^2 \phi \Vert_{L^\infty}+ \|\dtau\nabla^2\phi\|_{L^p}
    \\
    &\lesssim_{\mu^\ast} (\Vert \nabla\taub \Vert_{L^p} +\Vert \dtau\mu \Vert_{L^p}) \Vert \nabla^2 \phi \Vert_{L^\infty}+\|\dtau\omega\|_{L^p}.
\end{align*}
Now we bound $\|\nabla\alpha\|_{L^p}$ by use of  the relation between $a$ and $\alpha$ in \eqref{eq:a,alpha} and the $L^p$-boundedness of the Riesz operator as (recalling also \eqref{dtau2phi:Lp})
\begin{align}
    \|\nabla \alpha\|_{L^p}
    &\lesssim \|\nabla a\|_{L^p}
    +\|\dtau\alpha\|_{L^p}
    +\|\dtau(\mu\nabla^2\phi)\|_{L^p}
    +\|\nabla\taub\|_{L^p}\|\nabla^2\phi\|_{L^\infty}
    \nonumber \\
    &\lesssim \|\nabla a\|_{L^p}
    +(\Vert \nabla\taub \Vert_{L^p} +\Vert \dtau\mu \Vert_{L^p}) \Vert \nabla^2 \phi \Vert_{L^\infty}+\|\dtau\omega\|_{L^p}. \label{gradalpha}
\end{align}

Consequently, by use of the interpolation inequality
$$\|\alpha\|_{L^\infty}\lesssim \|\alpha\|_{L^p}^{1-\frac2p}\|\nabla\alpha\|_{L^p}^{\frac2p}$$
and the estimate \eqref{alpha:Lp} we achieve $\|\alpha\|_{L^\infty}\lesssim I$.
Hence, $\Vert \d_{n}\nabla\phi \Vert_{L^\infty}$ and $\Vert \nabla^2\phi \Vert_{L^\infty}$ are both controlled by $I$ by Step 1 and Step 2. In particular, this proves the desired estimate (\ref{u:Lip}).
\end{proof}

We fix $\epsilon>0$ from Lemma \ref{Rmu-inv:prop}, which depends only on $\mu_\ast, \mu^\ast$, and we may assume that $\epsilon\leq 2$. 
Choosing $p=2+\epsilon$ in \eqref{u:Lip} and combining Proposition \ref{u:Lip:prop} with Lemma \ref{Rmu-inv:prop} leads to the following corollary.

\begin{corollary}\label{coro}
Under the hypotheses of Proposition \ref{u:Lip:prop}, we have for $u:=\nabla^\perp\phi$
\begin{align}\label{u:Lip:eps}
    \Vert \nabla u \Vert_{L^\infty} \lesssim \Vert a \Vert_{L^{2+\epsilon}}^{\frac{\epsilon}{2+\epsilon}} \Bigl(\Vert \nabla a \Vert_{L^{2+\epsilon}} 
    + \Vert (\nabla \taub, \dtau\mu) \Vert_{L^{2+\epsilon}}    \Vert (\nabla u, a) \Vert_{L^\infty}  \Bigr)^{\frac{2}{2+\epsilon}}.
\end{align}
\end{corollary}
\begin{proof}
By definition of $a=R_\mu\omega$ and Lemma \ref{Rmu-inv:prop} we derive that
\begin{align*}
    \Vert \omega \Vert_{L^{2+\epsilon}} = \Vert R_\mu^{-1} R_\mu \omega \Vert_{L^{2+\epsilon}} \lesssim \Vert a \Vert_{L^{2+\epsilon}}.
\end{align*}
Now we rewrite  
\[\dtau \omega = R_\mu^{-1}R_\mu \dtau \omega = R_\mu^{-1}(\dtau a + [R_\mu, \dtau]\omega) .\]
By virtue of the commutator estimate (\ref{Rmu:comm}) and Lemma \ref{Rmu-inv:prop} again, we arrive at
\begin{equation}\label{dtauomega:2+eps}
\Vert \dtau \omega \Vert_{L^{2+\epsilon}} \lesssim \Vert \nabla a \Vert_{L^{2+\epsilon}} + (\Vert \nabla \taub \Vert_{L^{2+\epsilon}}+ \Vert \dtau\mu \Vert_{L^{2+\epsilon}}) (\Vert \nabla u \Vert_{L^\infty}+\Vert a \Vert_{L^\infty}) .
\end{equation}
Choosing $p=2+\epsilon$ in (\ref{u:Lip}) and using the above estimates for $\Vert \omega \Vert_{L^{2+\epsilon}}$ and $\Vert \dtau \omega \Vert_{L^{2+\epsilon}}$ we arrive at (\ref{u:Lip:eps}).
\end{proof}
 
\begin{remark}[Time-independent estimates of $\nabla u$  revisited]\label{rem:omega-a:Lp}
\begin{enumerate}[(i)]
    \item We can express $a,b$ in terms of $\mu,\omega$ in complex coordinates in $\R^2$:
\begin{align*}
    z=x_1+ix_2, \quad \bar z=x_1-ix_2,\quad x_1=\frac12(z+\bar z), \quad x_2=\frac1{2i}(z-\bar z),
\end{align*}
as follows (noticing $\d_1= (\d_z+\d_{\bar z})$, $\d_2=\frac{1}{i}(\d_z-\d_{\bar z})$, $\Delta=4\d_{z\bar z}$, $\frac{\d_z}{\d_{\bar z}}=\frac{4\d_{zz}^2}{\Delta}$, $\frac{\d_{\bar z}}{\d_{  z}}=\frac{4\d_{\bar z\bar z}^2}{\Delta}$) 
\begin{align*}
    &a=R_\mu\omega
=\Re[\frac{\d_z}{\d_{\bar z}}\mu \frac{\d_{\bar z}}{\d_z}\omega],
\quad b=Q_\mu\omega
=\mathrm{Im}[\frac{\d_z}{\d_{\bar z}}\mu \frac{\d_{\bar z}}{\d_z}\omega].
\end{align*}
Thus $\omega$ can be respresented in terms of $a,b,\mu$ as
\begin{align*}
\omega= \frac{\d_z}{\d_{\bar z}}\frac1\mu \frac{\d_{\bar z}}{\d_z}(a+ib).
\end{align*}
This shows that the vorticity $\omega$ can be written in terms of $a$ and $b$, and that in general, $a$ alone does not suffice to represent $\omega$.
As  is shown in \cite[Corollary 1.9, Theorem 1.11]{he2020solvability} that
the curl-free part (imaginary part)  $ \nabla b, \nabla\omega\not\in L^1_{\mathrm{loc}}$ for the stationary case with piecewise-constant viscosity,  we don't have energy estimates for $\nabla b, \nabla\omega$ in the presence of jumping viscosity coefficient. 

\item 
If $\mu\in [\mu_\ast,\mu^\ast]$, then we can control $a$ in terms of $\omega$ by use of the boundedness of the Riesz transform: 
\begin{equation}\label{Lp:a,omega}
    \|a\|_{L^p(\R^2)}\leq C(p,\mu^\ast)\|\omega\|_{L^p(\R^2)},
    \,\,  \forall p\in (1,\infty).
\end{equation}
We have already seen   in Lemma \ref{Rmu-inv:prop} that the reverse estimate holds for $p=2+\epsilon$, i.e. we can control the $L^{2+\epsilon}$-norm of $\omega$ by $\Vert a \Vert_{L^{2+\epsilon}}$.
We have taken $p=2+\epsilon$ close to $2$ when applying \eqref{u:Lip} to derive \eqref{u:Lip:eps}, since, in the proof,  when we control  $\|\dtau\omega\|_{L^p}$  by $\|\dtau a\|_{L^p}$, we make  use of the inverse $R_\mu^{-1}$, which in general is \textit{a priori} bounded in $L^{p}$ only for $p>2$ close to $2$.
\item In \cite{chemin1993persistance, chemin1991mouvement} J.-Y. Chemin established the celebrated (time-independent) Lipschitz estimate for the velocity field with a logarithm growth in the tangential regularity of $\omega$ with respect to the vector field $\tau$:
\begin{align*} 
    \Vert \nabla u \Vert_{L^\infty(\R^2)} \lesssim \Vert \omega \Vert_{L^p(\R^2)} + \Vert \omega \Vert_{L^\infty(\R^2)} \log \Bigl(e+ \Vert \frac{1}{\vert \tau \vert} \Vert_{L^\infty(\R^2)} \frac{\Vert \omega\Vert_{L^\infty(\R^2)} \Vert \tau \Vert_{C^\alpha(\R^2)} + \Vert \div(\tau \omega) \Vert_{C^{\alpha-1}(\R^2)}}{\Vert \omega \Vert_{L^\infty(\R^2)}}\Bigr)
\end{align*}
for $p\in [1,\infty)$.
This estimate comes essentially from the analysis of the elliptic equation $ \Delta\phi=(-\d_{\overline{\tau}}^\ast \, \d_{\overline{\tau}} -\d_n^\ast\,  \d_n)\phi=\omega$.
When taking time into account, the logarithmic growth in the $\tau$-norms, which grows exponentially   in $\|\nabla u\|_{L^1_tL^\infty}$ as   in \eqref{tau:exp},  implies finally the linear growth in  $\int^t_0\|\nabla u\|_{L^\infty}$ on the right hand side.
An application of  Gronwall's inequality yields the boundedness of $\|\nabla u\|_{L^\infty}$ on any bounded time interval. 
This is key in the regularity  propagation  of the vortex patch.

Our estimate \eqref{u:Lip:eps} is essentially of interpolation type, and we do not have an a priori $L^\infty$-estimate  for $\omega$.
When taking into account of time, we can not avoid the exponential growth in $\|\nabla u\|_{L^1_tL^\infty}$  on the right hand side.

\end{enumerate}
    
\end{remark}


\subsection{Step III. The $L^1_t\mathrm{Lip}(\R^2)$-estimate}\label{sect:L1Lip}

In this subsection we combine the results from the previous sections to deduce the $L^1Lip$-estimate for the velocity vector field.


\begin{proposition}[$L^1_t\mathrm{Lip}(\R^2)$-estimate]\label{u:L1Lip:prop}
Let $(\mu,u,\tau)$ be a sufficiently smooth solution of \eqref{muNS}-\eqref{eq:tau} on some time interval $[0,T^\ast)$, $T^\ast>0$. Then, under the assumptions of Theorem \ref{exthm} there exists a constant $C>0$ depending only on $\mu_\ast,\mu^\ast$ such that 
\begin{align}
    \Vert \nabla u \Vert_{L^1_tL^\infty} + \Vert t'^{\frac12} \nabla u \Vert_{L^2_tL^\infty} \leq 
    &C \Bigl( \Vert u_0 \Vert_{L^2}^{\frac{\epsilon}{2}}  (\|u_0\|_{\dot H^{-1}}+\|\mu_0-1\|_{L^2(\R^2)}\|u_0\|_{L^2})
\times \nonumber\\
&\qquad \times \bigl(\|\nabla u_0\|_{L^2}+\|(\nabla\taub_0, \d_{\taub_0}\mu_0)\|_{L^{2+\epsilon}}^{\frac{2+\epsilon}{\epsilon}}\bigr) \Bigr)^{\frac{2\epsilon}{(2+\epsilon)^2}},
    \quad t\in (0,T^\ast). \label{u:L1Lip-0} 
\end{align}
\end{proposition}
\begin{proof}
Let $t\in (0,T^\ast)$ be arbitrary but fixed. The goal is to prove that the $L^1_tLip$-norm of $u$ can be controlled independently of $t$.

\noindent\textbf{Step 1: Scaling consideration.}
For notational simplicity, we introduce 
\begin{align}\label{def:sigma}
\begin{split}
    \sigma_{-1} &:=\sigma_{-1}(\mu_0,u_0)= \Vert u_0(x) \Vert_{\dot H^{-1}(\R^2)} + \Vert \mu_0(x)-1 \Vert_{L^2(\R^2)} \Vert u_0(x) \Vert_{L^2(\R^2)}, \\
    \sigma_0 &:=\sigma_0(u_0)= \Vert u_0(x) \Vert_{L^2(\R^2)}, \\
    \sigma_1 &:=\sigma_1(\mu_0,u_0,\bar\tau_0)= \Vert u_0(x) \Vert_{\dot H^1(\R^2)} 
    +  \Vert (\d_{\taub_0}\mu_0(x), \nabla_x \taub_0 (x))\Vert_{L^{2+\epsilon}(\R^2)}  ^\frac{2+\epsilon}{\epsilon} ,
\end{split}
\end{align}
where $\epsilon$ depends only on $\mu_\ast$, $\mu^\ast$ by Lemma \ref{Rmu-inv:prop},
and 
\[
\tilde V(t):=\tilde V(u(t))= \exp\bigl( C(\|\nabla_x u(t',x)\|_{L^1_tL^\infty} + \Vert t'^{\frac12} \nabla_x u(t',x) \Vert_{L^2_tL^\infty}) \bigr),
\]
where $C$ is a big enough constant depending only on $\mu_\ast,\mu^\ast$.
We assume without loss of generality $\sigma_j>0$, $j=-1,0,1$.

For $\lambda>0$ we define the rescaled initial data
\begin{align*}
    \mu_{0,\lambda}(x):= \mu_0 (\lambda^{-1}x), \quad u_{0,\lambda}(x) := \lambda^{-1}u_0(\lambda^{-1}x), \quad \tau_{0,\lambda}(x) := \lambda^{-1}\tau_0 (\lambda^{-1}x),
    \quad \bar\tau_{0,\lambda}(x):=\frac{\tau_{0,\lambda}}{|\tau_{0,\lambda}|}(x).
\end{align*}
It is straightforward to verify that  $(\mu,u, \pi, \tau)$ is a solution of \eqref{muNS}-\eqref{eq:tau} with initial data $(\mu_0,u_0,\tau_0)$ on some time interval $[0,T^\ast)$, if and only if the rescaled triplet
\begin{align*}
    (\mu_\lambda,u_\lambda, \pi_\lambda, \tau_\lambda)(t,x):= (\mu, \lambda^{-1} u, \lambda^{-2}\pi, \lambda^{-1}\tau)(\lambda^{-2}t, \lambda^{-1}x)
\end{align*}
solves \eqref{muNS}-\eqref{eq:tau} with initial data $(\mu_{0,\lambda}, u_{0,\lambda}, \tau_{0,\lambda})$ on the time interval $[0,\lambda^2 T^\ast)$. 
Observe that after rescaling 
\begin{equation}\label{scale:sigma,lambda}\begin{split}
    \sigma_{-1,\lambda}&:=\sigma_{-1}(\mu_{0,\lambda}, u_{0,\lambda})=\lambda \sigma_{-1},
    \\
    \sigma_{0,\lambda}&:=\sigma_0(u_{0,\lambda})=\sigma_0,
    \\
    \sigma_{1,\lambda}&:=\sigma_1(\mu_{0,\lambda}, u_{0,\lambda}, \bar\tau_{0,\lambda})
    =\lambda^{-1}\sigma_1,
    \\
    \tilde V_\lambda(\lambda^2 t)
    &:=\tilde V(u_{\lambda}(\lambda^2 t))=\tilde V(t),
    \quad t\in (0,T^\ast).
\end{split}\end{equation}
In the following we fix 
\begin{align}\label{lambda}
   \lambda=\frac{\sigma_0}{\sigma_{-1}} = \frac{\Vert u_0(x) \Vert_{L^2(\R^2)}}{\Vert u_0(x) \Vert_{\dot H^{-1}(\R^2)} + \Vert \mu_0(x)-1 \Vert_{L^2(\R^2)} \Vert u_0(x) \Vert_{L^2(\R^2)}}  ,
\end{align}
such that
\begin{equation}\label{scale:sigma} 
  \sigma_{0,\lambda}+ \sigma_{-1,\lambda}=\sigma_0+\lambda\sigma_{-1}=2\sigma_0,
  \quad \sigma_{1,\lambda}=\lambda^{-1}\sigma_1=\sigma_0^{-1}(\sigma_1\sigma_{-1}).
\end{equation}
We consider the solution $(\mu_\lambda,u_\lambda,\tau_\lambda)$ of the system \eqref{muNS}-\eqref{eq:tau} with initial data $(\mu_{0,\lambda}, u_{0,\lambda}, \tau_{0,\lambda})$ on the time interval $[0,\lambda^2 T^\ast)$.
We define also $\bar\tau_{\lambda}(t,x)=\frac{\tau_\lambda}{|\tau_\lambda|}(t,x)$.
\smallbreak 

\noindent\textbf{Step 2: Preliminary estimates for $a$.}
We first summarize the energy estimates for $a$ from Section \ref{sect:energy} as follows  (noticing $\Vert a \Vert_{L^2_tL^2} \lesssim \Vert \nabla u \Vert_{L^2_tL^2}$)
\begin{align}\label{summary:a}
\begin{split}
 \Vert( t'^{\delta} a , t'^{\frac12+\delta} \nabla a)\Vert_{L^2_tL^2} 
    \leq C (\sigma_0+\sigma_{-1}) \tilde V(t),
    \,\,\Vert a \Vert_{L^2_tL^2} \leq C \sigma_0, \,\, 
    \Vert (\nabla a, t'^{\frac12} \Delta a) \Vert_{L^2_tL^2} 
    \leq C \sigma_1 \tilde V(t),   
    \, t\in (0,T^\ast),
\end{split}
\end{align}
where $\delta$  can be an arbitrary number in $(0,\frac12)$, as we have assumed initially $u_0\in L^2\cap \dot H^{-1}$.
In this paper we choose  $\delta$  such that
\begin{align}\label{delta}
    \delta\in (\frac{1}{2+\epsilon}, \frac{4+\epsilon}{4(2+\epsilon)})\subset (\frac{1}{2+\epsilon}, \frac{1}{2})\subset (\frac{1}{2+\epsilon}, \frac1{\epsilon}),
\end{align}
where we have   taken $\epsilon\leq 2$ (without loss of generality).
Thus the constant $C$ in \eqref{summary:a} depends
only on $\mu_\ast,\mu^\ast$. 
In the following we aim to achieve the $L^1_t W^{1,2+\epsilon}_x$-estimate  for the rescaled $a_\lambda$ by applying the interpolation idea.

Let $\omega_\lambda(t,x)=\nabla_x^\perp\cdot u_\lambda(t,x)=\lambda^{-2}\omega(\lambda^{-2}t,\lambda^{-1}x)$ be the rescaled vorticity and $a_\lambda(t,x)=(R_{\mu_\lambda}\omega_\lambda)(t,x)=\lambda^{-2}a(\lambda^{-2}t,\lambda^{-1}x)$ be the rescaled version of $a$.
Then by virtue of \eqref{scale:sigma,lambda} and \eqref{scale:sigma}, \eqref{summary:a} is rescaled into
\begin{align}\label{summary:a,lambda}
\begin{split}
\Vert(a_\lambda, t'^{\delta} a_{\lambda}, t'^{\frac12+\delta} \nabla a_{\lambda})\Vert_{L^2_{\lambda^2t}L^2} 
    \leq  C  \sigma_0 \tilde V(t),
    \,\, 
    \Vert (\nabla a_\lambda, t'^{\frac12} \Delta a_\lambda) \Vert_{L^2_{\lambda^2t}L^2} 
    \leq C  \sigma_0^{-1}(\sigma_{-1}\sigma_1) \tilde V(t),  
    \,\, t\in (0,T^\ast).
\end{split}
\end{align} 
By the interpolation inequality \eqref{ip-lr22} with $r=2+\epsilon$:
\begin{equation}\label{Interpolation:eps}
    \|g\|_{L^{2+\epsilon}}\lesssim \|g\|_{L^2}^{\frac{2}{2+\epsilon}}\|\nabla g\|_{L^2}^{\frac{\epsilon}{2+\epsilon}},
\end{equation}
we   derive  from \eqref{summary:a,lambda} that
\begin{align}\label{a:L1L2+}
    \Vert a_\lambda \Vert_{L^1_{\lambda^2t}L^{2+\epsilon}} &\lesssim \Bigl\Vert \Vert a_{\lambda} \Vert_{L^2}^\frac{2}{2+\epsilon} \Vert t'^{\frac12+\delta}\nabla a_{\lambda} \Vert_{L^2}^\frac{\epsilon}{2+\epsilon} t'^{-(\frac12+\delta)\frac{\epsilon}{2+\epsilon}}\Bigr\Vert_{L^1(0,{\lambda^2t})} \nonumber \\
    &\lesssim 
    \Vert t'^{\frac12+\delta}\nabla a_{\lambda} \Vert_{L^2_{\lambda^2t}L^2}^{\frac{\epsilon}{2+\epsilon}} \Bigl\Vert \Vert a_{\lambda} 
    \Vert_{L^2}^{\frac{2}{2+\epsilon}} 
    {t'}^{-(\frac12+\delta) \frac{\epsilon}{2+\epsilon} } \Bigr\Vert_{L^{\frac{2(2+\epsilon)}{2(2+\epsilon)-\epsilon}}(0,{\lambda^2t})}
    \lesssim \sigma_0\tilde V(t),\quad t\in (0,T^\ast),
\end{align}
where for the last inequality we used
\begin{itemize}
    \item If $\lambda^2 t<1$, then  (by \eqref{delta} such that $\frac12-(\frac12+\delta)\frac{\epsilon}{2+\epsilon}>0$, i.e. $\delta<\frac1\epsilon$)
\begin{align*}
    \Vert a_\lambda \Vert_{L^1_{\lambda^2t}L^{2+\epsilon}}  
    &\lesssim \sigma_0^{\frac{\epsilon}{2+\epsilon}}
    \tilde V(t)  \Vert a_{\lambda} \Vert_{L^2_{\lambda^2t}L^2}^\frac{2}{2+\epsilon}
    \Vert t'^{-(\frac12+\delta)\frac{\epsilon}{2+\epsilon}} \Vert_{L^2(0,1)}   
     \lesssim  \sigma_0  \tilde V(t).
\end{align*} 
\item If $\lambda^2t\geq 1$,  then we decompose the interval $(0,\lambda^2 t)$ into $(0,1)$ and $(1,\lambda^2t)$, such that (by \eqref{delta}: $\frac12-(\frac12+\delta)\frac{\epsilon}{2+\epsilon}
-\frac{2\delta}{2+\epsilon}<0$, i.e. $\delta>\frac{1}{2+\epsilon}$)
    \begin{align*}
    \Vert a_\lambda \Vert_{L^1_{\lambda^2t}L^{2+\epsilon}}  
    &\lesssim \sigma_0^{\frac{\epsilon}{2+\epsilon}}
    \tilde V(t) \Bigl(\Vert a_{\lambda} \Vert_{L^2_{\lambda^2t}L^2}^\frac{2}{2+\epsilon}
    \Vert t'^{-(\frac12+\delta)\frac{\epsilon}{2+\epsilon}} \Vert_{L^2(0,1)}
    \\
    &\qquad + \Vert t'^\delta a_{\lambda} \Vert_{L^2_{\lambda^2t}L^2}^\frac{2}{2+\epsilon} \Vert t'^{-(\frac12+\delta)\frac{\epsilon  }{2+\epsilon}-\frac{ 2\delta}{2+\epsilon}} \Vert_{L^2(1,\infty)}\Bigr)  \lesssim  \sigma_0  \tilde V(t).
\end{align*} 
\end{itemize}
Similarly, we obtain 
\begin{align}
    \Vert \nabla a_{\lambda} &\Vert_{L^1_{\lambda^2t}L^{2+\epsilon}}  \lesssim \Bigl\Vert \Vert \nabla a_{\lambda} \Vert_{L^2}^\frac{2}{2+\epsilon} \Vert t'^{\frac12} \Delta a_{\lambda} \Vert_{L^2}^\frac{\epsilon}{2+\epsilon} t'^{-\frac{\epsilon}{2(2+\epsilon)}} \Bigr\Vert_{L^1(0,{\lambda^2t})}\nonumber \\
    &\lesssim \Vert t'^{\frac12} \Delta a_{\lambda} \Vert_{L^2_{\lambda^2 t}L^2}^\frac{\epsilon}{2+\epsilon} \Bigl\Vert  \Vert \nabla a_{\lambda} \Vert_{L^2}^\frac{2}{2+\epsilon} t'^{-\frac{\epsilon}{2(2+\epsilon)}} \Bigr\Vert_{L^\frac{2(2+\epsilon)}{2(2+\epsilon)-\epsilon}(0,{\lambda^2t})} 
    \lesssim   \sigma_0^{\theta_1}(\sigma_1\sigma_{-1})^{\theta_2} \tilde V(t), 
    \quad t\in (0,T^\ast),\label{grada:L1L2+}
\end{align}
where $\theta_1=\frac{2\frac{1-2\delta}{1+2\delta}-\epsilon}{2+\epsilon},$  $\theta_2=\frac{2\frac{2\delta}{1+2\delta}+\epsilon}{2+\epsilon}$, and for the last inequality  we performed as follows:
\begin{itemize}
    \item Firstly, for some $t_1\in (0,\lambda^2t]$, we can bound
    \begin{align*}
         \Vert \nabla a_{\lambda} \Vert_{L^1_{\lambda^2t}L^{2+\epsilon}}  
    &\lesssim  (\sigma_0^{-1}\sigma_{1}\sigma_{-1})^\frac{\epsilon}{2+\epsilon}\tilde V(t) \Bigl( \Vert \nabla a _\lambda \Vert_{L^2_{\lambda^2 t}L^2}^\frac{2}{2+\epsilon} \Vert t'^{-\frac{\epsilon}{2(2+\epsilon)}} \Vert_{L^2(0, t_1)}
    \\
    &\qquad + \Vert t'^{\frac12+\delta} \nabla a_\lambda \Vert_{L^2_{\lambda^2 t}L^2}^\frac{2}{2+\epsilon} \Vert t'^{-\frac{\epsilon}{2(2+\epsilon)}- (\frac12+\delta)\frac{2}{2+\epsilon}} \Vert_{L^2( t_1,\lambda^2t)}\Bigr) 
    \\
    &\lesssim  (\sigma_0^{-1}\sigma_{1}\sigma_{-1})^\frac{\epsilon}{2+\epsilon}\tilde V(t) \Bigl( (\sigma_0^{-1}\sigma_1\sigma_{-1})^\frac{2}{2+\epsilon} 
      t_1 ^{ \frac{1}{ 2+\epsilon}}
    +  \sigma_0^\frac{2}{2+\epsilon} t_1^{ -  \frac{2\delta}{2+\epsilon}}  \Bigr). 
    \end{align*}
    \item Secondly, if $\lambda^2t\geq t_0:=(\frac{\sigma_0^2}{\sigma_1\sigma_{-1}})^{\frac{2}{1+2\delta}}$, then we take $t_1=t_0$ above, 
    while if $\lambda^2t<t_0$ we can simply bound with the first term in the bracket with $t_1=t_0$.   
\end{itemize}
Now we can interpolate between \eqref{a:L1L2+} and \eqref{grada:L1L2+} to achieve
\begin{align}\label{a:L1Linfty}
    \|a_\lambda\|_{L^1_{\lambda^2t}L^\infty}
    \lesssim  \|a_\lambda\|_{L^1_{\lambda^2t}L^{2+\epsilon}}^{\frac{\epsilon}{2+\epsilon}} \|\nabla a_\lambda\|_{L^1_{\lambda^2t}L^{2+\epsilon}}^{\frac{2}{2+\epsilon}}\lesssim  \sigma_0^{\theta_3}(\sigma_1\sigma_{-1})^{\theta_4}\tilde V(t), \quad t\in (0,T^\ast),
\end{align}
where $\theta_3=\frac{\epsilon}{2+\epsilon}+\frac{2}{2+\epsilon}\theta_1=\frac{4\frac{1-2\delta}{1+2\delta}+\epsilon^2}{(2+\epsilon)^2}>0$, $\theta_4=\frac{2}{2+\epsilon}\theta_2=\frac{4\frac{2\delta}{1+2\delta}+2\epsilon}{(2+\epsilon)^2}>0$.
 
Very similar calculations show that $\Vert t'^{\frac12}a \Vert_{L^2_tL^{2+\epsilon}}$, $\Vert t'^{\frac12}\nabla a \Vert_{L^2_tL^{2+\epsilon}}$ and $\Vert t'^{\frac12} a \Vert_{L^2_tL^\infty}$ can also be bounded by the right hand sides of \eqref{a:L1L2+}, \eqref{grada:L1L2+} and \eqref{a:L1Linfty}, respectively. We omit the details here.

\smallbreak

\noindent\textbf{Step 3:  $L^\infty L^{2+\epsilon}$-estimates for $(\nabla \taub, \dtau\mu)$.} 
%
We derive the evolution equation for $\taub=\frac{\tau}{|\tau|}$ from the equation \eqref{eq:tau} for $\tau$  as
\begin{align}\label{eq:taub}
    \d_t \taub + u\cdot \nabla \taub = \dtau u - \taub (\taub\otimes \taub : \nabla u),
\end{align}
so that by an application of the gradient to this equation we find that
\begin{align*}
    \Vert \nabla \taub \Vert_{L^{2+\epsilon}} \lesssim \Vert \nabla \taub_0 \Vert_{L^{2+\epsilon}} + \int_0^t \Vert \nabla \taub \Vert_{L^{2+\epsilon}} \Vert \nabla u \Vert_{L^\infty} + \Vert \dtau \nabla u \Vert_{L^{2+\epsilon}} dt'.
\end{align*}
By virtue of \eqref{dtau2phi:Lp} and \eqref{dtauomega:2+eps} we have
\begin{align*}
    \Vert \dtau \nabla u \Vert_{L^{2+\epsilon}} \lesssim \Vert \nabla \taub \Vert_{L^{2+\epsilon}} (\Vert \nabla u \Vert_{L^\infty} + \Vert a \Vert_{L^\infty}) + \Vert \nabla a \Vert_{L^{2+\epsilon}} ,
\end{align*}
and hence
\begin{align}\label{tau:LinftyL2+e}
    \Vert \nabla \taub \Vert_{L^\infty_tL^{2+\epsilon}} \lesssim \bigl(\Vert \nabla \taub_0 \Vert_{L^{2+\epsilon}} + \Vert \nabla a \Vert_{L^1_tL^{2+\epsilon}} \bigr) \exp\bigl(  C\|a\|_{L^1_tL^\infty} \bigr) V(t).
\end{align} 

Next, we deduce the evolution equation for $\dtau\mu$ from the equations of $\d_\tau\mu$ and $\frac{1}{\vert\tau\vert}$:
\begin{equation}\label{eq:dtaubmu}
\d_t \dtau\mu + u\cdot\nabla \dtau\mu = -\dtau\mu (\taub\cdot \dtau u) ,
\end{equation}
from which it follows that 
\begin{align}\label{dtaumu-Linfty}
    \Vert \dtau\mu\Vert_{L^\infty_tL^{2+\epsilon}} \leq \Vert \d_{\overline{\tau}_0}\mu_0\Vert_{L^{2+\epsilon}} V(t), \quad  \text{with } \overline{\tau}_0=\frac{\tau_0}{\vert\tau_0 \vert}.
\end{align}

By the definition \eqref{def:sigma},   the choice of $\lambda$ in \eqref{lambda} and the scaling relation \eqref{scale:sigma} 
we obtain
\begin{align}\label{tau,mu:LinftyL2+}
\begin{split}
    &\Vert \nabla \taub_\lambda \Vert_{L^\infty_{\lambda^2t}L^{2+\epsilon}} + \Vert \d_{\taub_\lambda}\mu_\lambda \Vert_{L^\infty_{\lambda^2t}L^{2+\epsilon}} \\
    &\lesssim \bigl(\lambda^{-\frac{\epsilon}{2+\epsilon}} (\Vert \nabla \taub_{0} \Vert_{L^{2+\epsilon}} + \Vert \d_{\taub_{0}} \mu_{0} \Vert_{L^{2+\epsilon}}) + \Vert \nabla a_\lambda \Vert_{L^1_{\lambda^2t} L^{2+\epsilon}}\bigr) \exp(C \Vert a_\lambda \Vert_{L^1_{\lambda^2t}L^\infty)}) \tilde V_\lambda(\lambda^2 t) \\
    &= \bigl((\sigma_0^{-1} \sigma_{-1} \sigma_1)^\frac{\epsilon}{2+\epsilon} +  \Vert \nabla a_\lambda \Vert_{L^1_{\lambda^2t} L^{2+\epsilon}}\bigr)\exp(C \Vert a_\lambda \Vert_{L^1_{\lambda^2t}L^\infty}) \tilde V(t), \quad t\in (0,T^\ast).
\end{split}
\end{align}

\noindent\textbf{Step 4: Lipschitz estimates  for $u$.} 
The time-independent Lipschitz estimate (\ref{u:Lip:eps}) for the rescaled solution $(\mu_\lambda, u_\lambda, \tau_\lambda)$  and H\"older's inequality with respect to the time variable yields
\begin{align*}
    \Vert \nabla u_\lambda \Vert_{L^1_{\lambda^2t}L^\infty} \lesssim 
    \Vert a_\lambda \Vert_{L^1_{\lambda^2t}L^{2+\epsilon}}^{\frac{\epsilon}{2+\epsilon}} \Bigl( \Vert \nabla a_\lambda \Vert_{L^1_{\lambda^2t} L^{2+\epsilon}}
    + & \Vert(  \nabla\taub_\lambda,  \d_{\taub_\lambda}\mu_\lambda)\Vert_{L^\infty_{\lambda^2t}L^{2+\epsilon}}   \Vert  (\nabla u_\lambda,   a_\lambda)\Vert_{L^1_{\lambda^2t}L^\infty}  \Bigr)^{\frac{2}{2+\epsilon}} , \\
    \|t'^{\frac12}\nabla u_\lambda\|_{L^2_{\lambda^2t} L^\infty} \lesssim \Vert t'^{\frac12} a_\lambda \Vert_{L^2_{\lambda^2t}L^{2+\epsilon}}^{\frac{\epsilon}{2+\epsilon}}
    \Bigl( \Vert t'^{\frac12}\nabla a_\lambda \Vert_{L^2_{\lambda^2t} L^{2+\epsilon}}
    + & \Vert (\nabla\taub_\lambda, \d_{\taub_\lambda}\mu_\lambda) \Vert_{L^\infty_{\lambda^2t}L^{2+\epsilon}}  
    \Vert t'^{\frac12} (\nabla u_\lambda,   a_\lambda)\Vert_{L^2_{\lambda^2t}L^\infty} \Bigr)^{\frac{2}{2+\epsilon}}, 
\end{align*}
for $t\in (0,T^\ast)$.
By use of the estimates \eqref{a:L1L2+}, \eqref{grada:L1L2+}, \eqref{a:L1Linfty} (together with the version with respect to the time-weighted norm $L^2(tdt)$) and \eqref{tau,mu:LinftyL2+} above, we obtain 
\begin{align*}
   \Vert \nabla u_\lambda \Vert_{L^1_{\lambda^2t}L^\infty}
   + \|t'^{\frac12}\nabla u_\lambda\|_{L^2_{\lambda^2t} L^\infty} &\lesssim  \sigma_0^\frac{\epsilon}{2+\epsilon} \Bigl((\sigma_0^{-1}\sigma_{-1}\sigma_1)^\frac{\epsilon}{2+\epsilon} + \sigma_0^{\theta_1}(\sigma_{-1}\sigma_1)^{\theta_2}\Bigr)^\frac{2}{2+\epsilon}  \tilde V(t) \exp \bigl(C \sigma_0^{\theta_3}(\sigma_{-1}\sigma_1)^{\theta_4} \tilde V(t)\bigr) \\
    &\lesssim  \sigma_0^{\frac{\epsilon^2}{(2+\epsilon)^2}} (\sigma_{-1} \sigma_1)^{\frac{2\epsilon}{(2+\epsilon)^2}} 
    \tilde V(t)\exp\bigl(C  \sigma_0^{\theta_3} (\sigma_{-1} \sigma_1)^{\theta_4} \tilde V(t)\bigr),
    \quad t\in (0,T^\ast).
\end{align*}


We now perform the bootstrap argument. 
Let 
$$A(t):=A(u(t))=\Vert \nabla u \Vert_{L^1_tL^\infty} +\|t'^{\frac12}\nabla u\|_{L^2_t L^\infty},\quad t\in (0,T^\ast),$$
denote a time-dependent nonnegative function, such that
\begin{align*}
    \tilde V(t)=e^{CA(t)},\quad A_\lambda(\lambda^2t):=A(u_\lambda(\lambda^2 t))=A(u(t))=A(t),
    \quad t\in (0,T^\ast).
\end{align*}
Thus from the above it satisfies 
\begin{align*}
    A(t) &\leq C\sigma_0^{\frac{\epsilon^2}{(2+\epsilon)^2}} (\sigma_{-1} \sigma_1)^{\frac{2\epsilon}{(2+\epsilon)^2}} 
    \exp\bigl(CA(t)+C \sigma_0^{\theta_3}  (\sigma_{-1} \sigma_1)^{\theta_4}  e^{CA(t)}\bigr).
\end{align*}
Recall  the definition of $\theta_3,\theta_4$ in \eqref{a:L1Linfty} and the restriction of $\delta$ in \eqref{delta}, where we have taken $\epsilon\in (0,2]$, 
such  that 
\begin{align*}
    &\frac{\theta_3}{\theta_4}=\frac{4\frac{1-2\delta}{1+2\delta}+\epsilon^2}{4\frac{2\delta}{1+2\delta}+2\epsilon}
    =-\frac{4-\epsilon^2}{2(2+\epsilon)}
    +\frac{2+\epsilon}{\delta(4+2\epsilon)+\epsilon}\in 
    \bigl( \frac{\epsilon(2+3\epsilon)}{2(4+3\epsilon)}, \frac{\epsilon}{2}\bigr), 
    \hbox{ is close to }\frac{\epsilon}{2} \hbox{ if }\delta\rightarrow (\frac{1}{2+\epsilon})_+,
    \\
    &\theta_4=\frac{4\frac{2\delta}{1+2\delta}+2\epsilon}{(2+\epsilon)^2}=\frac{2}{2+\epsilon}-(\frac{2}{2+\epsilon})^2\frac{1}{1+2\delta }\in \bigl( \frac{2}{4+\epsilon}, \frac{2(4+3\epsilon)}{(2+\epsilon)(8+3\epsilon)}\bigr)
    \hbox{ is uniformly bounded in terms of  }\mu_\ast,\mu^\ast.
\end{align*}
Under the smallness assumption 
\begin{align}\label{smallness:proof}
    2C^2\bigl( \sigma_0^{\frac{\epsilon}{2}}  \sigma_{-1} \sigma_1\bigr)^{\frac{2\epsilon}{(2+\epsilon)^2}}
    +C\sqrt{e} \Bigl(\sigma_0^{\frac{\theta_3}{\theta_4} }\sigma_{-1} \sigma_1\Bigr)^ { \theta_4} \leq \frac12,
\end{align}
by the standard bootstrap argument we have the uniform bound 
$$A(t)\leq 2C\bigl( \sigma_0^{\frac{\epsilon}{2}}  \sigma_{-1} \sigma_1\bigr)^{\frac{2\epsilon}{(2+\epsilon)^2}},\quad \forall t\in (0,T^\ast).$$
Notice that if the smallness assumption \eqref{u0:cond}: $\sigma_0^{\frac{\epsilon}{2}}\sigma_{-1}\sigma_1\leq c_0$ is satisfied, then we can choose $\delta\in (\frac{1}{2+\epsilon}, \frac{4+\epsilon}{4(2+\epsilon)})$ (recalling \eqref{delta}) close to $\frac{1}{2+\epsilon}$ such that  $\frac{\theta_3}{\theta_4}$ is close to $\frac{\epsilon}{2}$, and hence \eqref{smallness:proof} holds if $c_0$ is small enough.
This completes the proof.

\end{proof}

\subsection{Proof of Theorem \ref{exthm}}\label{sect:proofs}

In this subsection we prove Theorem \ref{exthm} by use of the a priori estimates from the previous subsections.

\begin{proof}[Proof of Theorem \ref{exthm}]
We start with the proof of existence. The idea is to smooth out the given initial data and then show the convergence of the approximation solution sequence by uniform bounds and compactness.

\noindent
\textbf{Step 1: Approximation solution sequence.} 
Given the initial data as in the hypotheses of Theorem \ref{exthm} we are going to smooth them out using the standard Friedrich's mollifier. Let $\eta \in C_c^\infty((0,\infty); [0,1])$ be a smooth cut-off function with $\int_{\R}\eta=1$.
Denote $\eta_j(x)=j^2 \eta (j\vert x\vert)$, $j\in \N$. Define the regularized initial data by the convolution with $\eta_j$ as
\begin{align*}
  &  \mu_0^j = \mu_0\ast \eta_j, \quad u_0^j = u_0\ast \eta_j.
\end{align*}
Then we have
\begin{align}
&\mu_\ast \leq \mu_0^j \leq \mu^\ast,
\,\, \|\mu_0^j-1\|_{L^2}\leq \|\mu_0-1\|_{L^2},
\,\,  
    \|u_0^j\|_{H}\leq \|u_0\|_{H},\,\, H=\dot H^1, L^2, \dot H^{-1}, 
\nonumber\\
&  \|\d_{\bar\tau_0}\mu^j_0\|_{L^{2+\epsilon}} \leq 
  \Vert (\d_{\taub_0}\mu_0)\ast\eta_j \Vert_{L^{2+\epsilon}}
  +\Vert [\d_{\taub_0}, \eta_j\ast]\mu_0 \Vert_{L^{2+\epsilon}} 
  \leq \|\d_{\bar\tau_0}\mu_0\|_{L^{2+\epsilon}}
  +C\mu^\ast\|\nabla\taub_0\|_{L^{2+\epsilon}}.
  \label{uniform:initial}
    \end{align}

We regularize the transported velocity and the viscosity coefficient in  the Cauchy problem of the coupled system \eqref{muNS}-\eqref{eq:tau} as follows: 
\begin{equation}\label{muNS-reg}
    \begin{cases}
        \d_t\mu+(u\ast\eta_j)\cdot\nabla \mu=0, 
        \quad \d_t\tau+(u\ast\eta_j)\cdot \nabla\tau = \d_\tau (u\ast\eta_j) ,\quad (t,x)\in (0,\infty)\times \R^2,\\
        \d_t  u + (u\ast\eta_j)\cdot \nabla u  - \div ((\mu \ast \eta_j) Su)+\nabla \pi = 0,\quad 
        \div u=0 , \\ 
        (\mu^j,u^j,\tau^j)|_{t=0}=(\mu_0^j,u_0^j,\tau_0),
        \hbox{ with }\taub_0^j=\taub_0.
    \end{cases}
\end{equation}
By the classical existence theory (see e.g. \cite{lions1996mathematical}) there exists for big enough $j\in\N$ a smooth global-in-time solution $(\mu^j,u^j,\nabla\pi^j,\tau^j)$ of  \eqref{muNS-reg}.

We remark that with the regularized ``material derivative" $$D_t^j=\d_t+(u\ast\eta_j)\cdot\nabla,$$
the first two equations in $\eqref{muNS-reg}$ mean that  $D_t^j\mu=0$ and    $\d_\tau=\tau\cdot\nabla$ commutes with $D_t^j$.
Hence \eqref{muNS-reg} implies the free transport of the tangential derivative $\d_\tau\mu$
\begin{align}\label{strreg-j}
    D_t^j(\d_\tau\mu)=\d_\tau (D_t^j\mu)=0.
\end{align}
Consequently, similar as in \eqref{eq:taub} and \eqref{eq:dtaubmu}, we have the following equations for $\taub^j=\frac{\tau^j}{|\tau^j|}$   and $\d_{\taub^j}\mu^j$:
\begin{align}
    & \d_t \taub  + (u\ast\eta_j)\cdot \nabla \taub = \dtau (u\ast\eta_j) - \taub (\taub\otimes \taub : \nabla u\ast\eta_j),\label{eq:taub,j}
    \\
    &\d_t \dtau\mu + (u\ast\eta_j)\cdot\nabla \dtau\mu = -\dtau\mu (\taub\cdot \dtau (u\ast\eta_j)).\label{eq:dtaubmu,j}
\end{align}
We notice that the $\tau$-equation in \eqref{muNS-reg} implies the boundedness and nondegenerity of the vector field $\tau^j$  
\begin{align*}
\|\tau^j\|_{L^\infty_tL^\infty}
\leq \|\tau_0\|_{L^\infty}V^j(t),\quad
    \|\frac{1}{|\tau^j|}\|_{L^\infty_t L^\infty}
    \leq \|\frac{1}{|\tau_0|}\|_{L^\infty} V^j(t),
    \quad V^j(t):=\exp(C  \|\nabla u^j\|_{L^1_tL^\infty}),
\end{align*}
as long as $V^j(t)<\infty$. 
We have this estimate for all time in \eqref{u:L1Lip-reg} below, which implies the legitimacy of the definition of $\taub^j$.

\noindent\textbf{Step 2: Uniform bounds.} 
We show that the a priori estimates in the previous Sections \ref{sect:energy}, \ref{sect:Lip} and \ref{sect:L1Lip} stay valid for solutions $(\mu^j,u^j,\nabla\pi^j,\tau^j)$ of \eqref{muNS-reg} and  we denote $a^j:=R_{\mu^j\ast\eta_j}\omega^j$ with $\omega^j = \nabla^\perp \cdot u^j$. 
Recall the uniform bounds \eqref{uniform:initial} for the initial data.

Observe that $\mu_\ast\leq \mu^j(t,x)\leq \mu^\ast$.
Firstly,  the energy estimates \eqref{u:energy} and \eqref{u:decay+} for $u^j$ follow exactly as before
\begin{align}\label{uniform:u}
    \|\langle t'\rangle^{\delta}u^j\|_{L^\infty_tL^2\cap L^2_t\dot H^1}\leq C(\mu_\ast,\mu^\ast)(\|u_0\|_{L^2\cap \dot H^{-1}}+\|\mu_0-1\|_{L^2}\|u_0\|_{L^2}),
\end{align}
where we choose $\delta\in (\frac{1}{2+\epsilon}, \frac{4+\epsilon}{4(2+\epsilon)})$ as in \eqref{delta}.
Next, an application of the curl operator to the regularized velocity equation \eqref{muNS-reg}$_2$ yields the following analogue of the vorticity equation \eqref{omega:eq} for $\omega^j$ and $a^j$:
\begin{align}\label{omega:eq-reg}
    D_t^j \omega^j -\Delta a^j = - (\nabla^\perp u^j\ast\eta_j): (\nabla u^j)^T , \quad a^j=R_{\mu^j\ast\eta_j} \omega^j, \quad u^j = \nabla^\perp \Delta^{-1} \omega^j .
\end{align}
We have the $L^2$-energy estimate \eqref{a:energy-low}, \eqref{a:tenergy-low} and\eqref{a:tenergy-low+} as well as  $H^1(\R^2)$-estimates \eqref{a:tenergy-high} for $a^j$:
\begin{align}\label{uniform:a}
    &\|a^j\|_{L^\infty_tL^2\cap L^2_t\dot H^1}\leq C\|\nabla u_0\|_{L^2} V^j(t),
    \,\, \|t'^{\frac12}a^j\|_{L^\infty_tL^2\cap L^2_t\dot H^1}\leq C\|u_0\|_{L^2}V^j(t),\\
    &\|t'^{\frac12+\delta}a^j\|_{L^\infty_tL^2\cap L^2_t\dot H^1}\leq C (\|u_0\|_{L^2\cap \dot H^{-1}}+\|\mu_0-1\|_{L^2}\|u_0\|_{L^2})V^j(t),\,\,  V^j(t)=\exp(C  \|\nabla u^j\|_{L^1_tL^\infty}),\label{uniform:a,delta}\\
   & \|t'^{\frac12}\nabla a^j\|_{L^\infty_tL^2\cap L^2_t\dot H^1}\leq C\|\nabla u_0\|_{L^2}\tilde V^j(t),
    \,\, \tilde V^j(t)=V^j(t)\exp(C\|t'^{\frac12}\nabla u^j\|_{L^2_tL^\infty}).\label{uniform:grada}
\end{align}
Indeed, as in the proof of \eqref{a:energy-low}, 
we take the $L^2$-inner product between \eqref{omega:eq-reg} and $a^j=R_{\mu^j\ast\eta_j}\omega^j$ to derive \eqref{a:low-en-eq}, with $\mu$ replaced by $\mu^j\ast \eta_j$, $u^j\cdot\nabla$ replaced by $(u^j\ast\eta_j)\cdot\nabla$ and the following additional   terms on the right hand side:
\begin{align*}
     -\int_{\R^2} (\nabla^\perp u^j \ast \eta_j):(\nabla u)^T R_{\mu^j\ast \eta_j} \omega^j dx 
     +\frac12\int_{\R^2} \bigl([D_t^j,\ast\eta_j]\mu^j \bigr) \Bigl(((R_2R_2-R_1R_1)\omega^j)^2+(2R_1R_2\omega^j)^2\Bigr) dx
\end{align*}
which can be bounded by $\|\nabla u^j\|_{L^\infty} \|\omega^j\|_{L^2}^2$.
The $L^2$-estimates \eqref{uniform:a} and \eqref{uniform:a,delta} follow from  (the modified version) of \eqref{a:low-en-eq} immediately.
Similarly,  we take the $L^2$-inner product between  \eqref{omega:eq-reg} and $R_{\mu^j\ast\eta_j}\Delta R_{\mu^j\ast\eta_j} \omega^j$ to derive \eqref{grad:en-inequality}, with $\mu$, $u$ replaced by $\mu^j\ast\eta_j$, $u^j\ast\eta_j$ respectively, and with the following additional integral on the right hand side
\begin{align*}
\int_{\R^2} R_{\mu^j\ast\eta_j} \Bigl((\nabla^\perp u^j\ast \eta_j):(\nabla u^j)^T\Bigr) \Delta a^j dx, 
\end{align*}
which can be bounded by $\|\nabla u^j\|_{L^\infty} \|a^j\|_{L^2}\|\Delta a^j\|_{L^2}$. 
The $H^1$-estimate \eqref{uniform:grada} follows.

As Corollary \ref{coro} holds via the consideration of $\alpha^j =(\frac{\tau^j}{\vert \tau^j\vert}\otimes \frac{(\tau^j)^\perp}{\vert \tau^j\vert}):  ((\mu^j\ast \eta_j) Su^j)$, under the smallness assumption \eqref{u0:cond} (with possibly a slightly smaller $c_0$), along the same lines as in the proof for Proposition \ref{u:L1Lip:prop},  we  deduce 
\begin{align}\label{u:L1Lip-reg}
   & \Vert \nabla u^j \Vert_{L^1_tL^\infty} + \Vert t'^{\frac12} \nabla u^j \Vert_{L^2_tL^\infty}
   \\
   &\leq  C(\mu_\ast,\mu^\ast) \Bigl( \Vert u_0 \Vert_{L^2}^{\frac{\epsilon}{2}}  (\|u_0\|_{\dot H^{-1}}+\|\mu_0-1\|_{L^2(\R^2)}\|u_0\|_{L^2})
  \bigl(\|\nabla u_0\|_{L^2}+\|(\nabla\taub_0, \d_{\taub_0}\mu_0)\|_{L^{2+\epsilon}}^{\frac{2+\epsilon}{\epsilon}}\bigr) \Bigr)^{\frac{2\epsilon}{(2+\epsilon)^2}},
  \nonumber
\end{align} 
where we have in between used the uniform bounds for $(\nabla\taub^j, \d_{\taub^j}\mu^j)$ (recalling \eqref{tau:LinftyL2+e}, \eqref{dtaumu-Linfty} and \eqref{eq:taub,j}, \eqref{eq:dtaubmu,j})
\begin{align*}
    \Vert (\nabla \taub^j, \d_{\taub^j}\mu^j) \Vert_{L^\infty_tL^{2+\epsilon}} \lesssim
    \bigl(\Vert (\nabla \taub_0, \d_{\taub_0}\mu^j_0) \Vert_{L^{2+\epsilon}} + \Vert \nabla a^j \Vert_{L^1_tL^{2+\epsilon}} \bigr) \exp\bigl(  C\|a^j\|_{L^1_tL^\infty} \bigr) V^j(t).
\end{align*}
To conclude,
\begin{align}\label{uniform:u,a}
    \|(\langle t'\rangle^{\delta}u^j, \langle t'\rangle^{\frac12+\delta}a, t'^{\frac12}\nabla a^j)\|_{L^\infty_tL^2\cap L^2_t\dot H^1}
    +\tilde V^j(t)
    +\|a^j\|_{L^1_tW^{1,2+\epsilon}}
    +\Vert (\nabla \taub^j, \d_{\taub^j}\mu^j) \Vert_{L^\infty_tL^{2+\epsilon}}\leq C_0, 
\end{align}
forall $j\in \N$ and $t\in (0,\infty)$, where $C_0$ is some constant depending   on the initial data.
Applying \eqref{dtauu:Lp}, \eqref{gradalpha} with $p=2+\epsilon$ and using \eqref{dtauomega:2+eps}, \eqref{uniform:a} and \eqref{uniform:grada} we deduce
\begin{align}\label{alpha-dtauu}
    \Vert (\alpha^j ,\d_{\tau^j} u^j) \Vert_{L^1_tW^{1,2+\epsilon}} \leq C_0
\end{align}
uniformly in $t\in (0,\infty)$ and $j\in\N$.
 
 \smallbreak
 
Now we turn to the uniform estimates for the stress tensor
\[ 
T_{\mu^j}(u^j,\pi^j) := (\mu^j\ast\eta_j) Su^j - \pi^j Id  .
\]
By  Lemma \ref{lemma:decomp} and the $u$-equation in \eqref{muNS-reg}  we have the following equality 
\[
\nabla^\perp a^j - \nabla \tilde \pi^j=\div T_{\mu^j}(u^j,\pi^j)  = D_t^j u^j , \quad \text{with } a^j=R_{\mu^j\ast\eta_j}\omega^j,\quad  \nabla\tilde\pi^j := \nabla(\pi^j-Q_{\mu^j\ast \eta_j}\omega^j) .
\]
The curl-free part of the above equation (noticing $\div u^j=0$)
\begin{align*}
    -\nabla\tilde \pi^j =  \nabla \Delta^{-1}\div D_t^j u^j 
    =  \nabla \Delta^{-1}\div ((u^j\ast\eta_j) \cdot \nabla u^j)
    =  \nabla \Delta^{-1} \bigl(  (\nabla u^j\ast\eta_j) : (\nabla u^j)^T\bigr)
\end{align*}
implies  from \eqref{uniform:u,a} that for any $t\in (0,\infty)$,
\begin{align}
    \Vert \nabla \tilde \pi^j \Vert_{L^2_tL^2} \lesssim \Vert (u^j\ast\eta_j)\cdot\nabla u^j \Vert_{L^2_tL^2} 
    &\lesssim \Bigl\Vert \Vert u^j \Vert_{L^2}^\frac{\epsilon}{2+\epsilon} \Vert \nabla u^j \Vert_{L^2}^\frac{2}{2+\epsilon} \Vert \nabla u^j \Vert_{L^{2+\epsilon}} \Bigr\Vert_{L^2(0,t)} \nonumber\\
    &\lesssim \Vert u^j \Vert_{L^\infty_t L^2}^\frac{\epsilon}{2+\epsilon} \Vert a^j \Vert_{L^2_tL^2}^\frac{2}{2+\epsilon} \Vert a^j \Vert_{L^\infty_t L^2}^\frac{2}{2+\epsilon} \Vert \nabla a^j \Vert_{L^2_tL^2}^\frac{\epsilon}{2+\epsilon} \leq C_0, \label{pi:est}\\
    \Vert t'^{\frac12} \nabla^2\tilde \pi^j \Vert_{L^2_tL^2}
    \lesssim \|t'^{\frac12} (\nabla u^j\ast\eta_j) : (\nabla u^j)^T \|_{L^2_tL^2}&\lesssim \Vert \nabla u^j \Vert_{L^\infty_t L^2} \Vert t'^{\frac12}\nabla u^j \Vert_{L^2_tL^\infty} \leq C_0.\nonumber
\end{align}
Thus
\begin{align}\label{tensor-j}
\begin{split}
    &\Vert (\div T_{\mu^j}(u^j,\pi^j), t'^{\frac12}\nabla\div T_{\mu^j}(u^j,\pi^j)) \Vert_{L^2_tL^2} \leq \Vert (\nabla^\perp a, t'^{\frac12}\Delta a^j) \Vert_{L^2_tL^2} 
    + \Vert (\nabla \tilde \pi^j, t'^{\frac12}\nabla^2\tilde\pi^j) \Vert_{L^2_tL^2} \leq C_0. \\
\end{split}
\end{align}

\noindent\textbf{Step 3: Convergence.} 
By virtue of the above uniform estimates and standard compactness arguments, there exists a subsequence of the approximation solution sequence, still denoted by  $(\mu^j,u^j,\nabla\pi^j,\tau^j)$, converging to the limit   $(\mu,u,\nabla\tilde\pi, \tau)$ which satisfies the properties stated in Theorem \ref{exthm}.  
Indeed,
\begin{align*}
    \begin{array}{lll}
        \mu^j \overset{\ast}{\rightharpoonup} \mu & \text{in} & L^\infty([0,\infty)\times\R^2;[0,\infty)),  \\
        u^j \overset{\ast}{\rightharpoonup} u & \text{in} & L^\infty([0,\infty);L^2(\R^2)),  \\
        \nabla u^j \rightharpoonup \nabla u & \text{in} & L^2([0,\infty);L^2(\R^2)),  \\
        \tau^j \overset{\ast}{\rightharpoonup} \tau & \text{in} & L^\infty([0,\infty);L^\infty \cap   \dot W^{1,2+\epsilon}(\R^2)),  \\
        \nabla \tilde\pi^j \rightharpoonup \nabla \tilde \pi & \text{in} & L^2((0,\infty);L^2(\R^2)) .
    \end{array} 
\end{align*}
Since $\d_t\mu^j = \div(\mu^j (u\ast \eta_j))$ is uniformly bounded in $L^2_tH^{-1}$ for any $t>0$, the sequence $(\mu^j)$ is relatively compact in $L^p_tL^2_{\mathrm{loc}}$ for any $p\in [1,\infty)$. Consequently, we have $\mu^j \to \mu$ almost everywhere on $[0,\infty)\times\R^2$, which implies that
\[
(\mu^j\ast \eta_j) Su^j \rightharpoonup \mu Su \quad \text{in } L^2_tL^2_{\mathrm{loc}},  \; \forall t>0 .
\]
Furthermore, by the $u$-equation in \eqref{muNS-reg}  and the  uniform estimates in Step 2, $\d_tu^j$ is bounded in $L^2_tL^2$, and hence $u^j$ is relatively compact in $L^p_tL^2_{\mathrm{loc}}$ for all $p\in [1,\infty)$ and $t>0$, which implies that at $u^j\to u$ almost everywhere on $(0,t)\times\R^2$. Together with the fact that $u^j$ is uniformly bounded in $L^4_tL^4$ we conclude that
\[
(u^j\ast\eta_j) \otimes u^j \rightharpoonup u\otimes u \quad \text{in } L^2_tL^2,\; \forall t>0 .
\]
Similarly $\mu^j (u^j\ast\eta_j)  \overset{\ast}{\rightharpoonup} \mu u$,
$\tau^j (u^j\ast\eta_j) \overset{\ast}{\rightharpoonup} \tau u$ in e.g. $L^\infty_t L^2$.
It follows that  $(\mu,u,\nabla\tilde\pi, \tau)$ weakly solves \eqref{muNS}-\eqref{eq:tau}. 
The properties \eqref{exspace}, \eqref{a:reg}, \eqref{alpha:intro0} and \eqref{Tensor} follow from the estimates in Step 2.

\noindent
\textbf{Step 4: Uniqueness.} 
The uniqueness follows from   the $L^1_t\textrm{Lip}$-bound for the velocity field.
More precisely, let $(\mu_i,u_i,\nabla\pi_i,\tau_i)$, $i=1,2$, be two solutions of (\ref{muNS})-\eqref{eq:tau} satisfying (\ref{exspace}). 
For the uniqueness of the viscosity function we make use of Lagrangian coordinates (see also \cite[Section 4]{danchin2019incompressible}). Let the flow $X_i:[0,\infty)\times \R^2 \to \R^2$ denote the flow of $u_i$ defined as $ 
  X_i(t,\xi)=\xi + \int_0^t u_i(t',X_i(t',\xi))dt',$ 
for $i=1,2$. Let $\tilde \mu_i(t,\xi) = \mu_i(t,X_i(t,\xi))$. Then the transport equation (\ref{muNS})$_1$ implies that $\d_t \tilde \mu_i=0$, and thus $\tilde \mu_i(t,\xi)=\mu_i(0,\xi)$ for any $\xi \in\R^2$. 

The uniqueness of the velocity follows from the energy estimate
\begin{align}\label{unique:est}
\begin{split}
    \Vert \delta u \Vert_{L^\infty_tL^2}^2 + \Vert \nabla \delta u \Vert_{L^2_tL^2}^2 \lesssim  \Vert (\delta u)(0) \Vert_{L^2}^2   \exp\Bigl(\int_0^t \Vert \nabla u_1 \Vert_{L^\infty} dt'\Bigr)
\end{split}
\end{align}
for the velocity difference $\delta u = u_2-u_1$.
Indeed, (\ref{unique:est}) follows by testing the difference of the momentum equations (\ref{muNS})$_2$ for $u_1, u_2$ by $\delta u$ and then applying Gronwall's inequality. 

Finally we have  $\nabla \pi_1 = \nabla \pi_2$ from the momentum equations, and $\tau_1=\tau_2$ from the $\tau$-equation.
\end{proof}

\subsection{Proof of Corollary \ref{propthm}}\label{sect:d-prop}
We follow the strategy performed for the density patch problem, cf. \cite[Section 2]{liao2019global} and \cite[Theorem 1.3]{paicu2020striated}, to show the regularity propagation of the viscosity patch problem.
\begin{proof}[Proof of  Corollary \ref{propthm} - \ref{Coro1}]
 
As the assumptions in   Theorem \ref{exthm} are fulfilled for the viscosity patch-type problem stated in   Corollary \ref{propthm} - \ref{Coro1}., there exists a unique global-in-time solution $(\mu, u, \nabla\pi)$ of \eqref{muNS}, satisfying all the estimates in Theorem \ref{exthm}.

The Lipschitz regularity of the velocity field (\ref{exspace})$_3$ guarantees the existence of the flow $X:[0,\infty)\times\R^2 \to \R^2$, defined by the initial value problem
 $X(t,\xi)=\xi + \int_0^t u(t',X(t',\xi))dt', $
such that $X(t,\cdot)\in C^1(\R^2)$ and $\|\nabla X\|_{L^\infty_tL^\infty}\leq \exp(\|\nabla u\|_{L^1_tL^\infty})<\infty$ for all $t\in[0,\infty)$. By classical transport theory we know that the fluid viscosity is given by  
$\mu(t,x)=\mu^+(t,x)1_{D_t}(x)+\mu^-(t,x) 1_{D_t^C}(x)$ with the time-evolved domain $D_t=X(t,D)$ and $\mu^\pm(t,x)=\mu_0^\pm(X^{-1}(t,x))$, where $X^{-1}(t,\cdot)$ denotes the inverse of $X(t,\cdot)$ with respect to the spatial variable. From the fact that $\mu_0^+\in W^{1,2+\epsilon}(\overline{D})$ and $\mu_0^--1\in L^2\cap W^{1,2+\epsilon}(\overline{D^C})$,  we deduce $\mu^+(t,\cdot) \in W^{1,2+\epsilon}(\overline{D_t})$  and $\mu^-(t,\cdot)-1\in L^2\cap W^{1,2+\epsilon}(\overline{D_t^C})$ for $t>0$.

Now we parametrize the boundary $\d D$ of the initial domain  with a function $\gamma_0\in W^{2-\frac{1}{2+\epsilon},2+\epsilon}(\mathbb{S}^1)$ defined as
\begin{align*}
    \gamma_0:\mathbb{S}^1\to \d D, \quad \text{such that} \quad \d_s \gamma_0 (s)=\tau_0(\gamma_0(s)).
\end{align*} 
Then the boundary of $D_t$ can be parametrized by $X(t,\gamma_0):\mathbb{S}^1\to \d D_t$. Differentiating with respect to $s$ yields
\begin{align}\label{dsXg}
    \d_s(X(t,\gamma_0(s)))= \tau_0(\gamma_0(s))\cdot \nabla X(t,\gamma_0(s)) = \tau(t,X(t,\gamma_0(s))) .
\end{align}
Due to the uniform bound of $\tau\in L^\infty_t(L^\infty\cap \dot W^{1,2+\epsilon})$, 
  the trace theorem implies the right hand side of (\ref{dsXg}) lies in $W^{1-\frac{1}{2+\epsilon},2+\epsilon}(\mathbb{S}^1)$. This shows that the parametrization $X(t,\gamma_0)$ 
is contained in $W^{2-\frac{1}{2+\epsilon},2+\epsilon}(\mathbb{S}^1)$. By another application of the trace theorem we conclude that $\d D_t \in W^{2,2+\epsilon}(\R^2)$.

Finally, due to the continuity of $u$ (see \eqref{exspace}) and $T(u,\pi)n$ (see \eqref{Tensor} or \eqref{Tn}) on the interface $\Gamma_t=\d D_t$, the solution $(\mu, u, \nabla\pi)$ also solves \eqref{mu2pNS} with $\Omega_t^+=D_t$, $\Omega_t^-=D_t^C$.
\end{proof}


\begin{proof}[Proof of  Corollary \ref{propthm} - \ref{Coro2}]
The assumptions and hence the results of Theorem \ref{exthm} hold.
The propagation of the viscosity coefficient $\nabla \mu\in L^\infty_tL^q$ follows immediately from  the Lipschitz regularity of the velocity field (\ref{exspace})$_3$ and  the evolution equation for $\nabla \mu$:
$\d_t\nabla\mu + u\cdot\nabla\mu = -\nabla u\cdot\nabla\mu$. 

Now we apply  the gradient to the velocity equation \eqref{muNS}$_2$ and then take the $L^2$-inner product with $\nabla u$ and use integration by parts (similarly as for the derivation of \eqref{u:energy}), to derive
\begin{align*}
    \frac12 \frac{d}{dt} \int_{\R^2} \vert \nabla u \vert^2 dx + \int_{\R^2} \mu \vert S \nabla u\vert^2 dx
    &= \int_{\R^2} [D_t,\nabla]u  : \nabla u dx + \int_{\R^2} Su : ((\nabla\mu\cdot\nabla)\nabla u) dx 
    \\
    &\leq \int_{\R^2}(\nabla u\cdot\nabla u) : \nabla u dx 
    +\Vert \nabla \mu \Vert_{L^q} \Vert \nabla u \Vert_{L^\frac{2q}{q-2}} \Vert \nabla^2 u \Vert_{L^2} 
    \\
    &\leq \|\nabla u\|_{L^\infty}\|\nabla u\|_{L^2}^2
    + \Vert \nabla \mu \Vert_{L^q} \Vert \nabla u \Vert_{L^2}^{1-\frac2q} \Vert \nabla^2 u\Vert_{L^2}^{1+\frac2q}. 
\end{align*} 
We deduce \eqref{gradu:H1} from the fact $\|S\nabla u\|_{L^2}\sim \|\nabla^2 u\|_{L^2}$ and the estimates $\|\nabla\mu\|_{L^\infty_tL^q}, \|\nabla u\|_{L^2_tL^2}, \|\nabla u\|_{L^1_tL^\infty}<\infty$ for all $t>0$, together with  Young's and Gronwall's inequalities.
\end{proof}


Next, we sketch the proof of the local-in-time well-posedness result for the Boussinesq equations \eqref{eq:theta} without heat conduction, coupled with \eqref{eq:tau}.
Recall the main consequences of   Step I, II and the main inequality, which were used in the proof of  $L^1_t\textrm{Lip}$-estimate (Step III), in Subsection \ref{sect:L1Lip} for the system \eqref{muNS}-\eqref{eq:tau}:
\begin{itemize}
    \item Step I. Estimates for $\|a\|_{L^1_t W^{1,2+\epsilon}}$, $\|t'^{\frac12}a\|_{L^2_t W^{1,2+\epsilon}}$, $\|a\|_{L^1_tL^\infty}$ in terms of 
   \begin{equation}\label{tildeV,V}
   \tilde V(t)=V(t)\exp(C\|t'^{\frac12}\nabla u\|_{L^2_tL^\infty}),
   \hbox{ with }V(t)=\exp(C\|\nabla u\|_{L^1_tL^\infty}).
   \end{equation}
    \item Step II.
    Estimates  which follow from Corollary \ref{coro}
    \begin{align}\label{nablau,L1tLinfty}
        &\|\nabla u\|_{L^1_tL^\infty}
        \leq \|a\|_{L^1_t L^{2+\epsilon}}^{\frac{\epsilon}{2+\epsilon}}
        \Bigl( \|\nabla a\|_{L^1_t L^{2+\epsilon}}
        +\|(\nabla\taub, \d_{\taub}\mu)\|_{L^\infty_t L^{2+\epsilon}}\|(\nabla u, a)\|_{L^1_tL^\infty}\Bigr)^{\frac{2}{2+\epsilon}},\\
        &\|t'^{\frac12}\nabla u\|_{L^2_tL^\infty}
        \leq \|t'^{\frac12}a\|_{L^2_t L^{2+\epsilon}}^{\frac{\epsilon}{2+\epsilon}}
        \Bigl( \|t'^{\frac12}\nabla a\|_{L^2_t L^{2+\epsilon}}
        +\|(\nabla\taub, \d_{\taub}\mu)\|_{L^\infty_t L^{2+\epsilon}}\|t'^{\frac12}(\nabla u, a)\|_{L^2_tL^\infty}\Bigr)^{\frac{2}{2+\epsilon}},
        \nonumber
    \end{align}
    where by use of the transport equations $\eqref{muNS}_1$, \eqref{eq:tau} for $\mu$, $\tau$ respectively,
    \begin{align}\label{nablataub,dtaubmu}
        \|(\nabla\bar\tau, \d_{\taub}\mu)\|_{L^\infty_t L^{2+\epsilon}}
        \lesssim
        (\|(\nabla\bar\tau_0, \d_{\taub_0}\mu_0)\|_{L^{2+\epsilon}}+\|\nabla a\|_{L^1_t L^{2+\epsilon}})\exp(C\|a\|_{L^1_tL^\infty})V(t).
    \end{align}
    \item Step III. Inequality for $A(t)=\|\nabla u\|_{L^1_tL^\infty}+\|t'^{\frac12}\nabla u\|_{L^2_tL^\infty}$ of type $A(t)\leq C\sigma \exp(CA(t)+\tilde\sigma\exp(CA(t)))$, 
    with $\sigma, \tilde\sigma$ depends only on the initial data.
\end{itemize}
As the estimates in Step II hold  universally, it suffices to derive the $W^{1,2+\epsilon}$-estimates for $a$ in Step I, such that the 
  the bootstrap argument in Step III works. 
  Different as for the system \eqref{muNS} where we derived directly the energy estimates for $a$ in Step I, for the Boussinesq equation \eqref{eq:theta} we derive below the $H^1$-energy estimate directly for $a_\vartheta=a-\mathcal{R}_{-1}\vartheta$, which takes into account the buoyancy force $\vartheta e_2$. 
\begin{proof}[Proof of   Corollary \ref{propthm} - \ref{Coro3}]
We aim to establish a priori estimates for 
\begin{align*}
  \Vert \vartheta \Vert_{L^\infty_tL^1\cap L^r}
  +\Vert u \Vert_{L^\infty_tL^2\cap L^2_t \dot H^1} +\Vert a \Vert_{L^\infty_t L^2 \cap L^2_t \dot H^1} + \Vert t'^{\frac12} \nabla a \Vert_{L^\infty_tL^2 \cap L^2_t \dot H^1} .
\end{align*} 
Firstly, the transport equation with divergence-free velocity vector for the temperature  \eqref{eq:theta}$_1$   yields 
\begin{align}\label{B:theta-est}
    \Vert \vartheta \Vert_{L^\infty_tL^{r_1}} = \Vert \vartheta_0 \Vert_{L^{r_1}}, \quad \forall r_1\in [1,r]\supset [1,2+\epsilon].
\end{align} 

Compared with the system \eqref{muNS}, there is an additional term $\vartheta e_2$ on the right hand side of the velocity equation   $\eqref{eq:theta}_2$.
Consequently, the vorticity equation \eqref{intro:omega} is replaced by 
\begin{align}\label{B:omega-eq}
    \d_t\omega +u\cdot\nabla \omega - \Delta a = \d_1 \vartheta ,\hbox{ with }\omega=\nabla^\perp\cdot u,\quad  a=R_\mu\omega,
\end{align}
which is the application of  the curl operator to the velocity equation \eqref{eq:theta}$_2$.
We follow the proofs of Proposition \ref{u-decay:prop} and Proposition \ref{energy:prop} to derive the energy estimates for $u$ and $a$.
Taking the $L^2$-inner product between the velocity equation \eqref{eq:theta}$_2$ and $u$ we derive by Cauchy-Schwarz inequality, Young's inequality and \eqref{B:theta-est}
\begin{align}\label{B:u-en}
    \Vert u \Vert_{L^\infty_tL^2}^2 + \Vert \nabla u \Vert_{L^2_tL^2}^2 \lesssim_{\mu_\ast}  
    \Vert (u_0, t\,\vartheta_0) \Vert_{L^2}^2.
\end{align}
Next, using the same arguments as in the proof for \eqref{a:energy-low} we deduce from the vorticity equation
\eqref{B:omega-eq} the following estimate
\begin{align}\label{B:a-en}
    \Vert a \Vert_{L^\infty_tL^2}^2 + \Vert \nabla a \Vert_{L^2_tL^2}^2 \lesssim_{\mu_\ast,\mu^\ast} 
    \Vert (\omega_0, t^{\frac12}\vartheta_0) \Vert_{L^2}^2   V(t),\quad V(t)=\exp(C\|\nabla u\|_{L^1_t L^\infty}).
\end{align}

\smallbreak 

\noindent\textbf{$H^1$-estimate for $\Gamma$.}
To obtain higher order energy estimates for $a$, motivated by e.g. \cite{HmidiKeraani}, we define the quantity 
$$\Gamma=\omega-R_\mu^{-1} \mathcal{R}_{-1}\vartheta,
\quad \hbox{ with }\mathcal{R}_{-1}= \d_1 (-\Delta)^{-1}.$$ 
From the energy estimate \eqref{B:a-en} above and the relation $\nabla R_\mu \Gamma = \nabla a - \nabla \mathcal{R}_{-1}\vartheta=\nabla a_\vartheta$ we deduce from \eqref{B:theta-est}
\begin{align}\label{Gamma-low}
    \Vert \nabla R_\mu \Gamma \Vert_{L^2_tL^2}^2 \lesssim \Vert \nabla a \Vert_{L^2_tL^2}^2 +  \Vert \vartheta \Vert_{L^2_tL^2}^2 \lesssim_{\mu_\ast,\mu^\ast} \Vert (\omega_0, t^{\frac12}\vartheta_0) \Vert_{L^2}^2 V(t).
\end{align} 

Now we derive the $\dot H^1$-energy estimate for $R_\mu\Gamma$, similar as in the proof for \eqref{a:tenergy-high}. 
Applying the operator $R_\mu^{-1}\mathcal{R}_{-1}$ to the temperature equation \eqref{eq:theta}$_1$,
and then subtracting this equation from the vorticity equation \eqref{B:omega-eq} we obtain (noticing $\Delta a+\d_1\vartheta=\Delta(R_\mu\Gamma)$ and $\frac{D}{Dt}=\d_t+u\cdot\nabla$)
\begin{align}\label{Gamma:eq}
    \d_t \Gamma +u\cdot\nabla \Gamma - \Delta R_\mu \Gamma 
    &=   [R_\mu^{-1}\mathcal{R}_{-1}, \frac{D}{Dt}] \vartheta.
\end{align}
We take the $L^2$-inner product between \eqref{Gamma:eq} and $R_\mu \dot\Gamma =R_\mu \frac{D}{Dt}\Gamma$ and perform similar calculations as for \eqref{a:tenergy-high} to derive
\begin{align}
    &\frac12 \frac{d}{dt} \int_{\R^2} \vert \nabla R_\mu \Gamma \vert^2 dx + \int_{\R^2} \mu \Bigl(((R_2R_2-R_1R_1) \dot \Gamma)^2+(2R_1R_2\dot \Gamma)^2\Bigr) dx \nonumber\\
    &=  \int_{\R^2} \bigl( [R_\mu^{-1}\mathcal{R}_{-1},\frac{D}{Dt}] \vartheta\bigr)\cdot ( R_\mu \dot\Gamma) dx - \int_{\R^2} \nabla R_\mu \Gamma \cdot\nabla u \cdot \nabla R_\mu \Gamma dx + \int_{\R^2} (\Delta R_\mu\Gamma)  \cdot ( [R_\mu,\frac{D}{Dt}]\Gamma) dx.
   \label{B:high-en}
\end{align}
Notice that applying \eqref{comm} implies the estimate for the commutator term $[R_\mu,\frac{D}{Dt}]$:
\begin{align}\label{comm:Rmu,Dt}
    \|[R_\mu, \frac{D}{Dt}]f\|_{L^2}\lesssim_{\mu^\ast, p_1, p_2} \|\nabla u\|_{L^{p_2}}\|f\|_{L^{p_1}},\quad \frac{1}{p_1}+\frac{1}{p_2}=\frac12,\quad p_1\in [2,\infty),\, p_2\in (2,\infty].
\end{align}
Hence we bound the commutator
\begin{align*}
    \|[R_\mu,\frac{D}{Dt}]\Gamma\|_{L^2}
    &\leq \|[R_\mu,\frac{D}{Dt}]\omega\|_{L^2}
    +\|[R_\mu,\frac{D}{Dt}] R_\mu^{-1}\mathcal{R}_{-1}\vartheta\|_{L^2}
    \\
    &\lesssim \|\nabla u\|_{L^\infty}\|\omega\|_{L^2}
    +\|\nabla u\|_{L^{\frac{2(2+\epsilon)}{\epsilon}}}
    \|R_\mu^{-1}\mathcal{R}_{-1}\vartheta\|_{L^{2+\epsilon}},
\end{align*}
where by Sobolev embedding we can bound
\[
\Vert R_\mu^{-1} \mathcal{R}_{-1} \vartheta \Vert_{L^{2+\epsilon}} \lesssim \Vert \mathcal{R}_{-1} \vartheta \Vert_{L^{2+\epsilon}} \lesssim \Vert \nabla \mathcal{R}_{-1} \vartheta \Vert_{L^{\frac{2(2+\epsilon)}{4+\epsilon}}} \lesssim \Vert \vartheta \Vert_{L^{\frac{2(2+\epsilon)}{4+\epsilon}}}.
\]
Similarly the commutator
\begin{align*}
 [R_\mu^{-1}\mathcal{R}_{-1}, \frac{D}{Dt}] \vartheta
 &   =R_\mu^{-1}[\mathcal{R}_{-1},\frac{D}{Dt}]\vartheta + [R_\mu^{-1},\frac{D}{Dt}]\mathcal{R}_{-1}\vartheta 
    \\
&     = R_\mu^{-1}   (\mathcal{R}_{-1}\div(u\vartheta)-u\cdot\nabla \mathcal{R}_{-1} \vartheta)- R_\mu^{-1}[R_\mu,\frac{D}{Dt}]R_\mu^{-1}\mathcal{R}_{-1} \vartheta
\end{align*}
can be bounded by (recalling the diffeomorphism of $R_\mu$ in $L^{2+\epsilon}$)  
\begin{align*}
    \|[R_\mu^{-1}\mathcal{R}_{-1}, \frac{D}{Dt}] \vartheta\|_{L^2}
    \lesssim \|u\|_{L^{ \infty}}\|\vartheta\|_{L^{2}}
    +\Vert \nabla u \Vert_{L^\frac{2(2+\epsilon)}{\epsilon}} \Vert \vartheta \Vert_{L^{\frac{2(2+\epsilon)}{4+\epsilon}}}.
\end{align*} 
To conclude, we obtain together with Young's inequality and $\Delta R_\mu \Gamma = \dot \Gamma + [R_\mu^{-1}\mathcal{R}_{-1}, \frac{D}{Dt}] \vartheta$
\begin{align*}
  & \frac{d}{dt} \Vert \nabla R_\mu \Gamma \Vert_{L^2}^2 +   \Vert (\dot \Gamma, \Delta R_\mu\Gamma) \Vert_{L^2}^2 
    \lesssim _{\mu_\ast,\mu^\ast}
    \|[R_\mu^{-1}\mathcal{R}_{-1}, \frac{D}{Dt}] \vartheta\|_{L^2}^2
    +\Vert \nabla u \Vert_{L^\infty} \Vert \nabla R_\mu \Gamma \Vert_{L^2}^2
    + \|[R_\mu,\frac{D}{Dt}]\Gamma\|_{L^2}^2
    \\
    &\qquad 
    \lesssim_{\mu_\ast,\mu^\ast} \Vert \nabla u \Vert_{L^\infty} \Vert \nabla R_\mu \Gamma \Vert_{L^2}^2 
    +\|\nabla u\|_{L^\infty}^2\|\omega\|_{L^2}^2+\|u\|_{L^{ \infty}}^2\|\vartheta\|_{L^{2}}^2  
     + \Vert \nabla u \Vert_{L^\frac{2(2+\epsilon)}{\epsilon}}^2  \Vert \vartheta \Vert_{L^{\frac{2(2+\epsilon)}{4+\epsilon}}}^2  .
\end{align*}

Next, we multiply by $t$ and make use of Gronwall's inequality and interpolation inequality to obtain (recalling the definition \eqref{tildeV,V} for $\tilde V$)
\begin{align*}
    &\Vert t^\frac12 \nabla R_\mu \Gamma(t) \Vert_{L^2}^2 +  \Vert t'^{\frac12} (\dot\Gamma, \Delta R_\mu\Gamma) \Vert_{L^2_t L^2}^2 
    \\
    &\lesssim_{\mu_\ast,\mu^\ast} \Bigl(\Vert \nabla R_\mu \Gamma\Vert_{L^2_tL^2}^2 +\|\omega\|_{L^\infty_tL^2}^2+ \int_0^t \Bigl[t' \|u\|_{L^{ \infty}}^2\|\vartheta\|_{L^{2}}^2
    + t'\Vert \nabla u \Vert_{L^\frac{2(2+\epsilon)}{\epsilon}}^2  \Vert \vartheta \Vert_{L^{\frac{2(2+\epsilon)}{4+\epsilon}}}^2  \Bigr] dt'\Bigr)\tilde  V(t)
    \\
    & \lesssim_{\mu_\ast,\mu^\ast} 
    (\Vert \nabla R_\mu \Gamma\Vert_{L^2_tL^2}^2+\|a\|_{L^\infty_t L^2}^2) \tilde V(t)
    +   \Vert u \Vert_{L^\infty_tL^2}  \Vert \nabla u \Vert_{L^1_tL^\infty}  \Vert t'^{\frac12}\vartheta \Vert_{L^\infty_tL^{2}}^2 \tilde V(t)
    \\
    &\qquad 
    +   \Vert t'^{\frac12} \nabla u \Vert_{L^2_tL^\infty}^\frac{4}{2+\epsilon}
    \Vert \nabla u \Vert_{L^2_tL^2}^\frac{2\epsilon}{2+\epsilon}
    \Vert t'^{\frac12}\vartheta \Vert_{L^\infty_tL^1} ^\frac{4 }{2+\epsilon}
    \|t'\vartheta\|_{L^\infty_t L^2}^{\frac{2\epsilon}{2+\epsilon}}V(t).
\end{align*}
Inserting the estimates  \eqref{B:theta-est}, \eqref{B:u-en}, \eqref{B:a-en} and \eqref{Gamma-low}, we conclude the time weighted $\dot H^1$-estimate for $R_\mu\Gamma$
\begin{align}\label{Gamma-en0}
 &   \Vert t'^{\frac12} \nabla R_\mu \Gamma \Vert_{L^\infty_tL^2}^2 
    + \Vert t'^{\frac12}(\dot \Gamma, \Delta R_\mu\Gamma) \Vert_{L^2_tL^2}^2
    \\
    &\lesssim  \Bigl( \Vert (\omega_0, t^{\frac12}\vartheta_0) \Vert_{L^2}^2  +  \| (u_0, t\vartheta_0)\|_{L^2} \| t^{ \frac{1}{2}}\vartheta_0\|_{L^{2}}^2  
   +  \Vert  (u_0, t\vartheta_0) \Vert_{ L^2}^\frac{2\epsilon}{2+\epsilon}
   \|t^{\frac12}\vartheta_0\|_{L^1}^{\frac{4}{2+\epsilon}}
   \|t\vartheta_0\|_{L^2}^{\frac{2\epsilon}{2+\epsilon}}\Bigr)\tilde V(t).\nonumber
\end{align}

\noindent\textbf{$W^{1,2+\epsilon}(\R^2)$-estimate for $a$.}
We set
\begin{align*}
     & \sigma_0  = \Vert  u_0   \Vert_{L^2}, \quad 
    \sigma_\vartheta 
    = \sigma_\vartheta(t)  = \Vert t^\frac12 \vartheta_0 \Vert_{L^1}+\Vert t  \vartheta_0 \Vert_{L^2} 
     + \Vert t^{\frac32-\frac{1}{2+\epsilon}} \vartheta_0 \Vert_{L^{2+\epsilon}} ,
    \\
    &\tilde\sigma_0=\tilde\sigma_0(t) =\sigma_0+\sigma_\vartheta,
    \\
   & \sigma_1  = \Vert \omega_0 \Vert_{L^2} +  \Vert (\nabla \taub_0 , \d_{\taub_0} \mu_0 )\Vert_{L^{2+\epsilon}} ^\frac{2+\epsilon}{\epsilon}, \\
    &\tilde \sigma_1 = \tilde \sigma_1(t)  = \sigma_1
    +t^{-\frac12}\sigma_\vartheta
    (1+\sigma_0 ^{\frac12}   +\sigma_\vartheta^{\frac12} \bigr), \hbox{ i.e. }t^{\frac12}\tilde\sigma_1=t^{\frac12}\sigma_1+\sigma_\vartheta(1+\sigma_0^{\frac12}+\sigma_\vartheta^{\frac12}) .
\end{align*}
Notice that the Boussinesq equations \eqref{eq:theta} are invariant under the following scaling:
\begin{align*}
    (\vartheta_\lambda, u_\lambda, \pi_\lambda)(t,x)
    =(\lambda^{-3}\vartheta, \lambda^{-1} u, \lambda^{-2}\pi)(\lambda^{-2}t, \lambda^{-1}x),\quad \lambda>0,
\end{align*}
and hence $\sigma_0, \sigma_\vartheta, t^{\frac12}\sigma_1, t^{\frac12}\tilde\sigma_1, V(t), \tilde V(t)$ are also scaling invariant.
Let us recall the estimates \eqref{B:u-en}, \eqref{B:a-en} and \eqref{Gamma-en0} we established above (noticing $\sigma^{\frac{\epsilon}{2+\epsilon}}\lesssim 1+\sigma^{\frac12}$):
\begin{align*}
    \Vert a \Vert_{L^2_tL^2} \leq C  \tilde\sigma_0, \quad 
    \Vert (\nabla a,\nabla R_\mu \Gamma) \Vert_{L^2_tL^2} \leq C  \tilde \sigma_1 V(t), \quad 
    \Vert t'^{\frac12}\Delta R_\mu \Gamma \Vert_{L^2_tL^2} \leq C  \tilde \sigma_1 \tilde V(t) .
\end{align*}
Using interpolation and H\"older's inequality we estimate
\begin{align*}
    \Vert a \Vert_{L^1_tL^{2+\epsilon}} &\lesssim t^\frac12 \Vert a \Vert_{L^2_tL^{2+\epsilon}}
    \lesssim t^\frac12 \Vert a \Vert_{L^2_tL^2}^\frac{2}{2+\epsilon} \Vert \nabla a \Vert_{L^2_tL^2}^\frac{\epsilon}{2+\epsilon} 
    \lesssim t^\frac{1}{2+\epsilon} \tilde \sigma_0^\frac{2}{2+\epsilon} (t^\frac12 \tilde \sigma_1)^\frac{\epsilon}{2+\epsilon} V(t), \\
    %
    \Vert \nabla a \Vert_{L^1_tL^{2+\epsilon}} &\leq \Vert \nabla R_\mu \Gamma \Vert_{L^1_tL^{2+\epsilon}} + \Vert \nabla \mathcal{R}_{-1} \vartheta \Vert_{L^1_tL^{2+\epsilon}} 
    \lesssim \Vert \nabla R_\mu \Gamma \Vert_{L^2_tL^2}^\frac{2}{2+\epsilon} \Vert t'^{\frac12} \Delta R_\mu \Gamma \Vert_{L^2_tL^2}^\frac{\epsilon}{2+\epsilon} t^{\frac12-\frac{\epsilon}{2(2+\epsilon)}} + t \Vert \vartheta_0 \Vert_{L^{2+\epsilon}} \\
    &\lesssim t^{-\frac{\epsilon}{2(2+\epsilon)}} \bigl(   t^\frac12\tilde \sigma_1  + \sigma_\vartheta \bigr) \tilde V(t)\lesssim t^{-\frac{\epsilon}{2(2+\epsilon)}}   \bigl(t^\frac12\tilde \sigma_1\bigr)   \tilde V(t), \\
    %
    \Vert a \Vert_{L^1_tL^\infty} &\lesssim \|a\|_{L^1_tL^{2+\epsilon}}^{\frac{\epsilon}{2+\epsilon}} \|\nabla a\|_{L^1_t L^{2+\epsilon}}^{\frac{2}{2+\epsilon}}
     \lesssim \tilde\sigma_0^\frac{2\epsilon}{(2+\epsilon)^2} (t^\frac12\tilde\sigma_1)^{\frac{\epsilon^2}{(2+\epsilon)^2}+\frac{2}{2+\epsilon}}
     \tilde V(t),
\end{align*}
and similarly for the quantities $\Vert t'^{\frac12} a \Vert_{L^2_tL^{2+\epsilon}}$, $\Vert t'^{\frac12} \nabla a \Vert_{L^2_tL^{2+\epsilon}}$ and $\Vert t'^{\frac12} a \Vert_{L^2_tL^\infty}$.

\smallbreak 

\noindent\textbf{Conclusion.}
Recalling \eqref{nablataub,dtaubmu}: 
\begin{align*}
\|(\nabla\taub, \d_{\taub}\mu)\|_{L^\infty_t L^{2+\epsilon}}
&\leq   t^{-\frac{\epsilon}{2(2+\epsilon)}} 
\bigl( (t^\frac12 \sigma_1)^\frac{\epsilon}{2+\epsilon}
+t^{\frac12}\tilde\sigma_1 \bigr) \tilde V(t)
\exp (C\|a\|_{L^1_tL^\infty}),
\end{align*}
and \eqref{nablau,L1tLinfty}:
\begin{align*}
    &\Vert \nabla u \Vert_{L^1_tL^\infty} + \Vert t'^{\frac12} \nabla u \Vert_{L^2_tL^\infty}
  \\
&     \lesssim \Bigl(t^\frac{1}{2+\epsilon} \tilde\sigma_0^\frac{2}{2+\epsilon} (t^\frac12 \tilde\sigma_1)^\frac{\epsilon}{2+\epsilon}\Bigr)^\frac{\epsilon}{2+\epsilon} 
    \Bigl(t^{-\frac{\epsilon}{2(2+\epsilon)}} \bigl((t^\frac12 \sigma_1)^\frac{\epsilon}{2+\epsilon}
    +t^{\frac12}\tilde\sigma_1\bigr)
    \Bigr)^\frac{2}{2+\epsilon}\tilde V(t)   \exp \bigl(C\tilde\sigma_0^\frac{2\epsilon}{(2+\epsilon)^2} (t^\frac12\tilde\sigma_1)^{\frac{\epsilon^2}{(2+\epsilon)^2}+\frac{2}{2+\epsilon}}
    \tilde V(t)\bigr) \\ 
    &\leq C\tilde\sigma_0^{\frac{2\epsilon}{(2+\epsilon)^2}} (t^{\frac12}\tilde\sigma_1 )^{\frac{\epsilon}{2+\epsilon}}  \tilde V(t)\exp \bigl(C\tilde\sigma_0^\frac{2\epsilon}{(2+\epsilon)^2} (t^\frac12\tilde\sigma_1)^{\frac{\epsilon^2}{(2+\epsilon)^2} +\frac{2}{2+\epsilon}}
    \tilde V(t)\bigr) .
\end{align*}
With $ 
A(t)=\Vert \nabla u \Vert_{L^1_tL^\infty} + \Vert t'^{\frac12} \nabla u \Vert_{L^2_tL^\infty}, $ 
the above shows that
\[
A(t) \leq C\tilde\sigma_0^{\frac{2\epsilon}{(2+\epsilon)^2}} (t^{\frac12}\tilde\sigma_1 )^{\frac{\epsilon}{2+\epsilon}} \exp \bigl(CA(t)+C\tilde \sigma_0^\frac{2\epsilon}{(2+\epsilon)^2} (t^\frac12\tilde\sigma_1)^{\frac{\epsilon^2}{(2+\epsilon)^2} +\frac{2}{2+\epsilon}}
e^{CA(t)}\bigr).
\]
We now choose $T>0$ such that  the following smallness condition is satisfied
\begin{align}\label{B:cond-1}
    2C^2\tilde\sigma_0(T)^{\frac{2\epsilon}{(2+\epsilon)^2}} (T^{\frac12}\tilde\sigma_1 (T))^{\frac{\epsilon}{2+\epsilon}} + C \sqrt{e} \tilde\sigma_0(T)^\frac{2\epsilon}{(2+\epsilon)^2} (T^\frac12\tilde\sigma_1(T))^{\frac{\epsilon^2}{(2+\epsilon)^2} +\frac{2}{2+\epsilon}}
    \leq \frac12 ,
\end{align}
so that we obtain via a bootstrap argument the uniform bound
\[
A(T) \leq 2C \tilde\sigma_0(T)^{\frac{2\epsilon}{(2+\epsilon)^2}} (T^{\frac12}\tilde\sigma_1 (T))^{\frac{\epsilon}{2+\epsilon}} \leq \frac{1}{2C}.
\]
Observe that if $T$ satisfies \eqref{B:cond} with a sufficiently small constant $c_1$ and suitable exponents $\theta_1^B, \theta_2^B, \theta_3^B, \theta_4^B$ (all of which depend only on the constant $C$ from above or $\epsilon$, and hence only on $\mu_\ast,\mu^\ast$), then the smallness condition \eqref{B:cond-1} is fulfilled. 

Finally, following the proof of Theorem \ref{exthm} in Subsection \ref{sect:proofs} we complete the proof of Corollary \ref{propthm} - \ref{Coro3}.
\end{proof}


Lastly,  as for the Boussinesq equations \eqref{eq:theta} above, it suffices to establish  the $W^{1,2+\epsilon}(\R^2)$-estimates for $a$, which may follow from the energy estimates 
for the density-dependent incompressible Navier-Stokes system (\ref{NS}), to conclude the fourth statement of Corollary \ref{propthm}.
\begin{proof}[Proof of   Corollary \ref{propthm} - \ref{d-propthm}]
Firstly, since the density function $\rho(t,x)$ and the viscosity coefficient $\mu(t,x)=\mu_\rho(\rho(t,x))$ both satisfy the free transport equation, the initial lower and upper bounds are preserved by the Navier-Stokes flow a priori
\begin{align*}
    0<\rho_\ast\leq \rho(t,x)\leq \rho^\ast,
    \quad 0<\mu_\ast\leq \mu(t,x)\leq \mu^\ast.
\end{align*}
In the following the constant $C$ depends only on the four positive constants $\rho_\ast, \rho^\ast, \mu_\ast, \mu^\ast$ and $\|\mu_\rho'\|_{L^\infty([\rho_\ast,\rho^\ast])}$, which may vary from line to line.

With appropriately adapted modifications,  we set  as in Subsection \ref{sect:L1Lip} 
\begin{align*}
    \sigma_0 &= \Vert u_0 \Vert_{L^2} + \Vert \rho_0-1 \Vert_{L^2}\Vert \nabla u_0 \Vert_{L^2}, \\ 
    \sigma_1 &= \Vert \nabla u_0 \Vert_{L^2} +  \Vert (\nabla \taub_0, \d_{\taub_0}\mu_0 ) \Vert_{L^{2+\epsilon}}  ^\frac{2+\epsilon}{\epsilon}, \\ 
    \sigma_{-1} &= \Vert u_0 \Vert_{\dot H^{-1}} + \Vert \rho_0-1 \Vert_{L^2} \Vert u_0 \Vert_{L^2} , \\
    V(t)&=\exp(C\Vert \nabla u \Vert_{L^1_tL^\infty}), \quad 
    \tilde V(t)= \exp \bigl(C(\Vert \nabla u \Vert_{L^1_tL^\infty}+ \Vert t'^{\frac12}\nabla u \Vert_{L^2_tL^\infty})\bigr).
\end{align*}
The energy estimates for $u$ in Proposition \ref{u-decay:prop} are still valid for equations (\ref{NS}):
\begin{align}
\Vert\sqrt{\rho} u \Vert_{L^\infty_tL^2}+ \Vert \nabla u \Vert_{L^2_tL^2} &
    \leq C(\mu_\ast)\sigma_0 ,
    \label{u:energy2}\\
    \Vert \langle t \rangle^{\delta}u\Vert_{L^2} + \Vert \langle t' \rangle^{\delta} \nabla u \Vert_{L^2_tL^2} 
    &\leq C(\mu_\ast, \mu^\ast) (\sigma_0+\sigma_{-1}) V(t) e^{C\sigma_0^2} 
    , \label{u:decay+2}
\end{align}
where   we have taken $ \delta\in (\frac{1}{2+\epsilon}, \frac{4+\epsilon}{4(2+\epsilon)})$,   with $\epsilon\in (0,2]$ given in Lemma \ref{Rmu-inv:prop}, as in \eqref{delta}. 
Indeed, \eqref{u:energy2} is the classical energy estimates by taking the $L^2$-inner product between $u$-equation and $u$ itself, see e.g. \cite{lions1996mathematical}.
The estimate  \eqref{u:decay+2} was also known   in e.g. \cite{abidi2015global2D, wiegner1987decay}, and we sketch its proof at the  end of Appendix \ref{sect:u-en} with  minor changes in the proof of Proposition \ref{u-decay:prop}. 

\smallbreak 

\noindent\textbf{Higher-order energy estimates.}
We claim the following estimates (similar as the energy estimates in Proposition \ref{energy:prop})
\begin{align}
    \Vert \nabla u \Vert_{L^\infty_tL^2} + \Vert \dot u \Vert_{L^2_tL^2} &\leq C(\mu_\ast,\mu^\ast, \rho_\ast, \rho^\ast) \sigma_1 e^{C\sigma_0^2} V(t) \label{a:en-low2}\\
    \Vert t'^{\frac12} \nabla u \Vert_{L^\infty_tL^2} + \Vert t'^{\frac12} \dot u \Vert_{L^2_tL^2} &\leq C(\mu_\ast,\mu^\ast, \rho_\ast, \rho^\ast) \sigma_0 e^{C\sigma_0^2} V(t), \label{a:ten-low2}\\
    \Vert t'^{\frac12} \sqrt{\rho}\dot u \Vert_{L^\infty_tL^2} + \Vert t'^{\frac12} \dot \omega \Vert_{L^2_tL^2} &\leq C(\mu_\ast,\mu^\ast,\rho_\ast,\rho^\ast) \sigma_1(1+\sigma_0) \tilde V(t) e^{C\sigma_0^2}, \label{a:ten-highNS} \\
    \Vert t'^{\frac12+\delta }a \Vert_{L^\infty_tL^2} + \Vert t'^{\frac12+\delta } \nabla a \Vert_{L^2_tL^2} &\leq C(\mu_\ast,\mu^\ast, \rho_\ast, \rho^\ast)(\sigma_0+\sigma_{-1})e^{C\sigma_0^2} V(t). \label{a:ten-high+}
\end{align}
We only explain the main ideas.  
  \eqref{a:en-low2} is established in e.g. \cite{abidi2015global2D}: taking the $L^2$ inner product of (\ref{NS})$_2$ with $\dot u$, performing integration by parts, using the duality between 
$$\pi =-(-\Delta)^{-1} \div \div(\mu Su)+ (-\Delta)^{-1} \div(\rho \dot u)\in L^2+\hbox{BMO}\,\,  \hbox{ and }\,\,  \div\dot u=\d_i u_j\d_j u_i\in L^2\cap \hbox{Hardy space }\mathcal{H}^1,$$
and finally applying Young's inequality and then Gronwall's inequality yield \eqref{a:en-low2}. 

The time-weighted version \eqref{a:ten-low2} of \eqref{a:en-low2} follows similarly.

The decay estimate \eqref{a:ten-high+} follows from \eqref{u:decay+2}.

We now show time-weighted $L^2$-estimate for $\dot u$ in \eqref{a:ten-highNS}.
With the decomposition \eqref{divmuSu:decomp} the momentum equation \eqref{NS}$_2$ reads
\begin{equation}\label{inNS:a}
\rho \dot u - \nabla^\perp a + \nabla \tilde\pi=0,\quad \tilde\pi=\pi-b.
\end{equation}
We apply $\frac{D}{Dt}$ onto both sides, take the $L^2$-innder product with $\dot u$ and use the transport equation $\frac{D}{Dt}\rho=0$ to derive
\[
\int_{\R^2} \rho \frac{D}{Dt}\dot u \cdot \dot u dx - \int_{\R^2} \frac{D}{Dt}\nabla^\perp a \cdot \dot u dx + \int_{\R^2} \frac{D}{Dt} \nabla \tilde\pi \cdot \dot u dx =0 .
\]
In the following we reformulate each integral one by one. 
\begin{itemize}
    \item By \eqref{NS}$_1$ the first integral is equal to
$ \frac12 \frac{d}{dt} \int_{\R^2} \rho \vert \dot u \vert^2 dx . $
\item 
The second integral can be rewritten as
\begin{align*}
&- \int_{\R^2} \frac{D}{Dt}\nabla^\perp a \cdot \dot u dx 
 =  - \int_{\R^2} [\frac{D}{Dt},\nabla^\perp]a \cdot \dot u dx-\int_{\R^2} \nabla^\perp \Bigl(R_\mu \dot \omega + [\frac{D}{Dt},R_\mu]\omega\Bigr) \cdot \dot u dx 
 \\
&\quad =\int_{\R^2} \mu \Bigl(((R_2R_2-R_1R_1)\dot \omega)^2+(2R_1R_1\dot \omega)^2\Bigr) dx + \int_{\R^2} [\frac{D}{Dt},R_\mu] \omega\,  \dot \omega dx + \int_{\R^2} (\nabla^\perp u \nabla a) \cdot \dot u dx .
\end{align*}

\item  Using integration by parts and the fact that $\div u=0$, $\div \dot u = \nabla u : (\nabla u)^T$ we obtain
\begin{align*}
    &\int_{\R^2} \frac{D}{Dt} \nabla\tilde\pi\cdot \dot u dx = \int_{\R^2} \nabla \frac{D}{Dt}\tilde\pi \cdot \dot u + [\frac{D}{Dt},\nabla]\tilde\pi \cdot \dot u\,  dx \\
    &= - \frac{d}{dt} \int_{\R^2} \tilde\pi \nabla u : (\nabla u)^T dx + \int_{\R^2} \tilde\pi \frac{D}{Dt}(\nabla u : (\nabla u)^T) dx + \int_{\R^2} \tilde\pi \nabla u : (\nabla \dot u)^T dx \\
    &= - \frac{d}{dt} \int_{\R^2} \tilde\pi \nabla u : (\nabla u)^T dx + 3\int_{\R^2} \tilde\pi\nabla u :(\nabla \dot u)^T dx ,
\end{align*}
where we used in the third line that (due to $\div u=0$)
\begin{align*}
    \frac{D}{Dt}(\nabla u : (\nabla u)^T)  = 2 \nabla u : (\nabla \dot u)^T - 2 (\d_iu\cdot\nabla u)\cdot \nabla u_i = 2 \nabla u : (\nabla \dot u)^T.
\end{align*} 
\end{itemize} 
Summing up, we showed that
\begin{align*}
    &\frac{d}{dt} \Bigl(\frac12 \int_{\R^2} \rho \vert \dot u \vert^2 dx - \int_{\R^2}\tilde\pi\nabla u: (\nabla u)^T dx\Bigr) + \int_{\R^2} \mu \Bigl(((R_2R_2-R_1R_1)\dot \omega)^2+(2R_1R_1\dot \omega)^2\Bigr) dx \\
    &= - \int_{\R^2} [\frac{D}{Dt},R_\mu]\omega\,  \dot \omega dx - \int_{\R^2} (\nabla^\perp u \nabla a) \cdot \dot u\, dx - 3\int_{\R^2} \tilde\pi \nabla u :(\nabla \dot u)^T dx .
\end{align*}

Applying the commutator estimate \eqref{comm} we see that the first two terms on the right hand side are bounded up to a constant by
\[
\Vert \nabla u \Vert_{L^\infty} (\Vert \omega \Vert_{L^2} \Vert \dot \omega \Vert_{L^2} + \Vert \dot u \Vert_{L^2} \Vert \nabla a \Vert_{L^2}) 
\lesssim  \Vert \nabla u \Vert_{L^\infty} (\Vert \omega \Vert_{L^2} \Vert \dot \omega \Vert_{L^2} + \Vert \dot u \Vert_{L^2} \|\rho\dot u\|_{L^2}), 
\]
where the second inequality holds 
due to $\nabla^\perp a = \mathbb{P}(\rho \dot u)$ with the Helmholtz projection $\mathbb{P}$ by \eqref{inNS:a}.
The formula $\nabla\tilde\pi = -\nabla \Delta^{-1}\div (\rho \dot u)$,
the fact that $\nabla u:(\nabla \dot u)^T = \div(\dot u \cdot \nabla u)$ and integration by parts yield
\[
\vert \int_{\R^2} \tilde\pi \nabla u :(\nabla \dot u)^T dx \vert 
=\vert -\int\nabla\tilde\pi\cdot(\dot u\cdot\nabla u) dx \vert 
\lesssim \Vert \rho \dot u \Vert_{L^2} \Vert \dot u \Vert_{L^2} \Vert \nabla u \Vert_{L^\infty} .
\]
We multiply the above equality by $t$, integrate in time to derive 
\begin{align*}
    &\Vert t^\frac12 \sqrt{\rho}\dot u \Vert_{L^2}^2 + \Vert t'^{\frac12}\dot \omega \Vert_{L^2_tL^2}^2  \\
    &\lesssim \int_0^t \|\sqrt{\rho}\dot u \Vert_{L^2}^2 dt'
    + \int_0^t \Bigl\vert \int_{\R^2}\tilde\pi\nabla u : (\nabla u)^T dx \Bigr\vert dt' 
     + t \Bigl\vert \int_{\R^2} \tilde\pi\nabla u : (\nabla u)^T dx \Bigr\vert \\
   &\qquad 
    + \int_0^t \Vert t'^{\frac12}\nabla u \Vert_{L^\infty} 
    \bigl( \Vert \omega \Vert_{L^2} \Vert t'^{\frac12} \dot \omega \Vert_{L^2} 
    + \Vert t'^{\frac12}\rho \dot u \Vert_{L^2} \Vert \dot u \Vert_{L^2} \bigr)
    dt' \\
    &\lesssim_{\rho_\ast,\rho^\ast} \Vert \sqrt{\rho}\dot u \Vert_{L^2_tL^2}^2 + \Vert \rho \dot u \Vert_{L^2_tL^2} \Vert \nabla u \Vert_{L^\infty_tL^2}\Vert \nabla u \Vert_{L^2_tL^2} + \Vert t^\frac12 \rho \dot u \Vert_{L^2} \Vert t'^{\frac12} \nabla u \Vert_{L^\infty_tL^2} \Vert \nabla u \Vert_{L^\infty_tL^2} \\
    &\quad + \Vert t'^{\frac12}\nabla u \Vert_{L^2_tL^\infty} \Vert \omega \Vert_{L^\infty_tL^2} \Vert t'^{\frac12} \dot \omega \Vert_{L^2_tL^2} + \Vert \dot u \Vert_{L^2_tL^2}^2 + \int_0^t  \Vert t'^{\frac12}\nabla u \Vert_{L^\infty}^2  \Vert t'^{\frac12} \sqrt{\rho} \dot u \Vert_{L^2}^2 dt', 
\end{align*}
where for the second inequality we used $\Bigl\vert \int_{\R^2} \tilde\pi \nabla u : (\nabla u)^T dx \Bigr\vert \lesssim \Vert \rho \dot u \Vert_{L^2} \Vert \nabla u \Vert_{L^2}^2 $.
We find by Young's and Gronwall's inequality
\begin{align*}
    \Vert t'^{\frac12} \sqrt{\rho}\dot u \Vert_{L^\infty_t L^2}^2 + \Vert t'^{\frac12}\dot \omega \Vert_{L^2}^2 dt' 
    &\lesssim_{\rho_\ast,\rho^\ast} \tilde V(t) \Bigl(
    \Vert \sqrt{\rho}\dot u \Vert_{L^2_tL^2}^2 + \Vert \nabla u \Vert_{L^\infty_tL^2}^2 \Vert \nabla u \Vert_{L^2_tL^2}^2 \\
    &\quad + \Vert t'^{\frac12}\nabla u \Vert_{L^\infty_tL^2}^2 \Vert \nabla u \Vert_{L^\infty_tL^2}^2 + \Vert t'^{\frac12}\nabla u \Vert_{L^2_tL^\infty}^2 \Vert \omega \Vert_{L^\infty_tL^2}^2 + \Vert \dot u \Vert_{L^2_tL^2}^2
    \Bigr).
\end{align*}
Inserting the estimates \eqref{u:energy2}, \eqref{a:en-low2} and \eqref{a:ten-low2} results in \eqref{a:ten-highNS}.
 
\smallbreak 

\noindent\textbf{$W^{1,2+\epsilon}(\R^2)$-estimate for $a$.}
First, notice that it follows from the Helmholtz-decomposition $\nabla \dot u = RR^\perp \dot \omega + RR(\nabla u : (\nabla u)^T)$ with the Riesz-transform $R=\frac{\frac1i \nabla}{\sqrt{-\Delta}}$, that
\begin{align}\label{gradDtu:L2L2}
    \Vert t'^{\frac12}\nabla \dot u \Vert_{L^2_tL^2} \lesssim \Vert t'^{\frac12}\dot \omega \Vert_{L^2_tL^2} + \Vert t'^{\frac12} \nabla u \Vert_{L^2_t L^\infty} \Vert \nabla u \Vert_{L^\infty_tL^2}
    \lesssim \sigma_1(1+\sigma_0)e^{C\sigma_0^2}\tilde V(t)
    \lesssim \sigma_1 e^{C\sigma_0^2}\tilde V(t).
\end{align}
For the last inequality above we estimated the polynomial growth in $\sigma_0$ by the exponential function.

We derive from \eqref{a:en-low2}, \eqref{a:ten-low2}, \eqref{a:ten-highNS}, \eqref{gradDtu:L2L2} and \eqref{a:ten-high+} the following estimates for $a$:
\begin{align*}
    \|(t'^\delta a, t'^{\frac12+\delta}\nabla a)\|_{L^2_t L^2}
    \leq C(\sigma_0+\sigma_{-1})e^{C\sigma_0^2}V(t),
    \quad \|a\|_{L^2_t L^2}\leq C\sigma_0,
    \quad \|(\dot u, t'^{\frac12}\nabla \dot u)\|_{L^2_t L^2}\leq C\sigma_1 e^{C\sigma_0^2}\tilde V(t).
\end{align*}
where $a$ and $\dot u$ is related by $\nabla^\perp a = \mathbb{P}(\rho \dot u)$.
These estimates are very similar as \eqref{summary:a} in Subsection \ref{sect:L1Lip}, up to an extra factor $e^{C\sigma_0^2}$ and  the replacement of $\nabla a$-estimate by  $\dot u$-estimate.
Thus we can proceed exactly as in Subsection \ref{sect:L1Lip}.
Scaling with $\lambda=\frac{\sigma_0}{\sigma_{-1}}$ yields the following for $a_\lambda$:
\begin{align*}
    \Vert a_\lambda \Vert_{L^1_{\lambda^2t}L^{2+\epsilon}} + \Vert t'^{\frac12} a_\lambda \Vert_{L^2_{\lambda^2t}L^{2+\epsilon}} &\lesssim \sigma_0 \tilde V(t) e^{C\sigma_0^2}, \\
    \Vert \nabla a_\lambda \Vert_{L^1_{\lambda^2t}L^{2+\epsilon}} + \Vert t'^{\frac12} \nabla a_\lambda \Vert_{L^2_{\lambda^2t}L^{2+\epsilon}} &\lesssim (1+\sigma_0)^\frac{\epsilon}{2+\epsilon} \sigma_0^{\theta_1} (\sigma_{-1}\sigma_1)^{\theta_2} \tilde V(t) e^{C\sigma_0^2} \lesssim \sigma_0^{\theta_1} (\sigma_{-1}\sigma_1)^{\theta_2} \tilde V(t) e^{C\sigma_0^2}, \\
    \Vert a_\lambda \Vert_{L^1_{\lambda^2t} L^\infty} + \Vert t'^{\frac12} a_\lambda \Vert_{L^2_{\lambda^2t} L^\infty} &\lesssim (1+\sigma_0)^\epsilon \sigma_0^{\theta_3} (\sigma_{-1}\sigma_1)^{\theta_4} \tilde V(t) e^{C\sigma_0^2} \lesssim \sigma_0^{\theta_3} (\sigma_{-1}\sigma_1)^{\theta_4} \tilde V(t) e^{C\sigma_0^2},
\end{align*}
with the same exponents $\theta_1,\theta_2,\theta_3,\theta_4$ as in Subsection \ref{sect:L1Lip}.

\smallbreak 

\noindent\textbf{Conclusion.} 
With $
A(t)= \Vert \nabla u \Vert_{L^1_tL^\infty} + \Vert t'^{\frac12} \nabla u \Vert_{L^2_tL^\infty},$  we have derived
\[
A(t)\leq C \sigma_0^{\frac{\epsilon^2}{(2+\epsilon)^2}} (\sigma_{-1} \sigma_1)^\frac{2\epsilon}{(2+\epsilon)^2} e^{C\sigma_0^2} \exp\bigl(CA(t)+ C \sigma_0^{\theta_3}(\sigma_{-1}\sigma_1)^{\theta_4} e^{C\sigma_0^2} e^{CA(t)}\bigr).
\]
If the initial data satisfies
\begin{align}\label{NS:smallness}
    2C^2e^{2C\sigma_0^2} \bigl(\sigma_0^\frac{\epsilon}{2}\sigma_{-1}\sigma_1\bigr)^\frac{2\epsilon}{(2+\epsilon)^2} + Ce^{C\sigma_0^2} \sqrt{e} \bigl( \sigma_0^\frac{\theta_3}{\theta_4}\sigma_{-1}\sigma_1\bigr)^{\theta_4} \leq \frac12 ,
\end{align}
then with a bootstrap argument we arrive at the uniform bound
\[
A(t) \leq 2Ce^{C\sigma_0^2}\bigl(\sigma_0^\frac{\epsilon}{2}\sigma_{-1}\sigma_1\bigr)^\frac{2\epsilon}{(2+\epsilon)^2}.
\]
Notice that as before, the smallness condition \eqref{NSu0:cond} implies the condition \eqref{NS:smallness} above.
Following the proof of Theorem \ref{exthm} in Subsection \ref{sect:proofs} completes the first part of the proof of Corollary \ref{propthm} - \ref{d-propthm}. The statement about the density-patch is proved similarly as Corollary \ref{propthm} - \ref{Coro1}. We omit the details here.

\end{proof}


\appendix

\section{Proof of Lemma  \ref{Rmu-inv:prop}: The $L^{2+\epsilon}(\R^2)$-estimate}\label{sect:intro-proofs}

We sketch the proof of  the invertibility in $L^{2+\epsilon}(\R^2)$ of the operator $$R_\mu=(R_2R_2-R_1R_1)\mu(R_2R_2-R_1R_1) + (2R_1R_2)\mu (2R_1R_2),$$
given the positive lower and upper bounds of the coefficient: $\mu\in [\mu_\ast, \mu^\ast]$.
The ideas can be generalized to   a wider class of elliptic operators.  

\begin{proof}[Proof of Lemma \ref{Rmu-inv:prop}]

\noindent\textbf{Step 1: $L^2$-invertibility.} 
This is another proof of \eqref{L2:omega,a}, by use of the ellipticity of the operator $L_\mu$.

Firstly, we define the homogeneous space $\dot H^2(\R^2)$ in such a way that it is complete, for example by factoring out polynomials of order $1$. Then $\dot H^2(\R^2)$ is a Hilbert space, on which we define the bilinear, symmetric form 
\begin{align*}
    &\mathfrak{a} : \dot H^2(\R^2)\times \dot H^2(\R^2) \to \R, \\
    &(v,w)\mapsto \int_{\R^2} \mu \Bigl((\d_{22}-\d_{11})v(\d_{22}-\d_{11})w + 4\d_{12}v\d_{12}w\Bigr)dx.
\end{align*}
The bilinear form $\mathfrak{a}$ is bounded and coercive with lower and upper bounds as follows
\begin{align*}
    \mathfrak{a}(v,v)\geq \frac{\mu_\ast}{2} \Vert \nabla^2 v \Vert_{L^2}^2, \quad \vert \mathfrak{a}(v,w) \vert \leq 2\mu^\ast \Vert \nabla^2 v \Vert_{L^2} \Vert \nabla^2 w \Vert_{L^2},\quad \forall v,w\in \dot H^2(\R^2).
\end{align*}
By the Lax-Milgram lemma there exists for all $g\in \dot H^{-2}(\R^2)$, the dual space of $\dot H^2(\R^2)$, a unique element $v\in \dot H^2(\R^2)$ such that
\begin{align}\label{lax-m}
    \mathfrak{a}(v,w)= 
    \langle w,g \rangle_{\dot H^2\times \dot H^{-2}}, \quad \forall w\in \dot H^2(\R^2).
\end{align}
That is, for any $g\in \dot H^{-2}(\R^2)$, there exists a unique weak solution $v\in \dot H^2(\R^2)$ of the elliptic equation
\begin{equation*}
    L_\mu v=g,\quad \hbox{with }L_\mu= (\d_{22}-\d_{11})\mu (\d_{22}-\d_{11})+(2\d_{12})\mu(2\d_{12}).
\end{equation*}

Now we define the bounded operator $\ddiv:L^2(\R^2;\R^3) \to \dot H^{-2}(\R^2; \R)$ as follows. 
For $G=(G_1,G_2,G_3)^T\in L^2(\R^2;\R^3)$, we define $\ddiv G\in \dot H^{-2}(\R^2)$   by 
\begin{align*}
    \langle w,\ddiv G \rangle_{\dot H^2\times \dot H^{-2}} = \int_{\R^2} \Bigl( G_1\d_{11}w+G_2\d_{22}w+G_3 \d_{12}w\Bigr) dx, \quad \forall w\in \dot H^2(\R^2).
\end{align*}
Then the operator
\begin{align*}
    \mathfrak{A}:L^2(\R^2;\R^3) \to L^2(\R^2;\R^3), 
    \quad G\mapsto \nabla^2 L_\mu^{-1} \ddiv G ,
\end{align*}
is bounded on $L^2(\R^2;\R^3)$, where we identify $\nabla^2 \cong (\d_{11}, \d_{22}, \d_{12})^T$. 
Indeed, for $G\in L^2(\R^2;\R^3)$, let $v_G\in \dot H^2(\R^2)$ be the Lax-Milgram solution of $L_\mu v_G=\ddiv G$ in the sense of (\ref{lax-m}). 
Choosing $w=v_G$ in (\ref{lax-m}) and using the coercivity of the sesquilinearform $\mathfrak{a}$ yields the boundedness of $\mathfrak{A}$ on 
$L^2(\R^2;\R^3)$ as follows
\begin{align*}
    \frac{\mu_\ast}{2} \Vert \nabla^2 v_G \Vert_{L^2}^2 \leq \mathrm{Re}\, \mathfrak{a} (v_G,v_G) 
    = \mathrm{Re} \langle v_G, \ddiv G \rangle_{\dot H^2\times \dot H^{-2}} \leq\Vert v_G\Vert_{\dot H^2} \Vert\ddiv G \Vert_{\dot H^{-2}} 
    \lesssim \Vert \nabla^2 v_G \Vert_{L^2} \Vert G \Vert_{L^2}.
\end{align*}

\noindent\textbf{Step 2: $L^{2+\epsilon}$-invertibility.} 
In order to prove that the operator $\mathfrak{A}$ is bounded on $L^{2+\epsilon}(\R^2;\R^3)$ for some $\epsilon>0$ we are going to make use of Z. Shen's theorem \cite[Theorem 3.1]{shen2005bounds}, which is a version of the Calder\'on-Zygmund Lemma.
More precisely, if   there exists some constant $C>0$ such that 
the  following holds   for all $x_0\in\R^2$, $r>0$ and   $G\in L^\infty(\R^2;\R^3)$ with compact support outside    $B_{3 r}(x_0)$ 
\begin{align}\label{shen:cond}
    \Bigl(\frac{1}{r^2} \int_{B_r(x_0)} \vert \mathfrak{A}G \vert^q dx\Bigr)^\frac1q \leq C \Bigl(\frac{1}{4r^2} \int_{B_{3 r}(x_0)} \vert \mathfrak{A}G\vert^2 dx \Bigr)^\frac12 ,
\end{align}  
then $\mathfrak{A}$ is bounded on $L^p(\R^2;\R^3)$ for any $p\in (2,q)$. 

We sketch the proof of (\ref{shen:cond}). For this let $x_0\in\R^2$, $r>0$ and $G\in L^\infty(\R^2;\R^3)$ have compact support with $G\equiv 0$ in $B_{3r}(x_0)$. Then $v_G=L_\mu^{-1}\ddiv G$ is the solution to
\begin{align*}
    \mathfrak{a}(v_G,w)=\langle w,\ddiv G\rangle_{\dot H^2\times \dot H^{-2}} =0 \quad \forall w\in C_c^\infty(B_{2r}(x_0)),
\end{align*}
and hence, $L_\mu v_G=0$ in $B_{2r}(x_0)$ in the sense of distributions. Thus A. Barton's higher order version of Meyer's reverse H\"older estimate \cite[Theorem 24]{barton2016gradient} yields the existence of some $q\in (2,\infty)$ such that (\ref{shen:cond}) holds.

Consequently, $\mathfrak{A}=\nabla^2 L_\mu^{-1}\ddiv$ is bounded on $L^{2+\epsilon}(\R^2;\R^3)$ for some $\epsilon >0$. In particular, $R_\mu^{-1}=\Delta L_\mu^{-1}\Delta$ is bounded on $L^{2+\epsilon}(\R^2)$, which concludes the proof.

\end{proof}

\section{Proof of Lemma \ref{commutator:lem}: Commutator estimates}\label{sect:proof,comm}

\begin{proof}[Proof of Lemma \ref{commutator:lem}]
The proof of the first estimate (\ref{comm}) can be found in A. P. Calder\'on's article \cite[Theorem 1]{calderon1965commutators}. 
We sketch the proof of the second statement in Lemma \ref{commutator:lem}.

Recall Bony's decomposition for any product into low-high frequency, high-low frequency and remainder parts below:
\begin{align*}
    FG=\mathcal{T}_F G+\mathcal{T}_G F+\mathcal{R}(F,G) ,
\end{align*}
and we refer to \cite{bahouri2011fourier} for the precise definitions of the paraproduct $\mathcal{T}_FG$ and the remainder term $\mathcal{R}(F,G)$.
We apply Bony's decomposition to the product $\d_X R^2g=X_k (R^2\d_k g)$ and $\div(Xg)=\d_k(X_k g)$,   for $X=(X_1, X_2)^T$, to achieve
\begin{align*}
    \d_XR^2g   &=  [\mathcal{T}_{X_k},R^2\d_k]g+  \mathcal{T}_{R^2\d_kg}X_k + \mathcal{R}(X_k,R^2\d_kg) \\
    &\quad +R^2 \div(Xg)  - R^2 \d_k\mathcal{R}(X_k,g) - R^2 \d_k \mathcal{T}_gX_k,
\end{align*}
where we used the Einstein's summation convention to omit $\sum_k$ above.
Observe that for $q>2$ (see for example \cite{bahouri2011fourier} or the proofs of \cite[Lemma 5.1]{paicu2020striated} and \cite[Lemma 2.10]{danchin2020well})
\begin{align}\label{para:g-X}
    \Vert \bigl(\mathcal{T}_{ \d_k h} X_k,\d_k \mathcal{T}_h X_k, \mathcal{R}(X_k, \d_k h),\d_k\mathcal{R}(X_k, h), [\mathcal{T}_{X_k},R^2\d_k]h \bigr) \Vert_{L^q} \lesssim \Vert \nabla X \Vert_{L^p} \Vert h \Vert_{L^\infty}.
\end{align}
This (with $h=R^2g$ or $g$),  together with 
\begin{align*}
    \Vert R^2 \div(Xg) \Vert_{L^p} &\lesssim \Vert \d_Xg \Vert_{L^p} + \Vert \nabla X \Vert_{L^p} \Vert g \Vert_{L^\infty}, 
\end{align*}
and the fact that 
$\Vert g \Vert_{L^\infty} =\|(R_1R_1+R_2R_2)g\|_{L^\infty} \leq 2\Vert R^2 g \Vert_{L^\infty}$ 
yields \eqref{comm:infty-0}, \eqref{comm:infty-1}.

Next, we show \eqref{Rmu:comm}. Denoting $\rttoo = R_2R_2-R_1R_1$, $\rto=2R_1R_2$, such that $R_\mu=P_1\mu P_1+P_2\mu P_2$ and the commutator reads (noticing $\d_X h=\div(X h)-h\div X$)
\begin{align*}
     [R_\mu, \d_X]g & = \rttoo\mu [\rttoo, \d_X] g + \rttoo[\mu,\d_X]\rttoo g + [\rttoo,\d_X]\mu \rttoo g + \rto\mu [\rto,\d_X]g + \rto [\mu,\d_X]\rto g + [\rto, \d_X]\mu\rto g \\
    &= -\rttoo\mu \bigl(\d_X \rttoo g - \rttoo \div(Xg) + \rttoo (g \div X)\bigr) 
    - \bigl(\d_X \rttoo\mu \rttoo g - \rttoo \div(X\mu \rttoo g) + \rttoo(\mu \rttoo  g\, \div X)\bigr) \\
    &\quad -\rto\mu \bigl(\d_X \rto g - \rto \div(Xg) + \rto (g \div X)\bigr) 
    - \bigl(\d_X \rto\mu \rto g - \rto \div(X\mu \rto g) + \rto(\mu \rto  g \,\div X)\bigr)
    \\
    &\quad - \rttoo (\d_X\mu \rttoo g) - \rto (\d_X\mu\rto g)
    \\
    &=-\Bigl(  P_1\mu(\d_X \rttoo g - \rttoo \div(Xg))
    +P_2\mu (\d_X \rto g - \rto \div(Xg))\Bigr)
    \\
    &\quad -\Bigl(R_\mu(g\div X)
    +\rttoo(\mu \rttoo  g\, \div X) + \rto(\mu \rto  g \,\div X) \Big)
    \\
    &\quad -\Bigl( \bigl( \d_X\rttoo \mu \rttoo g - \rttoo \div(X\mu \rttoo g) \bigr)
    +\bigl( \d_X \rto \mu \rto g - \rto \div(X\mu \rto g)\bigr)\Bigr) \\
    &\quad - \Bigl(\rttoo (\d_X\mu \rttoo g) + \rto (\d_X\mu\rto g) \Bigr) .
\end{align*}
We apply  \eqref{comm:infty-1}  and the $L^p$-boundedness of Riesz operators to bound the first  and second brackets on the right hand side in $L^p(\R^2)$ by $\Vert \nabla X \Vert_{L^p} \Vert R^2g \Vert_{L^\infty}$, respectively. The fourth bracket is bounded in $L^p(\R^2)$ by $\Vert \d_X\mu \Vert_{L^q} \Vert R^2 g \Vert_{L^\frac{qp}{q-p}}$. 
Similarly as above, we use Bony's decomposition to rewrite the third bracket on the right hand side above  as
\begin{align*}
    &[\mathcal{T}_{X_k}, \d_k \rttoo]\mu \rttoo g + \mathcal{T}_{\d_k \rttoo\mu \rttoo g} X_k + \mathcal{R}(X_k,\d_k \rttoo\mu \rttoo g) \\
    &- \rttoo \d_k\bigl(  \mathcal{T}_{\mu \rttoo g} X_k + \mathcal{R}(X_k,\mu \rttoo g)\bigr) \\
    &+[\mathcal{T}_{X_k},\d_k \rto]\mu \rto g + \mathcal{T}_{\d_k \rto\mu \rto g} X_k + \mathcal{R}(X_k,\d_k \rto\mu \rto g) \\
    &- \rto \d_k\bigl(  \mathcal{T}_{\mu \rto g} X_k + \mathcal{R}(X_k,\mu \rto g)\bigr),
\end{align*}
where by \eqref{para:g-X} all terms can be bounded in $L^p(\R^2)$ by $\Vert \nabla X \Vert_{L^p} \Vert R^2g \Vert_{L^\infty}$, except for
\begin{align*}
    &\mathcal{T}_{\d_k \rttoo \mu \rttoo g}X_k + \mathcal{R}(X_k,\d_k\rttoo \mu \rttoo g)
    + \mathcal{T}_{\d_k \rto \mu \rto g} X_k + \mathcal{R}(X_k, \d_k \rto \mu \rto g) \\
    &= \mathcal{T}_{\d_k R_\mu g} X_k + \mathcal{R}(X_k,\d_k R_\mu g).
\end{align*}
Again by \eqref{para:g-X}, these last terms satisfy
\begin{align*}
    \Vert \mathcal{T}_{\d_k R_\mu g} X_k \Vert_{L^q} + \Vert \mathcal{R}(X_k,\d_k R_\mu g) \Vert_{L^q} \lesssim \Vert \nabla X \Vert_{L^q} \Vert R_\mu g \Vert_{L^\infty}.
\end{align*}
This finishes the proof of \eqref{Rmu:comm}.

\end{proof}

\section{Proof of Proposition \ref{u-decay:prop}: Energy estimates for the velocity}\label{sect:u-en}

In this section we prove Proposition \ref{u-decay:prop}, and at the end we mention the minor changes in the proof of \eqref{u:decay+2} for the density-dependent Navier-Stokes equations \eqref{NS}. 
We recall the definition of the Fourier transform of a Schwartz function $f(x)\in \mathcal{S}(\R^2)$ as
\begin{align*}
    \hat{f}(\xi)=\mathcal{F}(f)(\xi)=\frac{1}{2\pi}\int_{\R^2} e^{-  ix\cdot\xi} f(x) \,dx,\quad \xi\in \R^2,
\end{align*}
and   we define the Fourier transform of a tempered distribution $g\in \mathcal{S}'(\R^2)$ by duality: $\langle \hat{g}, f\rangle_{\mathcal{S}', \mathcal{S}}=\langle g, \hat{f}\rangle_{\mathcal{S}', \mathcal{S}}$.

\begin{proof}[Proof of Proposition \ref{u-decay:prop}]

\begin{itemize}
    \item \textbf{Proof of (\ref{u:energy}):} 
    Multiplying the momentum equation (\ref{muNS})$_2$ by $u$, integrating over $\R^2$ and using integration by parts results in
\begin{align}\label{meq-u:energy}
    \frac12 \frac{d}{dt} \Vert u(t)\Vert_{L^2}^2 + 2\mu_\ast  \Vert \nabla u(t)\Vert_{L^2}^2 \leq 0.
\end{align}
The estimate (\ref{u:energy}) then follows from integrating in time over $[0,t]$.

\item \textbf{Proof of (\ref{u:decay+}):} 
We claim the following decay estimate
\begin{equation}\label{u:decay}
\Vert u(t)\Vert_{L^2}  \leq C_\delta C_0  \langle t \rangle^{-\delta_-},
\end{equation}
where $\delta_-\in (0,\delta)$, $C_0= \Vert u_0 \Vert_{L^2\cap \dot H^{-2\delta}} + \Vert \mu_0-1 \Vert_{L^2} \Vert u_0 \Vert_{L^2}$, and $C_\delta$ is a constant depending only on $\delta_-,\delta,\mu_\ast$.

Now multiplying both sides of (\ref{meq-u:energy}) by $\langle t \rangle^{2\delta'}=(e+t)^{2\delta'}$, $\delta'>0$ and integrating in time we obtain
\begin{align*}
    \Vert \langle t \rangle^{\delta'} u \Vert_{L^2}^2 + 2\mu_\ast\Vert \langle t' \rangle^{\delta'} \nabla u \Vert_{L^2_tL^2}^2 \lesssim \Vert u_0 \Vert_{L^2}^2 + \int_0^t \langle t' \rangle^{2\delta'-1} \Vert u(t')\Vert_{L^2}^2 dt'.
\end{align*}
Thus \eqref{u:decay+} follows from the claim \eqref{u:decay} by choosing $\delta'\in (0,\delta_-)$.

\textbf{Proof of the claim  (\ref{u:decay}):}
We now turn to showing (\ref{u:decay}). The idea is to use a time-dependent cut-off in frequency space. 
Let $g(t)$ be a positive function to be determined later, and let $S(t)$ denote a low-frequency set with respect to $g(t)$ as
\[S(t)=\Bigl\{\xi\in\R^2 : \vert \xi \vert \leq \sqrt{\frac{1}{2\mu_\ast}}g(t)\Bigr\}.\]
Then we deduce from (\ref{meq-u:energy}) that (noticing $\widehat{\d_{x_j}f}(\xi)=i\xi_j\hat{f}(\xi)$)
\begin{equation}\label{u:appendix,L2}\frac{d}{dt}\Vert u(t)\Vert_{L^2}^2 + g^2(t) \Vert u(t) \Vert_{L^2}^2 \leq g^2(t) \int_{S(t)} \vert \hat u (t,\xi) \vert^2 d\xi .\end{equation}

Now we rewrite  the velocity equation $\eqref{muNS}_2$: $(\d_t-\Delta) u=-u\cdot\nabla u+\div((\mu-1)Su)-\nabla\pi$ in the form of Duhamel's formula as follows
\begin{align}\label{u:formula}
    u(t)=e^{t\Delta}u_0 + \int_0^t e^{(t-t')\Delta} \mathbb{P}\Bigl(\div((\mu-1)Su)- u\cdot\nabla u\Bigr)(t')dt',
\end{align}
where $\mathbb{P}=\hbox{Id}+\nabla(-\Delta)^\perp\div$ denotes the Leray-Helmholtz projector.
Then \eqref{u:formula} implies for any fixed time $t>0$,
\begin{align*}
    \vert \hat u (t,\xi)\vert \lesssim e^{-t\vert \xi \vert^2} \vert \hat u_0 (\xi)\vert + \int_0^t e^{-(t-t')\vert \xi \vert^2} \vert \xi \vert \vert \mathcal{F}((\mu-1)Su)-\mathcal{F}(u\otimes u) \vert(t')dt',
\end{align*}
and thus (noticing $\int_{S(t)}|\xi|^2 \,d\xi\lesssim \frac{1}{(\mu_\ast)^2}g^4(t)$)
\begin{align*}
    g^2(t)\int_{S(t)} \vert \hat u(t,\xi)\vert^2 d\xi \lesssim_{\mu_\ast} g^2(t)\int_{S(t)}e^{-2t\vert \xi\vert^2} \vert \hat u_0 (\xi) \vert^2 d\xi + g^6(t)\Bigl(\int_0^t \Vert \mathcal{F}((\mu-1)Su- u\otimes u)(t')\Vert_{L^\infty_\xi} dt'\Bigr)^2.
\end{align*} 
The first integral on the right hand side satisfies
\begin{align*}
g^2(t)\int_{S(t)}e^{-2t\vert \xi\vert^2} \vert \hat u_0 (\xi) \vert^2 d\xi 
&\leq g^2(t)\int_{\R^2} \langle t\rangle^{-2\delta}\Bigl( e^{-2t|\xi|^2} (\langle t\rangle |\xi|^2)^{2\delta}\Bigr) \bigl( |\xi|^{-4\delta}|\hat u_0(\xi)|^2 \bigr)\,d\xi 
\\
&\lesssim 1_{\{t\leq 1\}}g^2(t)\|u_0\|_{L^2}^2
+1_{\{t\geq 1\}}g^2(t)  t ^{-2\delta} \Vert u_0 \Vert_{\dot H^{-2\delta}}^2,
\end{align*}
and the second one can be bounded as
\begin{align*}
    g^6(t)\Bigl(\int_0^t \Vert \mathcal{F}((\mu-1)Su- u\otimes u)(t')\Vert_{L^\infty_\xi} dt'\Bigr)^2 
    &\lesssim g^6(t)\Bigl(\int_0^t \Vert ((\mu-1)Su- u\otimes u)(t')\Vert_{L^1_x} dt'\Bigr)^2 \\
    &\lesssim g^6(t) t\Vert \mu-1\Vert_{L^\infty_tL^2}^2 \Vert \nabla u \Vert_{L^2_tL^2}^2 + g^6(t)\Bigl(\int_0^t \Vert u(t')\Vert_{L^2}^2 dt'\Bigr)^2 \\
     &\lesssim g^6(t) t \Vert \mu_0-1\Vert_{L^2}^2 \Vert u_0 \Vert_{L^2}^2  +  g^6(t)\Vert u\Vert_{L^2_tL^2}^2 .
\end{align*}
Inserting these estimates into \eqref{u:appendix,L2} we obtain  
\begin{align*}
     \exp\Bigl(\int_0^t g^2(t')dt'\Bigr) \Vert u(t)\Vert_{L^2}^2 
     &\lesssim \|u_0\|_{L^2}^2 + 
    \|u_0\|_{L^2}^2 \int^{1}_0 \exp\Bigl(\int_0^{t'} g^2(t'')dt''\Bigr)g^2(t') dt'
    \\
&    +\|u_0\|_{\dot H^{-2\delta}}^2 \int_{1}^t \exp\Bigl(\int_0^{t'} g^2(t'')dt''\Bigr)  g^2(t')   {t'}  ^{-2\delta}  dt'
    \\
    &+  \|\mu_0-1\|_{L^2}^2\|u_0\|_{L^2} ^2  \int_0^t \exp\Bigl(\int_0^{t'} g^2(t'')dt''\Bigr)  g^6(t')t'   dt' 
    \\
    &+     \int_0^t \exp\Bigl(\int_0^{t'} g^2(t'')dt''\Bigr)    g^6(t')  \Vert u \Vert_{L^2_{t'}L^2}^2    dt'.
\end{align*}
Choosing $g^2(t)=2\delta_- \langle t \rangle^{-1}$ such that  $\int^t_0 g^2=2\delta_-(\log\langle t\rangle-1)$ and $e^{\int^t_0 g^2}=e^{-2\delta_-}\langle t\rangle ^{2\delta_-}$  yields
\begin{align}\label{u:L2:t2d}
    \langle t \rangle^{2\delta_-} \Vert u(t) \Vert_{L^2}^2 \lesssim  C_0^2+ \int_0^t \langle t' \rangle^{-3+2\delta_-}  \Vert u  \Vert_{L^2 _{t'}L^2}^2   dt' .
\end{align}

We now define
\begin{align*}
    y(t) =\int_{t-1}^t \Vert u(t')\Vert_{L^2}^2 \langle t' \rangle^{2\delta_-} dt', \quad t\geq 1, \quad \hbox{and}\quad 
    Y(t) = \max_{1\leq t'\leq t} y(t')  .
\end{align*}
Notice that by the above definition
 $\Vert u  \Vert_{L^2_{t}L^2}^2\leq CY(t)\int^t_0\langle t'\rangle^{-2\delta_-}dt'  =CY(t)\frac{\langle t \rangle^{1-2\delta_-}}{1-2\delta_-}$.
Using this inequality after integrating (\ref{u:L2:t2d}) over $[t-1,t]$, we obtain
\begin{align*}
    y(t)\lesssim C_0^2+\int_{t-1}^t \int_0^{t'} \langle t'' \rangle^{-3+2\delta_-}    \Vert u  \Vert_{L^2_{t''}L^2}^2  dt''dt' 
    \lesssim C_0+ \int_0^t \langle t' \rangle^{-2}   Y(t')dt',
\end{align*}
and therefore by Gronwall's inequality it follows that
$Y(t)\lesssim C_0^2$.
Finally
\[\Vert u  \Vert_{L^2_{t}L^2}^2\leq C Y(t)\frac{\langle t \rangle^{1-2\delta_-}}{1-2\delta_-} \lesssim  C_0\langle t \rangle^{1-2\delta_-}.\]
Applying this inequality to (\ref{u:L2:t2d}) we finally arrive at
\begin{align*}
    \langle t \rangle^{2\delta_-} \Vert u(t) \Vert_{L^2}^2 \lesssim  C_0^2+ C_0^2\int_0^t \langle t' \rangle^{-2}  dt' \lesssim  C_0^2.
\end{align*}
This completes the proof of (\ref{u:decay}).
\end{itemize}
\end{proof}

In order to show \eqref{u:decay+2} for the system \eqref{NS}, we replace  the formula (\ref{u:formula})   by
\begin{align*}
    u(t) = e^{t\Delta}u_0 + \int_0^t e^{(t-t')\Delta} \mathbb{P} \Bigl(\div((\mu-1)Su)+ (1-\rho)\dot u-u\cdot\nabla u\Bigr)(t')dt'.
\end{align*}
The additional term can be estimated as
\begin{align*}
    \Bigl(\int_0^t \Vert \mathcal{F}((1-\rho)\dot u) \Vert_{L^\infty} dt'\Bigr)^2 &\lesssim \Vert 1-\rho_0 \Vert_{L^2}^2 \log\langle t\rangle \Vert \langle t' \rangle^\frac12 \dot u \Vert_{L^2_tL^2}^2 \\
    &\lesssim \Vert 1-\rho_0 \Vert_{L^2}^2 \Vert u_0\Vert_{H^1}^2 \log\langle t\rangle V(t) e^{C\Vert u_0 \Vert_{L^2}^2},
\end{align*}
where the second inequality follows from \eqref{a:en-low2}, \eqref{a:ten-low2}. We then proceed similarly as above (see also \cite[pp. 310-311]{wiegner1987decay} or \cite{abidi2015global2D}).

\subsection*{Acknowledgements}
X. Liao is partially supported by   the Deutsche Forschungsgemeinschaft (DFG, German Research Foundation)— Project-ID 258734477-SFB 1173.
The authors   thank sincerely their colleague Dr. Patrick Tolksdorf
for pointing out the references \cite{barton2016gradient, shen2005bounds} and his help in the proof of Lemma \ref{Rmu-inv:prop}.

\printbibliography

@article {simon1990nonhomogeneous,
    AUTHOR = {Simon, J.},
     TITLE = {Nonhomogeneous viscous incompressible fluids: existence of
              velocity, density, and pressure},
   JOURNAL = {SIAM J. Math. Anal.},
  FJOURNAL = {SIAM Journal on Mathematical Analysis},
    VOLUME = {21},
      YEAR = {1990},
    NUMBER = {5},
     PAGES = {1093--1117},
      ISSN = {0036-1410},
   MRCLASS = {35Q30 (76D05)},
  MRNUMBER = {1062395},
MRREVIEWER = {Paul\ G.\ Schmidt},
       DOI = {10.1137/0521061},
       URL = {https://doi.org/10.1137/0521061},
}

@article {ladyzhenskaya1978unique,
    AUTHOR = {Lady\v zenskaja, O. A. and Solonnikov, V. A.},
     TITLE = {The unique solvability of an initial-boundary value problem
              for viscous incompressible inhomogeneous fluids},
   JOURNAL = {Zap. Nau\v cn. Sem. Leningrad. Otdel. Mat. Inst. Steklov.
              (LOMI)},
  FJOURNAL = {Zapiski Nau\v cnyh Seminarov Leningradskogo Otdelenija
              Matemati\v ceskogo Instituta im. V. A. Steklova Akademii Nauk
              SSSR (LOMI)},
    VOLUME = {52},
      YEAR = {1975},
     PAGES = {52--109, 218--219},
   MRCLASS = {35Q10},
  MRNUMBER = {425391},
MRREVIEWER = {P.\ C.\ Fife},
}

@book {ladyzhenskaya1969mathematical,
    AUTHOR = {Ladyzhenskaya, O. A.},
     TITLE = {The mathematical theory of viscous incompressible flow},
    SERIES = {Mathematics and its Applications},
    VOLUME = {Vol. 2},
   EDITION = {English},
      NOTE = {Translated from the Russian by Richard A. Silverman and John
              Chu},
 PUBLISHER = {Gordon and Breach Science Publishers, New York-London-Paris},
      YEAR = {1969},
   MRCLASS = {35.00 (76.00)},
  MRNUMBER = {254401},
}

@article {lorcaBoldrini,
    AUTHOR = {Lorca, S. A. and Boldrini, J. L.},
     TITLE = {The initial value problem for a generalized {B}oussinesq
              model},
   JOURNAL = {Nonlinear Anal.},
  FJOURNAL = {Nonlinear Analysis. Theory, Methods \& Applications. An
              International Multidisciplinary Journal},
    VOLUME = {36},
      YEAR = {1999},
    NUMBER = {4},
     PAGES = {457--480},
      ISSN = {0362-546X,1873-5215},
   MRCLASS = {35Q53 (35Q35 76D99)},
  MRNUMBER = {1675260},
MRREVIEWER = {Ming\ Mei},
       DOI = {10.1016/S0362-546X(97)00635-4},
       URL = {https://doi.org/10.1016/S0362-546X(97)00635-4},
}

@article{WangZhang,
  title={Global well-posedness for the 2-D Boussinesq system with the temperature-dependent viscosity
and thermal diffusivity},
  author={Wang, C. and Zhang, Z.},
  journal={Adv. Math.},
  volume={228},
  pages={43--62},
  year={2011} 
}

@article{HmidiKeraani,
  title={On the global well-posedness of the Boussinesq system with zero viscosity},
  author={Hmidi, T. and Keraani, S.},
  journal={Indiana Univ. Math. J.},
  volume={58},
  pages={1591--1681},
  year={2009} 
}

@article {heliao,
    AUTHOR = {He, Z. and Liao, X.},
     TITLE = {On the two-dimensional {B}oussinesq equations with
              temperature-dependent thermal and viscosity diffusions in
              general {S}obolev spaces},
   JOURNAL = {Z. Angew. Math. Phys.},
  FJOURNAL = {Zeitschrift f\"ur Angewandte Mathematik und Physik. ZAMP.
              Journal of Applied Mathematics and Physics. Journal de
              Math\'ematiques et de Physique Appliqu\'ees},
    VOLUME = {73},
      YEAR = {2022},
    NUMBER = {1},
     PAGES = {Paper No. 16},
      ISSN = {0044-2275,1420-9039},
   MRCLASS = {35Q30 (76D03)},
  MRNUMBER = {4346491},
MRREVIEWER = {Wei\ Luo},
       DOI = {10.1007/s00033-021-01650-3},
       URL = {https://doi.org/10.1007/s00033-021-01650-3},
}

@article{heliao-correction,
  title={Correction to: On the two-dimensional Boussinesq equations with temperature-dependent thermal and viscosity diffusions in general Sobolev spaces},
  author={He, Z. and Liao, X.},
  journal={Z. Angew. Math. Phys.},
  volume={75}, 
  year={2024} 
}

@misc{WuJiahong,
  title={The 2D Incompressible Boussinesq Equations},
  author={Wu, J.},
  note="Peking University Summer School Lecture Notes",
  year={2012},
howpublished="\url{https://www.math.pku.edu.cn/amel/docs/20190318173755940995.pdf}"
}

@article{Nalimov,
    author = {Nalimov, V. I.},
    title ={The Cauchy-Poisson problem},
    journal = {Dyn. Splosh. Sredy},
    year = {1974},
    volume={18},
    pages={104--210}
}

@article {danchin2003density,
    AUTHOR = {Danchin, R.},
     TITLE = {Density-dependent incompressible viscous fluids in critical
              spaces},
   JOURNAL = {Proc. Roy. Soc. Edinburgh Sect. A},
  FJOURNAL = {Proceedings of the Royal Society of Edinburgh. Section A.
              Mathematics},
    VOLUME = {133},
      YEAR = {2003},
    NUMBER = {6},
     PAGES = {1311--1334},
      ISSN = {0308-2105,1473-7124},
   MRCLASS = {76D03 (35Q30 76D05)},
  MRNUMBER = {2027648},
MRREVIEWER = {Isabelle\ Gruais},
       DOI = {10.1017/S030821050000295X},
       URL = {https://doi.org/10.1017/S030821050000295X},
}

@article {abidi2007equation,
    AUTHOR = {Abidi, H.},
     TITLE = {\'Equation de {N}avier-{S}tokes avec densit\'e{} et
              viscosit\'e{} variables dans l'espace critique},
   JOURNAL = {Rev. Mat. Iberoam.},
  FJOURNAL = {Revista Matem\'atica Iberoamericana},
    VOLUME = {23},
      YEAR = {2007},
    NUMBER = {2},
     PAGES = {537--586},
      ISSN = {0213-2230,2235-0616},
   MRCLASS = {35Q30 (35B30 76D03 76D05)},
  MRNUMBER = {2371437},
MRREVIEWER = {Christopher\ C.\ Hallstrom},
       DOI = {10.4171/RMI/505},
       URL = {https://doi.org/10.4171/RMI/505},
}

@article {danchin2012lagrangian,
    AUTHOR = {Danchin, R. and Mucha, P. B.},
     TITLE = {A {L}agrangian approach for the incompressible
              {N}avier-{S}tokes equations with variable density},
   JOURNAL = {Comm. Pure Appl. Math.},
  FJOURNAL = {Communications on Pure and Applied Mathematics},
    VOLUME = {65},
      YEAR = {2012},
    NUMBER = {10},
     PAGES = {1458--1480},
      ISSN = {0010-3640,1097-0312},
   MRCLASS = {76D05 (35Q35)},
  MRNUMBER = {2957705},
MRREVIEWER = {Paolo\ Maremonti},
       DOI = {10.1002/cpa.21409},
       URL = {https://doi.org/10.1002/cpa.21409},
}

@article {danchin2013incompressible,
    AUTHOR = {Danchin, R. and Mucha, P. B.},
     TITLE = {Incompressible flows with piecewise constant density},
   JOURNAL = {Arch. Ration. Mech. Anal.},
  FJOURNAL = {Archive for Rational Mechanics and Analysis},
    VOLUME = {207},
      YEAR = {2013},
    NUMBER = {3},
     PAGES = {991--1023},
      ISSN = {0003-9527,1432-0673},
   MRCLASS = {76D05 (35Q35)},
  MRNUMBER = {3017294},
MRREVIEWER = {Michael\ Wolff},
       DOI = {10.1007/s00205-012-0586-4},
       URL = {https://doi.org/10.1007/s00205-012-0586-4},
}

@article {liao2019global,
    AUTHOR = {Liao, X. and Zhang, P.},
     TITLE = {Global regularity of 2{D} density patches for viscous
              inhomogeneous incompressible flow with general density: low
              regularity case},
   JOURNAL = {Comm. Pure Appl. Math.},
  FJOURNAL = {Communications on Pure and Applied Mathematics},
    VOLUME = {72},
      YEAR = {2019},
    NUMBER = {4},
     PAGES = {835--884},
      ISSN = {0010-3640,1097-0312},
   MRCLASS = {35Q30 (35B65 76D03 76D05)},
  MRNUMBER = {3914884},
MRREVIEWER = {Luigi\ Carlo\ Berselli},
       DOI = {10.1002/cpa.21782},
       URL = {https://doi.org/10.1002/cpa.21782},
}

@book {Lide,
    author={Lide, D.},
    title= {CRC Handbook of {C}hemistry and {Physics}, {I}nternet {V}ersion 2005},
    url={http://www.hbcpnetbase.com},
    publisher = {CRC Press, Boca Raton, FL},
    year={2005},
}

@article{PTBC,
    author = {P\'erez, C. and Thomas, J.-M. and Blancher, S. and Creff, R.},
    title ={The steady Navier–Stokes/energy system with temperature-dependent viscosity I. Analysis of the continuous problem \& II. The discrete problem and numerical experiments.},
    journal = {Int. J. Numer. Methods Fluids},
    volume={56},
    year = {2008},
    pages={63–89 \& 91–114},
}

@article {AbelsFei,
    AUTHOR = {Abels, H. and Fei, M.},
     TITLE = {Sharp interface limit for a {N}avier-{S}tokes/{A}llen-{C}ahn
              system with different viscosities},
   JOURNAL = {SIAM J. Math. Anal.},
  FJOURNAL = {SIAM Journal on Mathematical Analysis},
    VOLUME = {55},
      YEAR = {2023},
    NUMBER = {4},
     PAGES = {4039--4088},
      ISSN = {0036-1410,1095-7154},
   MRCLASS = {76T99 (35Q30 35R35 76D05 76D45)},
  MRNUMBER = {4631445},
MRREVIEWER = {Dalibor\ Pra\v z\'ak},
       DOI = {10.1137/22M1523698},
       URL = {https://doi.org/10.1137/22M1523698},
}

@book{Gill,
    author = {Gill, A.},
    title = {Atmosphere-Ocean dynamics},
    publisher ={International Geophysics Series, vol. 30, Academic Press, New York} ,
    year = {1982},
}

@article {danchinliao,
    AUTHOR = {Danchin, R. and Liao, X.},
     TITLE = {On the well-posedness of the full low {M}ach number limit
              system in general critical {B}esov spaces},
   JOURNAL = {Commun. Contemp. Math.},
  FJOURNAL = {Communications in Contemporary Mathematics},
    VOLUME = {14},
      YEAR = {2012},
    NUMBER = {3},
     PAGES = {1250022},
      ISSN = {0219-1997,1793-6683},
   MRCLASS = {35Q35 (35B30)},
  MRNUMBER = {2943836},
       DOI = {10.1142/S0219199712500228},
       URL = {https://doi.org/10.1142/S0219199712500228},
}

@article {Abels,
    AUTHOR = {Abels, H.},
     TITLE = {On generalized solutions of two-phase flows for viscous
              incompressible fluids},
   JOURNAL = {Interfaces Free Bound.},
  FJOURNAL = {Interfaces and Free Boundaries. Mathematical Analysis,
              Computation and Applications},
    VOLUME = {9},
      YEAR = {2007},
    NUMBER = {1},
     PAGES = {31--65},
      ISSN = {1463-9963,1463-9971},
   MRCLASS = {35Q35 (35D05 49Q15 76D27 76D45)},
  MRNUMBER = {2317298},
MRREVIEWER = {Beno\^it\ P.\ Desjardins},
       DOI = {10.4171/IFB/155},
       URL = {https://doi.org/10.4171/IFB/155},
}

@book {Lemarie2016,
    AUTHOR = {Lemari\'e-Rieusset, P. G.},
     TITLE = {The {N}avier-{S}tokes problem in the 21st century},
 PUBLISHER = {CRC Press, Boca Raton, FL},
      YEAR = {2016},
      ISBN = {978-1-4665-6621-7},
   MRCLASS = {76D05 (35Q30 76D03)},
  MRNUMBER = {3469428},
MRREVIEWER = {Isabelle\ Gruais},
       DOI = {10.1201/b19556},
       URL = {https://doi.org/10.1201/b19556},
}

@article {Abidi09,
    AUTHOR = {Abidi, H.},
     TITLE = {Sur l'unicit\'e{} pour le syst\`eme de {B}oussinesq avec
              diffusion non lin\'eaire},
   JOURNAL = {J. Math. Pures Appl. (9)},
  FJOURNAL = {Journal de Math\'ematiques Pures et Appliqu\'ees. Neuvi\`eme
              S\'erie},
    VOLUME = {91},
      YEAR = {2009},
    NUMBER = {1},
     PAGES = {80--99},
      ISSN = {0021-7824},
   MRCLASS = {35Q35 (35B35 35Dxx 76D03)},
  MRNUMBER = {2487901},
MRREVIEWER = {Benedetta\ Ferrario},
       DOI = {10.1016/j.matpur.2008.09.004},
       URL = {https://doi.org/10.1016/j.matpur.2008.09.004},
}

@article {NiuLu24,
    AUTHOR = {Niu, D. and Wang, L.},
     TITLE = {Global well-posedness of three-dimensional incompressible
              {B}oussinesq system with temperature-dependent viscosity},
   JOURNAL = {J. Math. Phys.},
  FJOURNAL = {Journal of Mathematical Physics},
    VOLUME = {65},
      YEAR = {2024},
    NUMBER = {4},
     PAGES = {Paper No. 041504},
      ISSN = {0022-2488,1089-7658},
   MRCLASS = {35Q53 (35B65 35D35 76B03)},
  MRNUMBER = {4728449},
       DOI = {10.1063/5.0177834},
       URL = {https://doi.org/10.1063/5.0177834},
}

@article {AbidiZhang19,
    AUTHOR = {Abidi, H. and Zhang, P.},
     TITLE = {On the global well-posedness of 3-{D} {B}oussinesq system with
              variable viscosity},
   JOURNAL = {Chinese Ann. Math. Ser. B},
  FJOURNAL = {Chinese Annals of Mathematics. Series B},
    VOLUME = {40},
      YEAR = {2019},
    NUMBER = {5},
     PAGES = {643--688},
      ISSN = {0252-9599,1860-6261},
   MRCLASS = {35Q35 (35B30 35B65 76D03)},
  MRNUMBER = {4002303},
       DOI = {10.1007/s11401-019-0156-2},
       URL = {https://doi.org/10.1007/s11401-019-0156-2},
}

@book {feireisl,
    AUTHOR = {Feireisl, E.},
     TITLE = {Dynamics of viscous compressible fluids},
    SERIES = {Oxford Lecture Series in Mathematics and its Applications},
    VOLUME = {26},
 PUBLISHER = {Oxford University Press, Oxford},
      YEAR = {2004},
      ISBN = {0-19-852838-8},
   MRCLASS = {76N10 (35Q35 76N15)},
  MRNUMBER = {2040667},
MRREVIEWER = {Piotr\ Bogus\l aw\ Mucha},
}

@book {Tsai2018,
    AUTHOR = {Tsai, T.-P.},
     TITLE = {Lectures on {N}avier-{S}tokes equations},
    SERIES = {Graduate Studies in Mathematics},
    VOLUME = {192},
 PUBLISHER = {American Mathematical Society, Providence, RI},
      YEAR = {2018},
      ISBN = {978-1-4704-3096-2},
   MRCLASS = {35Q30 (35Q35 76D05)},
  MRNUMBER = {3822765},
MRREVIEWER = {Nader\ Masmoudi},
       DOI = {10.1090/gsm/192},
       URL = {https://doi.org/10.1090/gsm/192},
}

@article {astala2008,
    AUTHOR = {Astala, K. and Faraco, D. and Sz\'ekelyhidi, Jr.,
              L.},
     TITLE = {Convex integration and the {$L^p$} theory of elliptic
              equations},
   JOURNAL = {Ann. Sc. Norm. Super. Pisa Cl. Sci. (5)},
  FJOURNAL = {Annali della Scuola Normale Superiore di Pisa. Classe di
              Scienze. Serie V},
    VOLUME = {7},
      YEAR = {2008},
    NUMBER = {1},
     PAGES = {1--50},
      ISSN = {0391-173X,2036-2145},
   MRCLASS = {35J25 (30C62 35B65 35D10)},
  MRNUMBER = {2413671},
MRREVIEWER = {Leonid\ V.\ Kovalev},
}

@article {huang2013mhd,
    AUTHOR = {Huang, X. and Wang, Y.},
     TITLE = {Global strong solution to the 2{D} nonhomogeneous
              incompressible {MHD} system},
   JOURNAL = {J. Differential Equations},
  FJOURNAL = {Journal of Differential Equations},
    VOLUME = {254},
      YEAR = {2013},
    NUMBER = {2},
     PAGES = {511--527},
      ISSN = {0022-0396,1090-2732},
   MRCLASS = {35Q35 (35B65)},
  MRNUMBER = {2990041},
       DOI = {10.1016/j.jde.2012.08.029},
       URL = {https://doi.org/10.1016/j.jde.2012.08.029},
}

@article {LiuShen,
    AUTHOR = {Liu, C. and Shen, J.},
     TITLE = {A phase field model for the mixture of two incompressible fluids and its approximation by a {F}ourier-spectral method},
   JOURNAL = {Phys. D},
  FJOURNAL = {Physica D. Nonlinear Phenomena},
    VOLUME = {179},
      YEAR = {2003},
    NUMBER = {3-4},
     PAGES = {211--228},
      ISSN = {0167-2789,1872-8022},
   MRCLASS = {35Q35 (35K55 65M70 76D05 76M22 76T99 82B26)},
  MRNUMBER = {1984386},
MRREVIEWER = {Jing\ Xue\ Yin},
       DOI = {10.1016/S0167-2789(03)00030-7},
       URL = {https://doi.org/10.1016/S0167-2789(03)00030-7},
}

@book {lions1996mathematical2,
    AUTHOR = {Lions, P.-L.},
     TITLE = {Mathematical topics in fluid mechanics. {V}ol. 2},
    SERIES = {Oxford Lecture Series in Mathematics and its Applications},
    VOLUME = {10},
      NOTE = {Compressible models,
              Oxford Science Publications},
 PUBLISHER = {The Clarendon Press, Oxford University Press, New York},
      YEAR = {1998},
      ISBN = {0-19-851488-3},
   MRCLASS = {76-02 (35-02 35Q30 76N10)},
  MRNUMBER = {1637634},
MRREVIEWER = {Denis\ Serre},
}

@book {Majda,
    AUTHOR = {Majda, A.},
     TITLE = {Compressible fluid flow and systems of conservation laws in
              several space variables},
    SERIES = {Applied Mathematical Sciences},
    VOLUME = {53},
 PUBLISHER = {Springer-Verlag, New York},
      YEAR = {1984},
      ISBN = {0-387-96037-6},
   MRCLASS = {35L65 (76L05 76N10)},
  MRNUMBER = {748308},
MRREVIEWER = {Joel\ Smoller},
       DOI = {10.1007/978-1-4612-1116-7},
       URL = {https://doi.org/10.1007/978-1-4612-1116-7},
}

@article {liao2016global,
    AUTHOR = {Liao, X. and Zhang, P.},
     TITLE = {On the global regularity of the two-dimensional density patch
              for inhomogeneous incompressible viscous flow},
   JOURNAL = {Arch. Ration. Mech. Anal.},
  FJOURNAL = {Archive for Rational Mechanics and Analysis},
    VOLUME = {220},
      YEAR = {2016},
    NUMBER = {3},
     PAGES = {937--981},
      ISSN = {0003-9527,1432-0673},
   MRCLASS = {35Q30 (35B65 76D03 76D05)},
  MRNUMBER = {3466838},
MRREVIEWER = {Luigi\ Carlo\ Berselli},
       DOI = {10.1007/s00205-015-0945-z},
       URL = {https://doi.org/10.1007/s00205-015-0945-z},
}

@article {abidi2015global,
    AUTHOR = {Abidi, H. and Zhang, P.},
     TITLE = {Global well-posedness of 3-{D} density-dependent
              {N}avier-{S}tokes system with variable viscosity},
   JOURNAL = {Sci. China Math.},
  FJOURNAL = {Science China. Mathematics},
    VOLUME = {58},
      YEAR = {2015},
    NUMBER = {6},
     PAGES = {1129--1150},
      ISSN = {1674-7283,1869-1862},
   MRCLASS = {35Q30 (35B30 76D03 76D05)},
  MRNUMBER = {3344050},
       DOI = {10.1007/s11425-015-4983-7},
       URL = {https://doi.org/10.1007/s11425-015-4983-7},
}

@book {antontsev1989boundary,
    AUTHOR = {Antontsev, S. N. and Kazhikhov, A. V. and Monakhov, V. N.},
     TITLE = {Boundary value problems in mechanics of nonhomogeneous fluids},
    SERIES = {Studies in Mathematics and its Applications},
    VOLUME = {22},
      NOTE = {Translated from the Russian},
 PUBLISHER = {North-Holland Publishing Co., Amsterdam},
      YEAR = {1990},
      ISBN = {0-444-88382-7},
   MRCLASS = {76D05 (35Q35 76N15 76S05 76T05)},
  MRNUMBER = {1035212},
MRREVIEWER = {Andro\ Mikeli\'c},
}

@book {lions1996mathematical,
    AUTHOR = {Lions, P.-L.},
     TITLE = {Mathematical topics in fluid mechanics. {V}ol. 1},
    SERIES = {Oxford Lecture Series in Mathematics and its Applications},
    VOLUME = {3},
      NOTE = {Incompressible models,
              Oxford Science Publications},
 PUBLISHER = {The Clarendon Press, Oxford University Press, New York},
      YEAR = {1996},
      ISBN = {0-19-851487-5},
   MRCLASS = {76-02 (35Q30 35Q35 76D05)},
  MRNUMBER = {1422251},
MRREVIEWER = {Denis\ Serre},
}

@article {desjardins1997regularity,
    AUTHOR = {Desjardins, B.},
     TITLE = {Regularity results for two-dimensional flows of multiphase
              viscous fluids},
   JOURNAL = {Arch. Ration. Mech. Anal.},
  FJOURNAL = {Archive for Rational Mechanics and Analysis},
    VOLUME = {137},
      YEAR = {1997},
    NUMBER = {2},
     PAGES = {135--158},
      ISSN = {0003-9527},
   MRCLASS = {76T05 (35Q35 76N10)},
  MRNUMBER = {1463792},
MRREVIEWER = {Denis\ Serre},
       DOI = {10.1007/s002050050025},
       URL = {https://doi.org/10.1007/s002050050025},
}

@article {gui2009global,
    AUTHOR = {Gui, G. and Zhang, P.},
     TITLE = {Global smooth solutions to the 2-{D} inhomogeneous
              {N}avier-{S}tokes equations with variable viscosity},
   JOURNAL = {Chinese Ann. Math. Ser. B},
  FJOURNAL = {Chinese Annals of Mathematics. Series B},
    VOLUME = {30},
      YEAR = {2009},
    NUMBER = {5},
     PAGES = {607--630},
      ISSN = {0252-9599,1860-6261},
   MRCLASS = {35Q30 (35B65 76D03)},
  MRNUMBER = {2601489},
MRREVIEWER = {Gang\ Wu},
       DOI = {10.1007/s11401-009-0027-3},
       URL = {https://doi.org/10.1007/s11401-009-0027-3},
}

@article {huang2013global,
    AUTHOR = {Huang, J. and Paicu, M. and Zhang, P.},
     TITLE = {Global solutions to 2-{D} inhomogeneous {N}avier-{S}tokes
              system with general velocity},
   JOURNAL = {J. Math. Pures Appl. (9)},
  FJOURNAL = {Journal de Math\'ematiques Pures et Appliqu\'ees. Neuvi\`eme
              S\'erie},
    VOLUME = {100},
      YEAR = {2013},
    NUMBER = {6},
     PAGES = {806--831},
      ISSN = {0021-7824,1776-3371},
   MRCLASS = {35Q30 (35B30 47F05 76D03 76D05)},
  MRNUMBER = {3125269},
MRREVIEWER = {Sergey\ Nikolaevich\ Alekseenko},
       DOI = {10.1016/j.matpur.2013.03.003},
       URL = {https://doi.org/10.1016/j.matpur.2013.03.003},
}

@article {abidi2015global2D,
    AUTHOR = {Abidi, H. and Zhang, P.},
     TITLE = {On the global well-posedness of 2-{D} inhomogeneous
              incompressible {N}avier-{S}tokes system with variable viscous
              coefficient},
   JOURNAL = {J. Differential Equations},
  FJOURNAL = {Journal of Differential Equations},
    VOLUME = {259},
      YEAR = {2015},
    NUMBER = {8},
     PAGES = {3755--3802},
      ISSN = {0022-0396,1090-2732},
   MRCLASS = {35Q30 (35B30 76D03 76D05)},
  MRNUMBER = {3369261},
       DOI = {10.1016/j.jde.2015.05.002},
       URL = {https://doi.org/10.1016/j.jde.2015.05.002},
}

@article {huang2012decay,
    AUTHOR = {Huang, J. and Paicu, M.},
     TITLE = {Decay estimates of global solution to 2{D} incompressible
              {N}avier-{S}tokes equations with variable viscosity},
   JOURNAL = {Discrete Contin. Dyn. Syst.},
  FJOURNAL = {Discrete and Continuous Dynamical Systems. Series A},
    VOLUME = {34},
      YEAR = {2014},
    NUMBER = {11},
     PAGES = {4647--4669},
      ISSN = {1078-0947,1553-5231},
   MRCLASS = {35Q30 (35B30 35B40 76D05)},
  MRNUMBER = {3223823},
MRREVIEWER = {Ionu\c t\ Munteanu},
       DOI = {10.3934/dcds.2014.34.4647},
       URL = {https://doi.org/10.3934/dcds.2014.34.4647},
}

@article {paicu2020striated,
    AUTHOR = {Paicu, M. and Zhang, P.},
     TITLE = {Striated regularity of 2-{D} inhomogeneous incompressible
              {N}avier-{S}tokes system with variable viscosity},
   JOURNAL = {Comm. Math. Phys.},
  FJOURNAL = {Communications in Mathematical Physics},
    VOLUME = {376},
      YEAR = {2020},
    NUMBER = {1},
     PAGES = {385--439},
      ISSN = {0010-3616,1432-0916},
   MRCLASS = {35Q30 (35B65)},
  MRNUMBER = {4093854},
MRREVIEWER = {Luigi\ Carlo\ Berselli},
       DOI = {10.1007/s00220-019-03446-z},
       URL = {https://doi.org/10.1007/s00220-019-03446-z},
}

@article {leray1933essai,
    AUTHOR = {Leray, J.},
     TITLE = {Sur le mouvement d'un liquide visqueux emplissant l'espace},
   JOURNAL = {Acta Math.},
  FJOURNAL = {Acta Mathematica},
    VOLUME = {63},
      YEAR = {1934},
    NUMBER = {1},
     PAGES = {193--248},
      ISSN = {0001-5962,1871-2509},
   MRCLASS = {99-04},
  MRNUMBER = {1555394},
       DOI = {10.1007/BF02547354},
       URL = {https://doi.org/10.1007/BF02547354},
}

@article {danchin2017persistence,
    AUTHOR = {Danchin, R. and Zhang, X.},
     TITLE = {On the persistence of {H}\"older regular patches of density
              for the inhomogeneous {N}avier-{S}tokes equations},
   JOURNAL = {J. \'Ec. polytech. Math.},
  FJOURNAL = {Journal de l'\'Ecole polytechnique. Math\'ematiques},
    VOLUME = {4},
      YEAR = {2017},
     PAGES = {781--811},
      ISSN = {2429-7100,2270-518X},
   MRCLASS = {35Q30 (35B65 76D05)},
  MRNUMBER = {3665613},
MRREVIEWER = {Iryna\ Ryzhkova},
       DOI = {10.5802/jep.56},
       URL = {https://doi.org/10.5802/jep.56},
}

@article {gancedo2018global,
    AUTHOR = {Gancedo, F. and Garc\'ia-Ju\'arez, E.},
     TITLE = {Global regularity of 2{D} density patches for inhomogeneous
              {N}avier-{S}tokes},
   JOURNAL = {Arch. Ration. Mech. Anal.},
  FJOURNAL = {Archive for Rational Mechanics and Analysis},
    VOLUME = {229},
      YEAR = {2018},
    NUMBER = {1},
     PAGES = {339--360},
      ISSN = {0003-9527,1432-0673},
   MRCLASS = {76D05 (35B65 35Q30 76D03)},
  MRNUMBER = {3799095},
MRREVIEWER = {Tsukasa\ Iwabuchi},
       DOI = {10.1007/s00205-018-1218-4},
       URL = {https://doi.org/10.1007/s00205-018-1218-4},
}

@article {danchin2019incompressible,
    AUTHOR = {Danchin, R. and Mucha, P. B.},
     TITLE = {The incompressible {N}avier-{S}tokes equations in vacuum},
   JOURNAL = {Comm. Pure Appl. Math.},
  FJOURNAL = {Communications on Pure and Applied Mathematics},
    VOLUME = {72},
      YEAR = {2019},
    NUMBER = {7},
     PAGES = {1351--1385},
      ISSN = {0010-3640,1097-0312},
   MRCLASS = {76D05 (35Q30)},
  MRNUMBER = {3957394},
MRREVIEWER = {Gudrun\ Th\"ater},
       DOI = {10.1002/cpa.21806},
       URL = {https://doi.org/10.1002/cpa.21806},
}

@article {gancedo2023global,
    AUTHOR = {Gancedo, F. and Garc\'ia-Ju\'arez, E.},
     TITLE = {Global regularity of 2{D} {N}avier-{S}tokes free boundary with
              small viscosity contrast},
   JOURNAL = {Ann. Inst. H. Poincar\'e{} C Anal. Non Lin\'eaire},
  FJOURNAL = {Annales de l'Institut Henri Poincar\'e{} C. Analyse Non
              Lin\'eaire},
    VOLUME = {40},
      YEAR = {2023},
    NUMBER = {6},
     PAGES = {1319--1352},
      ISSN = {0294-1449,1873-1430},
   MRCLASS = {35Q30 (76D03 76D05)},
  MRNUMBER = {4656417},
MRREVIEWER = {Zijin\ Li},
       DOI = {10.4171/aihpc/74},
       URL = {https://doi.org/10.4171/aihpc/74},
}

@article {denisova2007global,
    AUTHOR = {Denisova, I. V.},
     TITLE = {Global solvability of a problem on two fluid motion without
              surface tension},
   JOURNAL = {Zap. Nauchn. Sem. S.-Peterburg. Otdel. Mat. Inst. Steklov.
              (POMI)},
  FJOURNAL = {Rossi\u iskaya Akademiya Nauk. Sankt-Peterburgskoe Otdelenie.
              Matematicheski\u i\ Institut im. V. A. Steklova. Zapiski
              Nauchnykh Seminarov (POMI)},
    VOLUME = {348},
      YEAR = {2007},
     PAGES = {19--39},
      ISSN = {0373-2703},
   MRCLASS = {35Q35 (76D03 76D05)},
  MRNUMBER = {2743013},
MRREVIEWER = {Koji\ Ohkitani},
       DOI = {10.1007/s10958-008-9096-1},
       URL = {https://doi.org/10.1007/s10958-008-9096-1},
}

@article {saito2020some,
    AUTHOR = {Saito, H. and Shibata, Y. and Zhang, X.},
     TITLE = {Some free boundary problem for two-phase inhomogeneous
              incompressible flows},
   JOURNAL = {SIAM J. Math. Anal.},
  FJOURNAL = {SIAM Journal on Mathematical Analysis},
    VOLUME = {52},
      YEAR = {2020},
    NUMBER = {4},
     PAGES = {3397--3443},
      ISSN = {0036-1410,1095-7154},
   MRCLASS = {35R35 (35B65 35Q30 76D05)},
  MRNUMBER = {4127102},
MRREVIEWER = {Keiichi\ Watanabe},
       DOI = {10.1137/18M1225239},
       URL = {https://doi.org/10.1137/18M1225239},
}

@article {huang2013global2,
    AUTHOR = {Huang, J. and Paicu, M. and Zhang, P.},
     TITLE = {Global well-posedness of incompressible inhomogeneous fluid
              systems with bounded density or non-{L}ipschitz velocity},
   JOURNAL = {Arch. Ration. Mech. Anal.},
  FJOURNAL = {Archive for Rational Mechanics and Analysis},
    VOLUME = {209},
      YEAR = {2013},
    NUMBER = {2},
     PAGES = {631--682},
      ISSN = {0003-9527,1432-0673},
   MRCLASS = {35Q35 (35B30)},
  MRNUMBER = {3056619},
MRREVIEWER = {Amin\ Esfahani},
       DOI = {10.1007/s00205-013-0624-x},
       URL = {https://doi.org/10.1007/s00205-013-0624-x},
}

@article {liu2019global,
    AUTHOR = {Liao, X. and Liu, Y.},
     TITLE = {Global regularity of three-dimensional density patches for
              inhomogeneous incompressible viscous flow},
   JOURNAL = {Sci. China Math.},
  FJOURNAL = {Science China. Mathematics},
    VOLUME = {62},
      YEAR = {2019},
    NUMBER = {9},
     PAGES = {1749--1764},
      ISSN = {1674-7283,1869-1862},
   MRCLASS = {76D03 (35B65 35Q35)},
  MRNUMBER = {3998382},
MRREVIEWER = {Yu\ Zhu\ Wang},
       DOI = {10.1007/s11425-017-9196-7},
       URL = {https://doi.org/10.1007/s11425-017-9196-7},
}

@book {denisova2021motion,
    AUTHOR = {Denisova, I. V. and Solonnikov, V. A.},
     TITLE = {Motion of a drop in an incompressible fluid},
    SERIES = {Advances in Mathematical Fluid Mechanics},
      NOTE = {Lecture Notes in Mathematical Fluid Mechanics},
 PUBLISHER = {Birkh\"auser/Springer, Cham},
      YEAR = {2021},
      ISBN = {978-3-030-70052-2; 978-3-030-70053-9},
   MRCLASS = {35Q30 (76D03)},
  MRNUMBER = {4367399},
       DOI = {10.1007/978-3-030-70053-9},
       URL = {https://doi.org/10.1007/978-3-030-70053-9},
}

@article {he2020solvability,
    AUTHOR = {He, Z. and Liao, X.},
     TITLE = {Solvability of the two-dimensional stationary incompressible
              inhomogeneous {N}avier-{S}tokes equations with variable
              viscosity coefficient},
   JOURNAL = {Commun. Contemp. Math.},
  FJOURNAL = {Communications in Contemporary Mathematics},
    VOLUME = {26},
      YEAR = {2024},
    NUMBER = {7},
     PAGES = {Paper No. 2350039},
      ISSN = {0219-1997,1793-6683},
   MRCLASS = {35Q30 (76D03)},
  MRNUMBER = {4760554},
       DOI = {10.1142/S0219199723500396},
       URL = {https://doi.org/10.1142/S0219199723500396},
}

@article {chemin1991mouvement,
    AUTHOR = {Chemin, J.-Y.},
     TITLE = {Sur le mouvement des particules d'un fluide parfait
              incompressible bidimensionnel},
   JOURNAL = {Invent. Math.},
  FJOURNAL = {Inventiones Mathematicae},
    VOLUME = {103},
      YEAR = {1991},
    NUMBER = {3},
     PAGES = {599--629},
      ISSN = {0020-9910,1432-1297},
   MRCLASS = {35Q30 (35Q35 76C05)},
  MRNUMBER = {1091620},
MRREVIEWER = {Valery\ Covachev},
       DOI = {10.1007/BF01239528},
       URL = {https://doi.org/10.1007/BF01239528},
}

@article {chemin1993persistance,
    AUTHOR = {Chemin, J.-Y.},
     TITLE = {Persistance de structures g\'eom\'etriques dans les fluides
              incompressibles bidimensionnels},
   JOURNAL = {Ann. Sci. \'Ecole Norm. Sup. (4)},
  FJOURNAL = {Annales Scientifiques de l'\'Ecole Normale Sup\'erieure.
              Quatri\`eme S\'erie},
    VOLUME = {26},
      YEAR = {1993},
    NUMBER = {4},
     PAGES = {517--542},
      ISSN = {0012-9593},
   MRCLASS = {35Q35 (76C05)},
  MRNUMBER = {1235440},
MRREVIEWER = {Paolo\ Secchi},
       URL = {http://www.numdam.org/item?id=ASENS_1993_4_26_4_517_0},
}

@article {calderon1965commutators,
    AUTHOR = {Calder\'on, A.-P.},
     TITLE = {Commutators of singular integral operators},
   JOURNAL = {Proc. Nat. Acad. Sci. U.S.A.},
  FJOURNAL = {Proceedings of the National Academy of Sciences of the United
              States of America},
    VOLUME = {53},
      YEAR = {1965},
     PAGES = {1092--1099},
      ISSN = {0027-8424},
   MRCLASS = {47.70},
  MRNUMBER = {177312},
MRREVIEWER = {R.\ T.\ Seeley},
       DOI = {10.1073/pnas.53.5.1092},
       URL = {https://doi.org/10.1073/pnas.53.5.1092},
}

@article {wiegner1987decay,
    AUTHOR = {Wiegner, M.},
     TITLE = {Decay results for weak solutions of the {N}avier-{S}tokes
              equations on {${\bf R}^n$}},
   JOURNAL = {J. London Math. Soc. (2)},
  FJOURNAL = {Journal of the London Mathematical Society. Second Series},
    VOLUME = {35},
      YEAR = {1987},
    NUMBER = {2},
     PAGES = {303--313},
      ISSN = {0024-6107,1469-7750},
   MRCLASS = {35Q10 (35B40)},
  MRNUMBER = {881519},
MRREVIEWER = {L.\ Hsiao},
       DOI = {10.1112/jlms/s2-35.2.303},
       URL = {https://doi.org/10.1112/jlms/s2-35.2.303},
}

@article {dong2011parabolic,
    AUTHOR = {Dong, H. and Kim, D.},
     TITLE = {Parabolic and elliptic systems in divergence form with
              variably partially {BMO} coefficients},
   JOURNAL = {SIAM J. Math. Anal.},
  FJOURNAL = {SIAM Journal on Mathematical Analysis},
    VOLUME = {43},
      YEAR = {2011},
    NUMBER = {3},
     PAGES = {1075--1098},
      ISSN = {0036-1410,1095-7154},
   MRCLASS = {35K45 (35A01 35J57 35K59 35R05)},
  MRNUMBER = {2800569},
MRREVIEWER = {Luca\ Lorenzi},
       DOI = {10.1137/100794614},
       URL = {https://doi.org/10.1137/100794614},
}

@article {abidi2021global,
    AUTHOR = {Abidi, H. and Gui, G.},
     TITLE = {Global well-posedness for the 2-{D} inhomogeneous
              incompressible {N}avier-{S}tokes system with large initial
              data in critical spaces},
   JOURNAL = {Arch. Ration. Mech. Anal.},
  FJOURNAL = {Archive for Rational Mechanics and Analysis},
    VOLUME = {242},
      YEAR = {2021},
    NUMBER = {3},
     PAGES = {1533--1570},
      ISSN = {0003-9527,1432-0673},
   MRCLASS = {76D05 (35Q30 76D03)},
  MRNUMBER = {4334732},
MRREVIEWER = {Hermenegildo\ Borges\ de Oliveira},
       DOI = {10.1007/s00205-021-01710-y},
       URL = {https://doi.org/10.1007/s00205-021-01710-y},
}

@article {denisova2014global,
    AUTHOR = {Denisova, I. V.},
     TITLE = {Global {$L_2$}-solvability of a problem governing two-phase
              fluid motion without surface tension},
   JOURNAL = {Port. Math.},
  FJOURNAL = {Portugaliae Mathematica. A Journal of the Portuguese
              Mathematical Society},
    VOLUME = {71},
      YEAR = {2014},
    NUMBER = {1},
     PAGES = {1--24},
      ISSN = {0032-5155,1662-2758},
   MRCLASS = {35Q30 (35R35 76D05 76Txx)},
  MRNUMBER = {3194642},
MRREVIEWER = {Yoshiaki\ Teramoto},
       DOI = {10.4171/PM/1938},
       URL = {https://doi.org/10.4171/PM/1938},
}

@article{pruss2011analytic,
  title={Analytic solutions for the two-phase Navier-Stokes equations with surface tension and gravity},
  author={Pr{\"u}ss, J. and Simonett, G.},
  year={2011},
  journal={arXiv preprint arXiv:0908.3332v1}
}

@article {denisova1996classical,
    AUTHOR = {Denisova, I. V. and Solonnikov, V. A.},
     TITLE = {Classical solvability of the problem of the motion of two
              viscous incompressible fluids},
   JOURNAL = {Algebra i Analiz},
  FJOURNAL = {Rossi\u iskaya Akademiya Nauk. Algebra i Analiz},
    VOLUME = {7},
      YEAR = {1995},
    NUMBER = {5},
     PAGES = {101--142},
      ISSN = {0234-0852},
   MRCLASS = {35Q35 (35R35 76D45)},
  MRNUMBER = {1365814},
MRREVIEWER = {Vladimir\ V.\ Shelukhin},
}

@article {denisova2012global,
    AUTHOR = {Denisova, I. V. and Solonnikov, V. A.},
     TITLE = {Global solvability of the problem of the motion of two
              incompressible capillary fluids in a container},
   JOURNAL = {Zap. Nauchn. Sem. S.-Peterburg. Otdel. Mat. Inst. Steklov.
              (POMI)},
  FJOURNAL = {Rossi\u iskaya Akademiya Nauk. Sankt-Peterburgskoe Otdelenie.
              Matematicheski\u i\ Institut im. V. A. Steklova. Zapiski
              Nauchnykh Seminarov (POMI)},
    VOLUME = {397},
      YEAR = {2011},
     PAGES = {20--52},
      ISSN = {0373-2703},
   MRCLASS = {35R35 (35A01 35Q30 76D45)},
  MRNUMBER = {2870107},
MRREVIEWER = {J\"urgen\ Socolowsky},
       DOI = {10.1007/s10958-012-0951-8},
       URL = {https://doi.org/10.1007/s10958-012-0951-8},
}

@article {solonnikov2014p,
    AUTHOR = {Solonnikov, V. A.},
     TITLE = {{$L_p$}-theory of the problem of motion of two incompressible
              capillary fluids in a container},
   JOURNAL = {J. Math. Sci. (N.Y.)},
  FJOURNAL = {Journal of Mathematical Sciences (New York)},
    VOLUME = {198},
      YEAR = {2014},
    NUMBER = {6},
     PAGES = {761--827},
      ISSN = {1072-3374,1573-8795},
   MRCLASS = {35Q35 (35B45 35R35 76D45)},
  MRNUMBER = {3391324},
       DOI = {10.1007/s10958-014-1824-0},
       URL = {https://doi.org/10.1007/s10958-014-1824-0},
}

@article {shen2005bounds,
    AUTHOR = {Shen, Z.},
     TITLE = {Bounds of {R}iesz transforms on {$L^p$} spaces for second
              order elliptic operators},
   JOURNAL = {Ann. Inst. Fourier (Grenoble)},
  FJOURNAL = {Universit\'e{} de Grenoble. Annales de l'Institut Fourier},
    VOLUME = {55},
      YEAR = {2005},
    NUMBER = {1},
     PAGES = {173--197},
      ISSN = {0373-0956,1777-5310},
   MRCLASS = {35J15 (35B45 42B20)},
  MRNUMBER = {2141694},
MRREVIEWER = {Gabriella\ Bogn\'ar},
       DOI = {10.5802/aif.2094},
       URL = {https://doi.org/10.5802/aif.2094},
}

@article {barton2016gradient,
    AUTHOR = {Barton, A.},
     TITLE = {Gradient estimates and the fundamental solution for
              higher-order elliptic systems with rough coefficients},
   JOURNAL = {Manuscripta Math.},
  FJOURNAL = {Manuscripta Mathematica},
    VOLUME = {151},
      YEAR = {2016},
    NUMBER = {3-4},
     PAGES = {375--418},
      ISSN = {0025-2611,1432-1785},
   MRCLASS = {35J48 (31B10 35B45 35B65 35C15 35J08)},
  MRNUMBER = {3556825},
MRREVIEWER = {Siegfried\ Carl},
       DOI = {10.1007/s00229-016-0839-x},
       URL = {https://doi.org/10.1007/s00229-016-0839-x},
}

@article {huang2014global,
    AUTHOR = {Huang, X. and Wang, Y.},
     TITLE = {Global strong solution with vacuum to the two dimensional
              density-dependent {N}avier-{S}tokes system},
   JOURNAL = {SIAM J. Math. Anal.},
  FJOURNAL = {SIAM Journal on Mathematical Analysis},
    VOLUME = {46},
      YEAR = {2014},
    NUMBER = {3},
     PAGES = {1771--1788},
      ISSN = {0036-1410,1095-7154},
   MRCLASS = {35Q35 (35B65 76D05)},
  MRNUMBER = {3200422},
MRREVIEWER = {J\'auber\ Cavalcante\ Oliveira},
       DOI = {10.1137/120894865},
       URL = {https://doi.org/10.1137/120894865},
}

@article {huang2015global,
    AUTHOR = {Huang, X. and Wang, Y.},
     TITLE = {Global strong solution of 3{D} inhomogeneous {N}avier-{S}tokes
              equations with density-dependent viscosity},
   JOURNAL = {J. Differential Equations},
  FJOURNAL = {Journal of Differential Equations},
    VOLUME = {259},
      YEAR = {2015},
    NUMBER = {4},
     PAGES = {1606--1627},
      ISSN = {0022-0396,1090-2732},
   MRCLASS = {35Q35 (35B65 35D35 76D03 76D05)},
  MRNUMBER = {3345862},
MRREVIEWER = {Yongzhong\ Sun},
       DOI = {10.1016/j.jde.2015.03.008},
       URL = {https://doi.org/10.1016/j.jde.2015.03.008},
}

@article {zhang2015global,
    AUTHOR = {Zhang, J.},
     TITLE = {Global well-posedness for the incompressible {N}avier-{S}tokes
              equations with density-dependent viscosity coefficient},
   JOURNAL = {J. Differential Equations},
  FJOURNAL = {Journal of Differential Equations},
    VOLUME = {259},
      YEAR = {2015},
    NUMBER = {5},
     PAGES = {1722--1742},
      ISSN = {0022-0396,1090-2732},
   MRCLASS = {35Q30 (35B30 35B40 35B45 76D03 76D05 76D09)},
  MRNUMBER = {3349417},
MRREVIEWER = {Ionu\c t\ Munteanu},
       DOI = {10.1016/j.jde.2015.03.011},
       URL = {https://doi.org/10.1016/j.jde.2015.03.011},
}

@article {he2021global,
    AUTHOR = {He, C. and Li, J. and L\"u, B.},
     TITLE = {Global well-posedness and exponential stability of 3{D}
              {N}avier-{S}tokes equations with density-dependent viscosity
              and vacuum in unbounded domains},
   JOURNAL = {Arch. Ration. Mech. Anal.},
  FJOURNAL = {Archive for Rational Mechanics and Analysis},
    VOLUME = {239},
      YEAR = {2021},
    NUMBER = {3},
     PAGES = {1809--1835},
      ISSN = {0003-9527,1432-0673},
   MRCLASS = {35Q30 (76D05)},
  MRNUMBER = {4215202},
       DOI = {10.1007/s00205-020-01604-5},
       URL = {https://doi.org/10.1007/s00205-020-01604-5},
}

@article {cho2004unique,
    AUTHOR = {Cho, Y. and Kim, H.},
     TITLE = {Unique solvability for the density-dependent {N}avier-{S}tokes
              equations},
   JOURNAL = {Nonlinear Anal.},
  FJOURNAL = {Nonlinear Analysis. Theory, Methods \& Applications. An
              International Multidisciplinary Journal},
    VOLUME = {59},
      YEAR = {2004},
    NUMBER = {4},
     PAGES = {465--489},
      ISSN = {0362-546X,1873-5215},
   MRCLASS = {35Q30 (76D03 76D05)},
  MRNUMBER = {2094425},
MRREVIEWER = {Catherine\ Sulem},
       DOI = {10.1016/j.na.2004.07.020},
       URL = {https://doi.org/10.1016/j.na.2004.07.020},
}

@article {liu2020global,
    AUTHOR = {Liu, Y.},
     TITLE = {Global well-posedness of the 2{D} incompressible
              {N}avier-{S}tokes equations with density-dependent viscosity
              coefficient},
   JOURNAL = {Nonlinear Anal. Real World Appl.},
  FJOURNAL = {Nonlinear Analysis. Real World Applications. An International
              Multidisciplinary Journal},
    VOLUME = {56},
      YEAR = {2020},
     PAGES = {103156},
      ISSN = {1468-1218,1878-5719},
   MRCLASS = {76D03 (35B30 35B40 35Q35 76D05)},
  MRNUMBER = {4110294},
       DOI = {10.1016/j.nonrwa.2020.103156},
       URL = {https://doi.org/10.1016/j.nonrwa.2020.103156},
}

@book {pruss2016moving,
    AUTHOR = {Pr\"uss, J. and Simonett, G.},
     TITLE = {Moving interfaces and quasilinear parabolic evolution
              equations},
    SERIES = {Monographs in Mathematics},
    VOLUME = {105},
 PUBLISHER = {Birkh\"auser/Springer, Cham},
      YEAR = {2016},
      ISBN = {978-3-319-27697-7; 978-3-319-27698-4},
   MRCLASS = {35-02 (35B30 35K93 35R35 47F05 58Jxx 76A15 80A22)},
  MRNUMBER = {3524106},
MRREVIEWER = {Glen\ E.\ Wheeler},
       DOI = {10.1007/978-3-319-27698-4},
       URL = {https://doi.org/10.1007/978-3-319-27698-4},
}

@book {chemin1998perfect,
    AUTHOR = {Chemin, J.-Y.},
     TITLE = {Perfect incompressible fluids},
    SERIES = {Oxford Lecture Series in Mathematics and its Applications},
    VOLUME = {14},
      NOTE = {Translated from the 1995 French original by Isabelle Gallagher
              and Dragos Iftimie},
 PUBLISHER = {The Clarendon Press, Oxford University Press, New York},
      YEAR = {1998},
      ISBN = {0-19-850397-0},
   MRCLASS = {76B47 (35-02 35Q35 76-02)},
  MRNUMBER = {1688875},
}

@article {bertozzi1993global,
    AUTHOR = {Bertozzi, A. L. and Constantin, P.},
     TITLE = {Global regularity for vortex patches},
   JOURNAL = {Comm. Math. Phys.},
  FJOURNAL = {Communications in Mathematical Physics},
    VOLUME = {152},
      YEAR = {1993},
    NUMBER = {1},
     PAGES = {19--28},
      ISSN = {0010-3616,1432-0916},
   MRCLASS = {35Q35 (76C99)},
  MRNUMBER = {1207667},
MRREVIEWER = {A.\ Elcrat},
       URL = {http://projecteuclid.org/euclid.cmp/1104252307},
}

@article {gamblin1995three,
    AUTHOR = {Gamblin, P. and Saint Raymond, X.},
     TITLE = {On three-dimensional vortex patches},
   JOURNAL = {Bull. Soc. Math. France},
  FJOURNAL = {Bulletin de la Soci\'et\'e{} Math\'ematique de France},
    VOLUME = {123},
      YEAR = {1995},
    NUMBER = {3},
     PAGES = {375--424},
      ISSN = {0037-9484,2102-622X},
   MRCLASS = {35Q35 (76C05)},
  MRNUMBER = {1373741},
MRREVIEWER = {Long\ An\ Ying},
       URL = {http://www.numdam.org/item?id=BSMF_1995__123_3_375_0},
}

@incollection {gerard1992resultats,
    AUTHOR = {G\'erard, P.},
     TITLE = {R\'esultats r\'ecents sur les fluides parfaits incompressibles
              bidimensionnels (d'apr\`es {J}.-{Y}.\ {C}hemin et {J}.-{M}.\
              {D}elort)},
      NOTE = {S\'eminaire Bourbaki, Vol.\ 1991/92},
   JOURNAL = {Ast\'erisque},
  FJOURNAL = {Ast\'erisque},
    NUMBER = {206},
      YEAR = {1992},
     PAGES = {Exp. No. 757, 5, 411--444},
      ISSN = {0303-1179,2492-5926},
   MRCLASS = {76C05 (35Q35)},
  MRNUMBER = {1206075},
MRREVIEWER = {Paul\ G.\ Schmidt},
}

@book {bahouri2011fourier,
    AUTHOR = {Bahouri, H. and Chemin, J.-Y. and Danchin, R.},
     TITLE = {Fourier analysis and nonlinear partial differential equations},
    SERIES = {Grundlehren der mathematischen Wissenschaften [Fundamental
              Principles of Mathematical Sciences]},
    VOLUME = {343},
 PUBLISHER = {Springer, Heidelberg},
      YEAR = {2011},
      ISBN = {978-3-642-16829-1},
   MRCLASS = {35-02 (35L72 35Q30 42-02 42B37 76B03 76D03 76N10)},
  MRNUMBER = {2768550},
MRREVIEWER = {Peter\ R.\ Massopust},
       DOI = {10.1007/978-3-642-16830-7},
       URL = {https://doi.org/10.1007/978-3-642-16830-7},
}

@article {danchin2020well,
    AUTHOR = {Danchin, R. and Fanelli, F. and Paicu, M.},
     TITLE = {A well-posedness result for viscous compressible fluids with
              only bounded density},
   JOURNAL = {Anal. PDE},
  FJOURNAL = {Analysis \& PDE},
    VOLUME = {13},
      YEAR = {2020},
    NUMBER = {1},
     PAGES = {275--316},
      ISSN = {2157-5045,1948-206X},
   MRCLASS = {35Q35 (35B30 35B65 76N10)},
  MRNUMBER = {4047647},
MRREVIEWER = {V\'aclav\ M\'acha},
       DOI = {10.2140/apde.2020.13.275},
       URL = {https://doi.org/10.2140/apde.2020.13.275},
}

@article {fanelli2012conservation,
    AUTHOR = {Fanelli, F.},
     TITLE = {Conservation of geometric structures for non-homogeneous
              inviscid incompressible fluids},
   JOURNAL = {Comm. Partial Differential Equations},
  FJOURNAL = {Communications in Partial Differential Equations},
    VOLUME = {37},
      YEAR = {2012},
    NUMBER = {9},
     PAGES = {1553--1595},
      ISSN = {0360-5302,1532-4133},
   MRCLASS = {35Q35 (35B65 35Q31 76B03)},
  MRNUMBER = {2969490},
MRREVIEWER = {Rapha\"el\ Danchin},
       DOI = {10.1080/03605302.2012.698343},
       URL = {https://doi.org/10.1080/03605302.2012.698343},
}

@article{hao2024density,
  title={On the density patch problem for the 2-D inhomogeneous Navier-Stokes equations},
  author={Hao, T. and Shao, F. and Wei, D. and Zhang, Z.},
  journal={arXiv preprint arXiv:2406.07984},
  year={2024}
}

@book {krylov2008lectures,
    AUTHOR = {Krylov, N. V.},
     TITLE = {Lectures on elliptic and parabolic equations in {S}obolev
              spaces},
    SERIES = {Graduate Studies in Mathematics},
    VOLUME = {96},
 PUBLISHER = {American Mathematical Society, Providence, RI},
      YEAR = {2008},
      ISBN = {978-0-8218-4684-1},
   MRCLASS = {35-01 (35J15 35K10 35S05 46E35 46N20)},
  MRNUMBER = {2435520},
MRREVIEWER = {Vicen\c tiu\ D.\ R\u adulescu},
       DOI = {10.1090/gsm/096},
       URL = {https://doi.org/10.1090/gsm/096},
}

@misc{CMI,
    author = {Clay Mathematics Institute},
    title = {The Millenium Prize Problems},
    howpublished = "\url{https://www.claymath.org/millennium/navier-stokes-equation/}",
    note = "[Online, accessed 31-07-2024]"
}

@article {chae2022global,
    AUTHOR = {Chae, D. and Miao, Q. and Xue, L.},
     TITLE = {Global regularity of nondiffusive temperature fronts for the
              two-dimensional viscous {B}oussinesq system},
   JOURNAL = {SIAM J. Math. Anal.},
  FJOURNAL = {SIAM Journal on Mathematical Analysis},
    VOLUME = {54},
      YEAR = {2022},
    NUMBER = {4},
     PAGES = {4043--4103},
      ISSN = {0036-1410,1095-7154},
   MRCLASS = {76D03 (35Q35 35Q86)},
  MRNUMBER = {4448829},
MRREVIEWER = {Suhua\ Lai},
       DOI = {10.1137/21M1457345},
       URL = {https://doi.org/10.1137/21M1457345},
}

@article{zodji2023discontinuous,
  title={Discontinuous solutions for the Navier-Stokes equations with density-dependent viscosity},
  author={Zodji, S. M.},
  journal={arXiv preprint arXiv:2312.07578},
  year={2023}
}

@article {hoff2002dynamics,
    AUTHOR = {Hoff, D.},
     TITLE = {Dynamics of singularity surfaces for compressible, viscous
              flows in two space dimensions},
   JOURNAL = {Comm. Pure Appl. Math.},
  FJOURNAL = {Communications on Pure and Applied Mathematics},
    VOLUME = {55},
      YEAR = {2002},
    NUMBER = {11},
     PAGES = {1365--1407},
      ISSN = {0010-3640,1097-0312},
   MRCLASS = {76N10 (35Q30)},
  MRNUMBER = {1916987},
MRREVIEWER = {Denis\ Serre},
       DOI = {10.1002/cpa.10046},
       URL = {https://doi.org/10.1002/cpa.10046},
}

@article{matsumura1979initial,
  title={The initial value problem for the equations of motion of compressible viscous and heat-conductive fluids},
  author={Matsumura, A. and Nishida, T.},
  journal={Proc. Japan Acad. Ser. A Math. Sci.},
  volume={55},
  pages={337--342},
  year={1979}
}

\bigbreak
\vspace{1cm}

\noindent Xian Liao \& Rebekka Zimmermann\\

\noindent
\textit{Institute for Analysis, Karlsruhe Institute of Technology\\
 Englerstraße 2, 76131 Karlsruhe, Germany\\}
 \vspace{.2cm}
\noindent  xian.liao@kit.edu, rebekka.zimmermann@kit.edu

\end{document}